%% file: thesis_main_arx.tex
\documentclass[a4paper,12pt]{report}
\usepackage{amsfonts,amssymb,amsmath,amsthm,mathrsfs,url,geometry}
\usepackage{comment}
\usepackage{afterpage}
\usepackage[outline]{contour}
\usepackage{tikz}
\usetikzlibrary{calc,decorations.markings}
\usepackage{graphicx}
\usepackage{float}
\usepackage{caption}
\usepackage[T1]{fontenc}
\usepackage{titlesec, blindtext, color}
\usepackage{titletoc}
\titlecontents{chapter}[1.5em]
  {\addvspace{1em}\bfseries}
  {\contentslabel{2em}}
  {\hspace*{-2.3em}}
  {\hfill\contentspage}
\definecolor{gray6}{gray}{0.6}
\newcommand{\hsp}{\hspace{10pt}}
\titleformat{\chapter}[hang]{\LARGE\bfseries}{\textcolor{gray6}{[}\hsp\thechapter\hsp\textcolor{gray6}{]}\hsp}{0pt}{\LARGE\bfseries\textsc}
\renewcommand{\thechapter}{\Roman{chapter}}

\usepackage{setspace}
\date{}
\allowdisplaybreaks

\usepackage{nomencl}
\makenomenclature
\newlength\nomunitwidh
\renewcommand\nomgroup[1]{ \item[\textbf{Symbol}]\textbf{Meaning/Description}\hfill} 

\newtheorem{thm}{Theorem}
\newtheorem*{thmu}{Theorem}
\newtheorem*{defiu}{Definition}
\newtheorem{lem}{Lemma}
\newtheorem{prop}{Proposition}
\newtheorem{coro}{Corollary}
\newtheorem{defi}{Definition}
\newtheorem{rmk}{Remark}

\newtheorem{note}{Note}
\newtheorem*{notations}{Notations}
\newtheorem*{notation}{Notation}
\newtheorem*{claim}{Claim}
\newtheorem{asump}{Assumptions}

\def\A{\mathcal{A}}
\def\half{\frac{1}{2}}
\def\quater{\frac{1}{4}}
\def\res{\text{Res}}
\def\Co{\mathscr{C}}
\def\Con{\mathscr{C}_N}

\def\xset{\textbf{X}\text{-Set}}
\def\d{\mathrm{d}}
\def\mulint{\underset{l~\mathrm{ times }\quad}{\int_0^{\delta T}\cdots\int_0^{\delta T}}}
\def\A{\mathcal{A}}
\def\M{\mathcal{M}}
\def\set{\mathcal{S}}

\begin{document}
\newpage
\clearpage\pagenumbering{roman}
\include{title}

\begin{doublespacing}

\include{authorstatement}
\include{declaration}

\include{publications}
\include{acknowledgements}
\end{doublespacing}
\singlespacing
\tableofcontents
\cleardoublepage
\newpage
\include{notations}
\newpage
\clearpage
\clearpage\thispagestyle{empty}\addtocounter{page}{0}
\clearpage
\afterpage{\null\newpage}
\include{synopsis}
\clearpage\thispagestyle{empty}\addtocounter{page}{-1}
\clearpage
\afterpage{\null\newpage}
\clearpage\thispagestyle{empty}\addtocounter{page}{-1}
\clearpage
\afterpage{\null\newpage}
\newpage
\setcounter{page}{1}
\clearpage\pagenumbering{arabic}
\include{introduction}
\newpage
\include{mellin_transform}
\newpage
\include{landau_theorem}
\newpage
\include{influence_measure}
\newpage
\include{twisted_divisor_new}
\newpage
\cleardoublepage
\bibliographystyle{abbrv}
\bibliography{refs_omega}
\end{document}

%% file: title.tex
\contourlength{0.4pt}
\contournumber{50}
\thispagestyle{empty}
\begin{center}

\vspace{1mm}

\textbf{ \Large \sc{Measure Theoretic Aspects Of Error Terms}}\\

\vspace{2cm}

By\\

\vspace{.5cm}

\textbf{ \large Kamalakshya Mahatab}

\vspace{.5cm}
\textbf { MATH10201005001}

\vspace{.75cm}
\textbf{\large The Institute of Mathematical Sciences}

\vspace{2.5cm}

\textbf{\textit{
A thesis  submitted to the\\
Board of Studies in Mathematical Sciences\\
}}
\vspace{.5cm}
\textbf{\textit{
In partial fulfillment of requirements\\
for the Degree of\\
}}
\vspace{.5cm}

\textbf{ DOCTOR OF PHILOSOPHY\\
of\\
HOMI BHABHA NATIONAL INSTITUTE}\\

\vspace{1.5cm}

\includegraphics[scale=0.4]{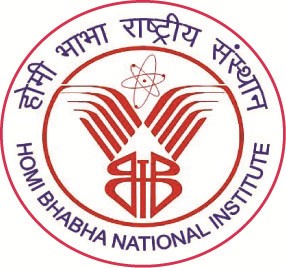}\\
\vspace{1cm}

{\bf April, 2016}\\
\end{center}
\newpage
\thispagestyle{empty}
\mbox{}
\newpage

%% file: authorstatement.tex
\begin{center}
{\Large \bf STATEMENT BY THE AUTHOR}
\end{center}

\vspace{1cm}

\noindent
This dissertation has been submitted in partial fulfillment of requirements for an
advanced degree at Homi Bhabha National Institute (HBNI) and is deposited in the
Library to be made available to borrowers under rules of the HBNI.\\

\noindent
Brief quotations from this dissertation are allowable without special permission,
provided that accurate acknowledgement of source is made. Requests for permission
for extended quotation from or reproduction of this manuscript in whole or in part
may be granted by the Competent Authority of HBNI when in his or her judgement the
proposed use of the material is in the interests of scholarship. In all other instances,
however, permission must be obtained from the author.

\vspace{2cm}

\noindent
\textbf{Date:} \hfill \textbf{Kamalakshya Mahatab}

\newpage
\thispagestyle{empty}
\mbox{}
\newpage

%% file: declaration.tex
\begin{center}
{\bf \Large DECLARATION}
\end{center}

\vspace{1cm}

\noindent
I, Kamalakshya Mahatab, hereby declare that the investigation presented in this thesis has been carried out by
me. The work is original and has not been submitted earlier as a whole or in part for a
degree or diploma at this or any other Institution or University.

\vspace{2cm}

\noindent
\textbf{Date:} \hfill \textbf{Kamalakshya Mahatab}

\newpage
\thispagestyle{empty}
\mbox{}
\newpage

%% file: publications.tex
\begin{center}
{\Large \bf List of Publications}
\end{center}
\vspace{5mm}

{\large \textbf{Journal}}
\noindent
\begin{enumerate}
\item[1.] Kamalakshya Mahatab  and Kannappan Sampath, Chinese Remainder Theorem for Cyclotomic Polynomials in $\mathbb{Z}[X]$. 
Journal of Algebra, 435 (2015), Pages 223-262. doi:10.1016/j.jalgebra.2015.04.006.
\item[2.] Kamalakshya Mahatab, Number of Prime Factors of an Integer. Mathematics News Letter, Ramanujan Mathematical Society, 
volume 24 (2013).
\end{enumerate}

{\large \textbf{Others}}
\begin{enumerate}
\item[1.] Kamalakshya Mahatab and Anirban Mukhopadhyay, Measure Theoretic Aspects of Oscillations of Error Terms.
arXiv:1512.03144v1 (2015). \\ Available at: http://arxiv.org/pdf/1512.03144v1.pdf
\end{enumerate}

\vspace{2cm}

\noindent
\hfill \begin{tabular}{c} 
\\
\bf{Kamalakshya Mahatab} \\
\end{tabular}

\newpage
\thispagestyle{empty}
\mbox{}
\newpage

%% file: acknowledgements.tex
\thispagestyle{empty}

\begin{center}
{\huge \bf Acknowledgments}
\end{center}
\vspace{.4cm}
First and foremost, I would like to express my sincere gratitude to my adviser Prof. Anirban Mukhopadhyay for his continuous 
support and insightful guidance throughout my thesis work. I have greatly benefited from his patience, many a times he has listened to my naive 
ideas carefully and corrected my mistakes. In all my needs, I always found him as a kind and helpful person. I am indebted to him for
all the care and help that I got from him during my Ph.D. time.

\vspace{4mm}
\noindent
I am indebted  to Prof. Amritanshu Prasad for initiating me into research. 
I have learnt many beautiful mathematics while working under him on my M.Sc. thesis. 
He has always encouraged me to think freely and has guided me on several research projects.

\vspace{4mm}
\noindent
Prof. Srinivas Kotyada was always available whenever I had any doubts. 
I sincerely thank him for reading my mathematical writings patiently and giving his valuable suggestions. 
In addition to being my teacher, he has been a loving and caring friend.

\vspace{4mm}
\noindent
I am grateful to Prof. R. Balasubramanian for giving his valuable time to help me 
understand several difficult concepts in number theory,  Prof. Aleksandar Ivi{\'c} and Prof. Olivier Ramar{\'e} for their 
valuable suggestions while writing \cite{MeasureOmega}, Prof. Gautami Bhowmik for initiating me to work on Omega theorems, and 
Prof. Vikram Sharma for his help in writing \cite{Kann}.

\vspace{4mm}
\noindent
I am also grateful to Prof. Partha Sarathi Chakraborty, Prof. D. S. Nagaraj, Prof. S. Kesavan, Prof. K. N. Raghavan,
Prof. S. Viswanath,  Prof. Vijay Kodiyalam, Prof. V. S. Sunder, Prof. Murali Srinivasan, Prof. Xavier Viennot, Prof. P. Sankaran, Prof. Sanoli Gun, 
Prof. Kaneenika Sinha, Prof. Shanta Laishram, Prof. Stephan Baier, Prof. A. Sankaranarayanan, Prof. Ritabrata Munshi,
Prof. R. Thangadurai, Prof. D. Surya Ramana, Prof. Gyan Prakash and many others 
for sharing their mathematical insights through their beautiful lectures and courses.

\vspace{4mm}
\noindent
I thank Dr. C. P. Anil Kumar for being a dear friend. 
During my first three years in IMSc, I have always enjoyed my discussions with him which used to last for several hours at a stretch.

\vspace{4mm}
\noindent
I would also like to thank Kannappan Sampath, my first co-author \cite{Kann} and a great friend, 
for the many insightful mathematical discussions I had with him.

\vspace{4mm}
\noindent
I would like to appreciate the inputs of Prateep Chakraborty, 
Krishanu Dan, Bhavin Kumar Mourya, Prem Prakash Pandey,  
Senthil Kumar, Jaban Meher, Binod Kumar Sahoo, Sachin Sharma,
Kamal Lochan Patra, Neeraj Kumar, Geetha Thangavelu, Sumit Giri, B. Ravinder, 
Uday Bhaskar Sharma, Akshaa Vatwani, Sudhir Pujahari and  Anish Mallick during the mathematical discussions that I had with them.

\vspace{4mm}
\noindent
I am also indebted to my parents and my sister for their unconditional love and emotional support, especially during the times of difficulties.

\vspace{4mm}
\noindent
My stay at IMSc would not have been pleasant without my friends Sandipan De, 
Mitali Routaray, Arghya Mondal, Chandan Maity, Issan Patri, Dibyakrupa Sahoo, Archana Mishra, 
Mamta Balodi, Zodinmawia, Vivek M. Vyas, Ankit Agrawal, Biswajit Ransingh, Ria Ghosh, Devanand T, Sneh Sinha, 
Kasi VIswanadham, Jay Mehta,
Dhriti Ranjan Dolai, Maguni Mahakhud, Keshab Bakshi, 
Priyamvad Srivastav, Jyothsnaa Sivaraman, Narayanan P., G. Arun Kumar, Pranabesh Das, Meesum Syed, Sridhar Narayanan and others. 
They have been a part of several beautiful memories during my Ph.D.

\vspace{4mm}
\noindent
Last, but not the least, I would like to thank my institute 
\lq The Institute of Mathematical Sciences\rq \ for providing me a vibrant research environment. 
I have enjoyed excellent computer facility, a well managed library, clean and well furnished hostel and office rooms, and catering services at my institute.
I would like to express my special gratitude to the library and administrative staffs of my institute for their efficient service.

\vspace{4mm}
\noindent
\hfill \begin{tabular}{c} 
\\
{\bf Kamalakshya Mahatab} \\
\end{tabular}
\newpage
\thispagestyle{empty}
\mbox{}
\newpage

%% file: notations.tex
\chapter*{Notations}
\addcontentsline{toc}{chapter}{Notations}

We denote the set of natural numbers by $\mathbb N$, the set of integers by $\mathbb Z$, the set of real numbers by $\mathbb R$, the set of 
positive real numbers by $\mathbb R^+$,
and the set of complex numbers by $\mathbb C$.

\vspace{0.3cm}
\noindent
The notaion $i$ stands for $\sqrt{-1}$, the square root of $-1$ that belongs to the upper half plane in $\mathbb C$.

\vspace{0.3cm}
\noindent
We denote the Lebesgue mesure on the real line $\mathbb R$ by $\mu$.

\vspace{0.3cm}
\noindent
For $z=\sigma+it\in \mathbb C$, we denote $\sigma$ by $\Re(z)$ and $t$ by $\Im(z)$.

\vspace{0.3cm}
\noindent
Let $f(x)$ and $g(x)$ be a complex valued function on $\mathbb R^+$. 
As $x\rightarrow\infty$, we write
\begin{itemize}
 \item[] $f(x)=O(g(x))$, \ if \ $\lim_{x\rightarrow\infty}\left|\frac{f(x)}{g(x)}\right|>0$; 
 \item[] $f(x)=o(g(x))$, \ if \ $\lim_{x\rightarrow\infty}\left|\frac{f(x)}{g(x)}\right|=0$;
 \item[] $f(x)\ll g(x)$, \ if \ $f(x)=O(g(x))$;
 \item[] $f(x)\gg g(x)$, \ if \ $g(x)=O(f(x))$;
 \item[] $f(x)\sim g(x)$, \ if \ $\lim_{x\rightarrow\infty}\frac{f(x)}{g(x)}=1$;
 \item[] $f(x)\asymp g(x)$, \ if \ $0<\lim_{x\rightarrow\infty}\left| \frac{f(x)}{g(x)}\right| <\infty$.
\end{itemize}

\vspace{0.3cm}
\noindent
Let $f(x)$ be a complex valued function on $\mathbb R^+$, and let $g(x)$ be a positive monotonic function on $\mathbb R^+$. 
As $x\rightarrow\infty$, we write
\begin{itemize}
 \item[] $f(x)=\Omega(g(x))$, \ if \ $\limsup_{x\rightarrow \infty}\frac{|f(x)|}{g(x)}>0$; 
 \item[] $f(x)=\Omega_+(g(x))$, \ if \ $\limsup_{x\rightarrow \infty}\frac{f(x)}{g(x)}>0$;
 \item[] $f(x)=\Omega_-(g(x))$, \ if \ $\liminf_{x\rightarrow \infty}\frac{f(x)}{g(x)}<0$;
 \item[] $f(x)=\Omega_\pm(g(x))$, \ if \ $f(x)=\Omega_+(g(x))$ and $f(x)=\Omega_-(g(x))$.
\end{itemize}

%% file: synopsis.tex
\chapter*{Synopsis}\label{chap:synopsis}
\addcontentsline{toc}{chapter}{Synopsis}

This thesis studies fluctuation of error terms that appears in various asymptotic formulas and size of the sets where these fluctuations occur.
As a consequence, this approach replaces Landau's criterion on oscillation of error terms.

\begin{center}
\textbf{\large{General Theory}}
\end{center}

Consider a sequence of real numbers $\{a_n\}_{n=1}^{\infty}$ having Dirichlet series 
\[D(s)=\sum_{n=1}^{\infty}\frac{a_n}{n^s},\] 
which is convergent in some half-plane. As in Perron summation formula \cite[II.2.1]{TenenAnPr}, we write
\[\sum^*_{n\leq x}a_n=\M(x)+\Delta(x),\]
where $\M(x)$ is the main term, $\Delta(x)$ is the error term and $\sum^*$ is defined as 
\begin{equation*}
\sum^*_{n\leq x} a_n =
\begin{cases}
\sum_{n\leq x} a_n \ & \text{if } x\notin \mathbb N, \\
\sum_{n< x} a_n +\half a_x  \ &\text{if } x\in\mathbb N.
\end{cases}
\end{equation*}
In this thesis, we obtain $\Omega$ and $\Omega_\pm$ estimates for $\Delta(x)$. 
We shall use the Mellin transform of $\Delta(x)$ (defined below) to obtain such estimates.
\begin{defiu}
 The Mellin transform of $\Delta(x)$ be $A(s)$, defined as
 \[A(s)=\int_{1}^{\infty}\frac{\Delta(x)}{x^{s+1}}\d x.\]
\end{defiu}

In this direction, under some natural assumptions and for a suitably defined contour $\Co$, we shall show that
\[A(s)=\int_{\Co}\frac{D(\eta)}{\eta(s-\eta)}\d \eta.\]
In the above formula, the poles of $D(s)$ that lie left to $\Co$ are all the poles that contributes to the main term $\M(x)$.
Landau \cite{Landau} used the meromorphic continuation of $A(s)$ to obtain $\Omega_\pm$ results for $\Delta(x)$. He proved that 
if $A(s)$ has a pole at $\sigma_0+it_0$ for some $t_0\neq 0$ and has no real pole for $s\ge \sigma_0$, then 
\[ \Delta(x)=\Omega_{\pm} (x^{\sigma_0}). \]
We shall show a quantitative version of Landau's theorem, which also generalizes a theorem of Gautami, Ramar\'e and Schlage-Puchta \cite{gautami}. 
Below we state this theorem in a simplified way. We introduce the following notations to state these theorems.
\begin{defiu}
 Let 
\begin{align*}
\mathcal{A}^+_T(x^{\sigma_0})&:=\{T\leq x \leq 2T: \Delta(x)> \lambda x^{\sigma_0}\},\\ 
\mathcal{A}^-_T(x^{\sigma_0})&:=\{T\leq x \leq 2T: \Delta(x)< -\lambda x^{\sigma_0}\},\\
\mathcal{A}_T (x^{\sigma_0})&:= \mathcal{A}^+_T(x^{\sigma_0})\cup \mathcal{A}^-_T(x^{\sigma_0}),
\end{align*}
for some $\lambda, \sigma_0 >0$.
\end{defiu}

\begin{thmu}
Let $\sigma_0>0$, and let the following conditions hold:
\begin{itemize}
 \item[(1)] $A(s)$ has no real pole for $\Re(s)\geq\sigma_0$,
 \item[(2)] there is a complex pole $s_0=\sigma_0+it_0$, $t_0\neq 0$, of $A(s)$, and 
 \item[(3)] for positive functions $h^\pm(x)$ such that $h\pm(x)\rightarrow \infty$ as $x\rightarrow \infty$, we have
 $$\int_{\A^\pm_T(x^{\sigma_0})}\frac{\Delta^2(x)}{x^{2\sigma_0+1}}\d x\ll h^\pm(T).$$ 
\end{itemize}
Then 
\[\mu(\mathcal{A}^\pm_{T}(x^{\sigma_0})) = \Omega\left(\frac{T}{h^\pm(T)}\right),\]
where $\mu$ denotes the Lebesgue measure. 
\end{thmu}
In the above theorem, Condition~2 is a very strong criterion. In the following theorem, we replace Condition~2 by an $\Omega$-bound of 
$\mu(\mathcal{A}_T (x^{\sigma_0}))$ and obtain an $\Omega_{\pm}$-result from the given $\Omega$-bound. 
\begin{thmu}
Let $\sigma_0>0$, and let the following conditions hold: 
\begin{itemize}
 \item[(1)] $A(s)$ has no real pole for $\Re(s)\geq\sigma_0$, and 
 \item[(2)]  $\mu( \mathcal{A}_T(x^{\sigma_0}))=\Omega(T^{1-\delta})$ for $0<\delta<\sigma_0$.
\end{itemize}
Then
\[\Delta(x)=\Omega_\pm(T^{\sigma_0-\delta'})\] for any $\delta'$ such that $0<\delta'<\delta$.
\end{thmu}

The above two theorems are applicable to a wide class of arithmetic functions.
Now we mention some results obtained by applying these theorems.

\begin{center}
\textbf{\large{A Twisted Divisor Function}}
\end{center}
Given $\theta\neq 0$, define 
\[\tau(n, \theta)=\sum_{d|n}d^{i\theta}.\]
The Dirichlet series of $|\tau(n, \theta)|^2$ can be expressed in terms of Riemann zeta function as
\begin{equation*}
 D(s)=\sum_{n=1}^{\infty}\frac{|\tau(n, \theta)|^2}{n^s}=\frac{\zeta^2(s)\zeta(s+i\theta)\zeta(s-i\theta)}{\zeta(2s)}
 \quad  \text{for}\quad \Re(s)>1.
\end{equation*}
In \cite[Theorem 33]{DivisorsHallTenen}, Hall and Tenenbaum proved that
\begin{equation*}
 \sum_{n\leq x}^*|\tau(n, \theta)|^2=\omega_1(\theta)x\log x + \omega_2(\theta)x\cos(\theta\log x)
+\omega_3(\theta)x + \Delta(x),
\end{equation*}
where $\omega_i(\theta)$s are explicit constants depending only on $\theta$. They also showed that 
\begin{equation*}
 \Delta(x)=O_\theta(x^{1/2}\log^6x).
\end{equation*}
Here the main term comes from the residues of  $D(s)$ at $s=1, 1\pm i\theta $.
All other poles of $D(s)$ come from zeros of $\zeta(2s)$. Using a pole on the line $\Re(s)=1/4$, 
Landau's method gives
\[\Delta(x)=\Omega_{\pm}(x^{1/4}).\]
We prove the following bounds for a computable $\lambda(\theta)>0$ and for any $\epsilon>0$:
\begin{align*}
 &\mu\left(\{T \leq  x \leq 2T: \Delta(x)>(\lambda(\theta)-\epsilon)x^{1/4}\}\right)
=\Omega\left(T^{1/2}(\log T)^{-12}\right),\\
&\mu\left(\{T \leq  x \leq 2T: \Delta(x)<(-\lambda(\theta)+\epsilon)x^{1/4}\}\right)
=\Omega\left(T^{1/2}(\log T)^{-12}\right).
\end{align*}
For a constant $c>0$, define
\[\alpha(T) =\frac{3}{8}-\frac{c}{(\log T)^{1/8}}.\]
Applying a method due to Balasubramanian, Ramachandra and Subbarao \cite{BaluRamachandraSubbarao}, we prove
\[\Delta(T)=\Omega\left(T^{\alpha(T)}\right).\]
In fact, this method gives $\Omega$-estimate for the measure of the sets involved: 
\[ \mu(\A\cap [T,2T])=\Omega\left(T^{2\alpha(T)}\right),\] 
where
\[\A=\{x: |\Delta(x)|\ge Mx^{\alpha(x)} \}\] 
and $M>0$ is a positive constant. 
We also show that
\[ \text{either } \ \Delta(x)=\Omega\left(x^{ \alpha(x)+\delta/2}\right)
\ \text{ or } \
 \Delta(x)=\Omega_{\pm}\left(x^{3/8-\delta'}\right),\]
for $0<\delta<\delta'<1/8$.
For any $\epsilon>0$, this result and the conjecture  
\[ \Delta(x)=O(x^{3/8+\epsilon})\]
proves that
\[\Delta(x)=\Omega_{\pm}(x^{3/8-\epsilon}).\]

\begin{center}
\textbf{\large{Prime Number Theorem Error}}
\end{center}
Let $a_n$ be the von Mandoldt function $\Lambda(n)$:
\[\Lambda(n):=\begin{cases}
                  \log p \quad &\mbox{ if } n=p^r, \ r\geq 1, \ p \text{ prime },\\
                  0 & \mbox{ otherwise.}
                  \end{cases}.\]
Let                  
\[\sum_{n\le x}^* \Lambda_n = x + \Delta(x). \]
From the Vinogradov's zero free region for Riemann zeta function, one gets \cite[Theorem~12.2]{IvicBook}
\begin{equation*}
 \Delta(x)=O\left( x\exp\left(-c(\log x)^{3/5}(\log\log x)^{-1/5}\right)\right)
\end{equation*}
for some constant $c>0$.

Hardy and Littlewood \cite{HardyLittlewoodPNTOmegapm} proved that 
\[\Delta(x)=\Omega_\pm\left(x^{1/2}\log\log\log x\right).\]
But this result does not say about the measure of the sets, where the above $\Omega_\pm$ bounds are attained by $\Delta(x)$.
We obtain the following weaker result, but with an $\Omega$-estimates for the measure of the corresponding sets.\\
\noindent
Let $\lambda_1>0$ denotes a computable constant. For a fixed $\epsilon$, $0<\epsilon<\lambda_1$, we write
\begin{align*}
 \A_1&:=\left\{x: \Delta(x)>(\lambda_1-\epsilon)x^{1/2}\right\},\\
 \A_2&:=\left\{ x : \Delta(x)<(-\lambda_1+\epsilon)x^{1/2}\right\}.\vspace{5mm}~\\
\end{align*}
Then
\begin{align*}
 \mu([T, 2T]\cap\A_j)
=\Omega\left(T^{1-\epsilon}\right), \text{ for } j=1, 2 \ \text{ and for any } \ \epsilon>0.
\end{align*}
Under Riemann Hypothesis, we have
\begin{align*}
 \mu([T, 2T]\cap A_j)=\Omega\left(\frac{T}{(\log T)^4}\right) \text{ for } j= 1, 2. 
\end{align*}
We also show the following unconditional 
$\Omega$-bounds for the second moment of $\Delta$:
\begin{equation*}
 \int_{[T, 2T]\cap \A_j}\Delta^2(x)\d x = \Omega(T^2) \quad \text{ for } j=1, 2.
\end{equation*}
\begin{center}
\textbf{\large{Non-isomorphic Abelian Groups}}
\end{center}

Let $a_n$ denote the number of non-isomorphic abelian groups of order $n$. We write 
\begin{equation*}
\sum_{n\leq x}^*a_n = \sum_{k=1}^{6}b_kx^{1/k} +  \Delta(x).
\end{equation*}
In the above formula, we define $b_k$ as
\[b_k:=\prod_{j=1, j\neq k}^\infty \zeta(j/k).\]
It is an open problem to show that 
\begin{equation}\label{conj1s}
\Delta(x)\ll x^{1/6+\delta} \ \text{ for any } \ \delta>0.
\end{equation}
The best result on upper bound of $\Delta(x)$ is 
due to O. Robert and P. Sargos \cite{sargos}, which gives 
\[\Delta(x)\ll x^{1/4+\epsilon} \ \text{ for any } \ \epsilon>0.\]
Also Balasubramanian and Ramachandra \cite{BaluRamachandra2} proved that
\begin{equation*}
 \Delta(x)=\Omega\left(x^{1/6}\sqrt{\log x}\right).
\end{equation*}
Following their method, we prove 
\begin{equation*}
\mu\left( \{ T\le x\le 2T: |\Delta(x)|\ge \lambda_2 x^{1/6}(\log x)^{1/2}\} \right)
=\Omega(T^{5/6-\epsilon}),
\end{equation*}
for some $\lambda_2>0$ and for any $\epsilon>0$.
They also obtained
\[ \Delta(x)=\Omega_{\pm}(x^{92/1221}),\]
while it has been conjectured that
\[\Delta(x)=\Omega_\pm(x^{1/6-\delta}),\]
for any $\delta>0$.
We shall show that either
\[\int_T^{2T}\Delta^4(x)\d x=\Omega( T^{5/3+\delta} )\text{ or }\Delta(x)=\Omega_\pm(x^{1/6-\delta}),\]
for any $0<\delta< 1/42$.
The conjectured upper bound (\ref{conj1s}) of $\Delta(x)$ gives 
 \[\int_T^{2T}\Delta^4(x)\d x \ll T^{5/3+\delta}.\]
This along with our result implies that
\[\Delta(x)=\Omega_\pm(x^{1/6-\delta})\ \text{ for any } \ 0<\delta <1/42.\]


%% file: introduction.tex
\chapter{Introduction}\label{chap:intro}
In 1896, Jacques Hadamard and 
Charles Jean de la Vall\'{e}e-Poussin proved that 
the number of primes upto $x$ is asymptotic to $x/\log x$. This result is well known as the Prime Number Theorem (PNT). 
Below we state a version of this theorem (PNT*) in terms of the von-Mangoldt function.
\begin{defi} 
For $n\in \mathbb{N}$, 
the von-Mangoldt function $\Lambda(n)$ is defined as
 \[\Lambda(n)=\begin{cases}
                   \log p \quad &\mbox{ if } \ n=p^r, \ r\in \mathbb{N}\text{ and } p \text{ prime,}\\
                   0 & \mbox{ otherwise .}
                   \end{cases}
 \]
\end{defi}
\noindent 
\begin{thmu}[PNT*]
For a constant $c_1>0$, we have 
 \[\sum_{n\leq x}^*\Lambda(n)=x + O\left( x\exp\left(-c_1(\log x)^{3/5}(\log\log x)^{-1/5}\right)\right),\]
where 
\[\sum_{n\leq x}^*\Lambda(n)=\begin{cases}
                     \sum_{n\leq x}\Lambda(n) \quad & \mathrm{ if } \ x\notin \mathbb{N},\\
                   \sum_{n\leq x}\Lambda(n)-\Lambda(x)/2 \quad& \mbox{ otherwise .}
                   \end{cases}
\]
\end{thmu}
\noindent
For a proof of the above theorem see \cite[Theorem~12.2]{IvicBook}.
Proof of PNT* uses analytic continuation of the function
\[\zeta(s)=\sum_{n=1}^{\infty}\frac{1}{n^s},\]
defined for $\Re(s)>1$.
The function $\zeta(s)$ is called the \lq Riemann zeta function\rq, named after the famous German mathematician 
Bernhard Riemann. In 1859, Riemann showed that this
has a meromorphic continuation to the whole complex plane. He also showed
PNT by assuming that the meromorphic continuation of $\zeta(s)$ 
does not have zeros for $\Re(s)>\half$. This conjecture of Riemann is popularly known 
as the \lq Riemann Hypothesis\rq \ (RH), and is an unsolved problem. Under RH, the upper bound for $\Delta(x)$ in PNT* 
can be improved as in the following theorem:
\begin{thmu}[PNT**]
Let $\Delta(x)$ be defined as in PNT*. Further, if we assume RH, then
\[\Delta(x)=O\left(x^{\half}\log^2 x\right).\]
\end{thmu}
\begin{proof}
 See \cite{PNT_Under_RH}.
\end{proof}
\noindent
In fact, we shall see in Theorem~\ref{thm:landu_omegapm} that PNT** is equivalent to RH.
At this point, it is natural to ask the following questions:
\texttt{
\begin{itemize}
 \item[-] Can we obtain a bound for $\Delta(x)$, better than the bound in PNT**?
 \item[-] Is $\Delta(x)$ an increasing or a decreasing function?
 \item[-] Can $\Delta(x)$ be both positive and negative depending on $x$?
 \item[-] How large are positive and negative values of $\Delta(x)$?
\end{itemize}}
\noindent
We shall make an attempt to answer these question by obtaining $\Omega$ and $\Omega_\pm$ results. 
The following result was obtained by 
Hardy and Littlewood \cite{HardyLittlewoodPNTOmegapm} in the year 1916:
\begin{align}
 \Delta(x)=\Omega_\pm\left(x^\half\log\log\log x\right).\
\end{align}
The above $\Omega_\pm$ bound on $\Delta(x)$ gives some answer to our earlier questions. 
It says that we can not have an upper bound for $\Delta(x)$ which is smaller than $x^\half\log\log\log x$.
It also says that $\Delta(x)$ often takes both positive and negative values with magnitude of order $x^\half\log\log\log x$.
This suggests, it is important to obtain $\Omega$ and $\Omega_\pm$ bounds for various other error terms. In this direction,  Landau's
theorem \cite{Landau} (see Theorem~\ref{thm:landu_omegapm} below) gives an elegant tool to obtain $\Omega_\pm$ results. Applying this theorem, we have
\begin{align*}
 \Delta(x)=\Omega_\pm\left(x^\half\right).
\end{align*}
The advantage of Landau's method as compared to Hardy and Littlewood's method is in its
applicability to a wide class of error terms of various summatory functions. In Landau's method, the existence of a complex pole 
with real part $\half$ serves as a criterion for the existence of above limits. In this thesis, we shall investigate on a 
quantitative version of Landau's result by obtaining the Lebesgue measure of the sets where $\Delta(x)>\lambda x^{1/2}$ and 
$\Delta(x)<-\lambda x^{\half}$, for some $\lambda>0$. We shall show that the large Lebesgue measure of the set where
$|\Delta(x)|>\lambda x^{\half}$, for some $\lambda>0$ replaces the criterion of existence of a complex pole in Landau's method.
This approach has the advantage of getting $\Omega_\pm$ results even when no such complex pole exists. This is evident from some 
applications which we discuss in this thesis.

\section{Framework}
In this thesis, we consider a sequence of real numbers $\{a_n\}_{n=1}^{\infty}$ having Dirichlet series
\[D(s)=\sum_{n=1}^{\infty}\frac{a_n}{n^s}\] 
that converges in some half-plane. 
The Perron summation formula (see Theorem~\ref{thm:perron_formula}) uses analytic properties of $D(s)$ to give
\[\sum^*_{n\leq x}a_n=\M(x)+\Delta(x),\]
where $\M(x)$ is the main term, $\Delta(x)$ is the error term ( which would be specified later ) and $\sum^*$ is defined as 
\begin{equation*}
\sum^*_{n\leq x} a_n =
\begin{cases}
\sum_{n\leq x} a_n \ & \text{ if } x\notin \mathbb N \\
\sum_{n\leq x} a_n -\half a_x  \ &\text{if } x\in\mathbb N.
\end{cases}
\end{equation*}

In Chapter~\ref{chap:analytic_continuation}, we analyze the Mellin transform $A(s)$ of $\Delta(x)$, which is defined as:
\begin{defi}\label{def:mellin_transform}
For a complex variable $s$, the Mellin transform $A(s)$ of $\Delta(x)$ is defined by
 \[A(s)=\int_{1}^{\infty}\frac{\Delta(x)}{x^{s+1}}\d x.\]
\end{defi}
\noindent
In general, $A(s)$ is holomorphic in some half plane.
We shall discuss a method to obtain a meromorphic continuation of $A(s)$ from the meromorphic continuation of $D(s)$.
In particular, we shall prove in Theorem~\ref{thm:analytic_continuation_mellin_transform} that under some natural assumptions
\[A(s)=\int_{\Co}\frac{D(\eta)}{\eta(s-\eta)}\d \eta,\]
where the contour $\Co$ is as in Definition \ref{def:contour} and $s$ lies 
to the right of $\Co$. Later, this result will complement Theorem~\ref{thm:omega_pm_main} and  Theorem~\ref{thm:omega_pm_main_new} in their applications.

In Chapter~\ref{chap:landau_theorem}, we revisit Landau's method and obtain measure theoretic results.
Also we generalize a theorem of Kaczorowski and Szyd{\l}o \cite{KaczMeasure}, and a theorem of Bhowmik, Ramar{\'e} and Schlage-Puchta \cite{gautami}
in Theorem~\ref{thm:omega_pm_main_new}.

Let
\[\A(\alpha, T):=\{x: x\in[T, 2T], |\Delta(x)|>x^{\alpha}\},\]
and let $\mu$ denotes the Lebesgue measure on $\mathbb R$. 
In Chapter~\ref{chap:measure_analysis}, we establish a connection between  $\mu(\A(\alpha, T))$ and 
fluctuations of $\Delta(x)$. In Proposition~\ref{prop:refine_omega_from_measure}, we see 
that 
\[\mu(\A(\alpha, T))\ll T^{1-\delta}  \ \text{ implies } \ \Delta(x)=\Omega(x^{\alpha + \delta/2}).\]
However, Theorem~{\ref{thm:omega_pm_measure}}
gives that
\[ \mu(\A(\alpha, T))=\Omega(T^{1-\delta}) \ \text{ implies } \ \Delta(x)
=\Omega_\pm(x^{\alpha-\delta}),\]
provided $A(s)$ does not have a real pole for $\Re (s) \geq \alpha-\delta$. In particular, 
this says that either we can improve on the $\Omega$ result or we can obtain a tight $\Omega_\pm$ result for $\Delta(x)$.

In Chapter~\ref{chap:twisted_divisor} we study a twisted divisor function defined as follows:
\begin{equation}\label{eq:tau-n-theta_def}
 \tau(n, \theta)=\sum_{d\mid n}d^{i\theta}\ \text{ for } \ \theta\neq 0.
\end{equation}
This function is used in \cite[Chapter 4]{DivisorsHallTenen} to measure the clustering of divisors.
We give a brief note on some applications of $\tau(n, \theta)$ in Section~\ref{chap:twisted_divisor}.\ref{sec:applications_tau_n_theta}.
In \cite[Theorem 33]{DivisorsHallTenen}, Hall and Tenenbaum proved that
\begin{equation}\label{eq:formmula_tau_ntheta}
 \sum_{n\leq x}^*|\tau(n, \theta)|^2=\omega_1(\theta)x\log x + \omega_2(\theta)x\cos(\theta\log x)
+\omega_3(\theta)x + \Delta(x),
\end{equation}
where $\omega_i(\theta)$s are explicit constants depending only on $\theta$. They also showed that 
\begin{equation}\label{eq:upper_bound_delta}
 \Delta(x)=O_\theta(x^{1/2}\log^6x).
\end{equation}
We give a proof of this formula in Theorem~\ref{thm:asymp_formula_tau_n_theta}. Also, we derive 
$\Omega$ and $\Omega_\pm$ bounds for $\Delta(x)$ using techniques from previous chapters.
In Theorem~\ref{omega_integral}, we obtain an $\Omega$ bound for the second moment of $\Delta(x)$ by adopting
a technique due to Balasubramanian, Ramachandra and Subbarao \cite{BaluRamachandraSubbarao}.

The main theorems of this thesis, except Theorem~\ref{thm:omega_pm_main_new}, are from \cite{MeasureOmega},
which is a joint work of the author with A. Mukhopadhyay.

\section{Applications}
Now we conclude the introduction by mentioning few applications of the methods given in this thesis. 

\subsection{Twisted Divisors}
Consider the twisted divisor function $\tau(n, \theta)$ defined in the previous section.
The Dirichlet series of $|\tau(n, \theta)|^2$ can be expressed in terms of the Riemann zeta function as:
\begin{equation}\label{eq:dirichlet_series_tauntheta}
 D(s)=\sum_{n=1}^{\infty}\frac{|\tau(n, \theta)|^2}{n^s}=\frac{\zeta^2(s)\zeta(s+i\theta)\zeta(s-i\theta)}{\zeta(2s)}
 \quad\quad  \text{for}\quad \Re(s)>1.
\end{equation}
In Theorem~\ref{thm:asymp_formula_tau_n_theta}, we shall show
\begin{equation*}
\sum_{n\leq x}^*|\tau(n, \theta)|^2=\omega_1(\theta)x\log x + \omega_2(\theta)x\cos(\theta\log x)
 +\omega_3(\theta)x + \Delta(x),
\end{equation*}
where $\omega_i(\theta)$s are explicit constants depending only on $\theta$ 
and
\begin{equation*}
\Delta(x)=O_\theta(x^{1/2}\log^6x).
\end{equation*}
The Dirichlet series $D(s)$ has poles at $s=1, 1\pm i\theta $ and at the zeros of $\zeta(2s)$. 
Using a complex pole on the line $\Re(s)=1/4$, 
Landau's method gives
\[\Delta(x)=\Omega_{\pm}(x^{1/4}).\]
In order to apply the method of Bhowmik, Ramar{\'e} and Schlage-Puchta, we need
\[ \int_T^{2T} \Delta^2(x) \d x \ll T^{2\sigma_0+1+\epsilon}, \]
for any $\epsilon >0$ and  $\sigma_0=1/4$; such an estimate is not possible due to
Corollary~\ref{coro:balu_ramachandra1}.
Generalization of this method in Theorem~\ref{thm:omega_pm_main} can be applied to get
\begin{align*}
 \mu \left( \A_j\cap [T, 2T]\right)=\Omega\left(T^{1/2}(\log T)^{-12}\right) \quad \text{ for } j=1, 2,
\end{align*}
and here $\A_j$s' for $\Delta(x)$ are defined as
\begin{align*}
\A_1=\left\{x: \Delta(x)>(\lambda(\theta)-\epsilon)x^{1/4}\right\}
 \ \text{and} \ \A_2=\left\{ x : \Delta(x)<(-\lambda(\theta)+\epsilon)x^{1/4}\right\},
\end{align*}
for any $\epsilon>0$ and $\lambda(\theta)>0$ as in (\ref{eqn:lambda_theta}). But under Riemann Hypothesis, we show in  
(\ref{result:tau_n_theta_omegapm_underRH}) that the above $\Omega$ bounds can be improved to
\begin{align*}
\mu\left(\A_j\right) =\Omega\left(T^{3/4-\epsilon}\right),\ \text{ for } j=1, 2 \ \text{ and for any } \ \epsilon>0.
\end{align*}

Fix a constant $c_2>0$ and define
\[\alpha(T) =\frac{3}{8}-\frac{c_2}{(\log T)^{1/8}}.\]
In Corollary~\ref{coro:balu_ramachandra2}, we prove that 
\[\Delta(T)=\Omega\left(T^{\alpha(T)}\right).\]
In Proposition~\ref{Balu-Ramachandra-measure}, we give an $\Omega$ estimate for the measure of the sets involved in the above bound:
\[ \mu(\A\cap [T,2T])=\Omega\left(T^{2\alpha(T)}\right),\]
where
\[\A=\{x: |\Delta(x)|\ge Mx^{\alpha}\}\] 
for a positive constant $M>0$. 
In Theorem~\ref{thm:tau_theta_omega_pm}, we show
that
\[ \text{either } \ \Delta(x)=\Omega\left(x^{ \alpha(x)+\delta/2}\right)
\ \text{ or } \
 \Delta(x)=\Omega_{\pm}\left(x^{3/8-\delta'}\right), \]
for $0<\delta<\delta'<1/8$.
We may conjecture that 
\[ \Delta(x)=O(x^{3/8+\epsilon}) \ \text{ for any } \epsilon>0.\]
Theorem~\ref{thm:tau_theta_omega_pm} and this conjecture imply that 
\[\Delta(x)=\Omega_{\pm}(x^{3/8-\epsilon})\ \text{ for any } \epsilon>0.\]

\subsection{Square Free Divisors}
Let $\Delta(x)$ be the error term in the asymptotic formula for partial sums of the square free divisors:
\begin{align*}
\Delta(x)=\sum_{n\leq x}^* 2^{\omega(n)}-\frac{x\log x}{\zeta(2)}+\left(-\frac{2\zeta'(2)}{\zeta^2(2)} + \frac{2\gamma - 1}{\zeta(2)}\right)x,
\end{align*}
where $\omega(n)$ denotes the number of distinct primes divisors of $n$. It is known that $\Delta(x)\ll x^{1/2}$ (see \cite{holder}).
Let $\lambda_1>0$ and the sets $\A_j$ for $j=1, 2$ be defined as in Section~\ref{subsec:sqfree_divisors}:
\begin{align*}
 \A_1=\left\{x: \Delta(x)>(\lambda_1-\epsilon)x^{1/4}\right\},\ \text{ and } \ 
 \A_2=\left\{ x : \Delta(x)<(-\lambda_1+\epsilon)x^{1/4}\right\}.
\end{align*}
In (\ref{eq:sqfree_divisors_omegaset_uc}), we show that
\begin{equation*}
 \mu\left(\A_j\cap[T, 2T]\right)=\Omega\left(T^{1/2}\right) \text{ for } j=1, 2.
\end{equation*}
But under Riemann Hypothesis, we prove the following $\Omega$ bounds in (\ref{eq:sqfree_divisors_omegaset}):
\begin{equation*}
 \mu\left(\A_j\cap[T, 2T]\right)=\Omega\left(T^{1-\epsilon}\right), \text{ for } j=1, 2 \ \text{ and for any } \ \epsilon>0. 
\end{equation*}

\subsection{Divisors}
Let $d(n)$ denotes the number of divisors of $n$:
\[d(n)=\sum_{d|n}1.\]
Dirichlet \cite[Theorem~320]{HardyWright} showed that
\[\sum_{n\leq x}^*d(n) = x\log x + (2\gamma -1)x + \Delta(x), \]
where $\gamma$ is the Euler constant and 
\[\Delta(x)=O(\sqrt{x}).\]
Latest result on $\Delta(x)$ is due to Huxley \cite{HuxleyDivisorProblem}, which is
\[\Delta(x)=O(x^{131/416}).\]
On the other hand, Hardy \cite{HardyDirichletDivisor} showed that 
\begin{align*}
 \Delta(x)&=\Omega_+((x\log x)^{1/4}\log\log x),\\
 &=\Omega_-(x^{1/4}).
\end{align*}
There are many improvements on Hardy's result due to K. Corr{\'a}di and I. K{\'a}tai \cite{CorradiKatai},
J. L. Hafner \cite{Hafner} and K. Sounderarajan \cite{Sound}. As a consequence of Theorem~\ref{thm:omega_pm_measure},
we shall show in Chapter~\ref{chap:measure_analysis} that 
for all sufficiently large $T$ and for a constant $c_3>0$, there exist $x_1, x_2 \in [T, 2T]$ such that
\begin{align*}
 \Delta(x_1)> c_3x_1 \ \text{ and } \ \Delta(x_2)< - c_3x_2.
\end{align*}
In particular, we get 
\begin{align*}
\Delta(x)=\Omega_\pm(x^{1/4}).
\end{align*}

\subsection{Error Term in the Prime Number Theorem}
Let $\Delta(x)$ be the error term in the Prime Number Theorem:
\[\Delta(x)=\sum_{n\leq x}^*\Lambda(n)-x.\]
We know from Landau's theorem \cite{Landau} that 
\[\Delta(x)=\Omega_\pm\left(x^{1/2}\right)\]
and from the theorem of Hardy and Littlewood \cite{HardyLittlewoodPNTOmegapm} that
\[\Delta(x)=\Omega_\pm\left(x^{1/2}\log\log x\right).\]
We define
\begin{align*}
 \A_1=\left\{x: \Delta(x)>(\lambda_2-\epsilon)x^{1/2}\right\} \
 \text{ and } \ \A_2=\left\{ x : \Delta(x)<(-\lambda_2+\epsilon)x^{1/2}\right\},\\
\end{align*}
where $\lambda_2>0$ be as in Section~\ref{subsec:pnt_error}.
If we assume Riemann Hypothesis, then the theorem of Kaczorowski and Szyd{\l}o ( see Theorem~\ref{thm:kaczorowski} below )
along with PNT** gives 
\[ \mu\left(\A_j\cap[T, 2T]\right)=\Omega\left(\frac{T}{\log^4 T}\right) \text{ for } j=1, 2.\]
However, as an application of Corollary~\ref{cor:measure_omega_pm_from_upper_bound}, we prove the following weaker bound unconditionally:
\[ \mu\left(\A_j\cap[T, 2T]\right)=\Omega\left(T^{1-\epsilon}\right), \text{ for } j=1, 2\ \text{ and for any }\epsilon>0.\]

\subsection{Non-isomorphic Abelian Groups}
Let $a_n$ be the number of non-isomorphic abelian groups of order $n$, and 
the corresponding Dirichlet series is given by
\[ \sum_{n=1}^{\infty}\frac{a(n)}{n^s} = \prod_{k=1}^{\infty}\zeta(ks)\  \text{  for } \Re(s)>1. \] 
Let $\Delta(x)$ be defined as 
\begin{equation*}
 \Delta(x)=\sum_{n\leq x}^*a_n - \sum_{k=1}^{6}\Big(\prod_{j \neq k}\zeta(j/k)\Big) x^{1/k}.
\end{equation*}
It is an open problem to show that 
\begin{equation}\label{conj1}
\Delta(x)\ll x^{1/6+\epsilon} \ \text{ for any }\epsilon>0.
\end{equation}
The best result on upper bound of $\Delta(x)$ is 
due to O. Robert and P. Sargos \cite{sargos}, which gives 
\[\Delta(x)\ll x^{1/4+\epsilon}\ \text{ for any }\epsilon>0.\]
Balasubramanian and Ramachandra \cite{BaluRamachandra2} proved that
\begin{equation*}
 \int_T^{2T}\Delta^2(x)\d x=\Omega(T^{4/3}\log T).
\end{equation*}
Following the proof of Proposition \ref{Balu-Ramachandra-measure}, we get
\begin{equation*}
\mu\left( \{ T\le x\le 2T: |\Delta(x)|\ge \lambda_3 x^{1/6}(\log x)^{1/2}\} \right)
=\Omega(T^{5/6-\epsilon}),
\end{equation*}
for some $\lambda_2>0$ and for any $\epsilon>0$.
Sankaranarayanan and Srinivas \cite{srini} proved that
\[ \Delta(x)=\Omega_{\pm}\left(x^{1/10}\exp\left(c\sqrt{\log x}\right)\right)\]
for some constant $c>0$.
It has been conjectured that
\[\Delta(x)=\Omega_\pm(x^{1/6-\delta})\ \text{ for any } \delta>0.\]
In Proposition~\ref{prop:abelian_group}, we prove that either
\[\int_T^{2T}\Delta^4(x)\d x=\Omega( T^{5/3+\delta} )\text{ or }\Delta(x)=\Omega_\pm(x^{1/6-\delta}),\]
for any $0<\delta< 1/42$.
The conjectured upper bound (\ref{conj1}) of $\Delta(x)$ gives 
 \[\int_T^{2T}\Delta^4(x)\d x \ll T^{5/3+\delta}.\]
This along with Proposition~\ref{prop:abelian_group} implies that 
\[\Delta(x)=\Omega_\pm(x^{1/6-\delta})\ \text{ for any }\ 0<\delta <1/42.\]

%% file: mellin_transform.tex
\chapter{Analytic Continuation Of The Mellin Transform}\label{chap:analytic_continuation}
In this chapter, we express the error term $\Delta(x)$ as a contour integral using the Perron's formula. This allows us to obtain a 
meromorphic continuation of $A(s)$ (see Definition~\ref{def:mellin_transform}) in terms of the meromorphic continuation of $D(s)$, which is the main 
theorem of this chapter ( Theorem~\ref{thm:analytic_continuation_mellin_transform} ). This theorem will be used in the next chapter to obtain 
$\Omega_\pm$ results for $\Delta(x)$.
\section{Perron's Formula}
Recall that we have a sequence of real numbers $\{a_n\}_{n=1}^{\infty}$, with its 
Dirichlet series $D(s)$. The Perron summation formula approximates the partial sums of $a_n$ by expressing it as a contour integral
involving $D(s)$.
\begin{thm}[Perron's Formula, Theorem~II.2.1 \cite{TenenAnPr}]\label{thm:perron_formula} 
Let $D(s)$ be absolutely convergent for $\Re(s)>\sigma_c$, and let $\kappa>\max(0, \sigma_c)$. Then for $x\geq 1$, we have
\[\sum_{n\leq x}^*a_n=\int_{\kappa-i\infty}^{\kappa+i\infty}\frac{D(s)x^s}{s}\d s.\] 
\end{thm}
\noindent
But in practice, we use the following effective version of the Perron's formula.
\begin{thm}[Effective Perron's Formula, Theorem~II.2.1 \cite{TenenAnPr}]\label{thm:effective_perron_formula}
 Let $\{a_n\}_{n=1}^\infty, \ D(s)$ and $\kappa$ be defined as in Theorem~\ref{thm:perron_formula}. Then for $T\geq 1$ and $x\geq 1$, we have
 \[\sum_{n\leq x}^*a_n=\int_{\kappa-iT}^{\kappa+iT}\frac{D(s)x^s}{s}\d s + O\left(x^\kappa\sum_{n=1}^\infty 
 \frac{|a_n|}{n^\kappa(1+T|\log(x/n)|)}\right).\]
\end{thm}
\noindent
The above formulas are used by shifting the line of integration, and thus by collecting the residues of $D(s)x^s/s$ at its poles lying to 
the right of the shifted contour. 
The residues contribute to the main term $\M(x)$, leaving an expression for $\Delta(x)$ as a contour integral.
So we write
\[\sum_{n\leq x}^*a_n=\M(x)+\Delta(x),\]
where $\M(x)$ is the main term and $\Delta(x)$ is the error term. We make the following natural assumptions on $D(s), \M(x)$ and $\Delta(x)$.
\begin{asump}\label{as:for_continuation_mellintran}
Suppose there exist real numbers $\sigma_1$ and $\sigma_2$ satisfying $0<\sigma_1<\sigma_2$,  
such that
\begin{enumerate}
 \item[(i)]
$D(s)$ is absolutely convergent for $\Re(s)> \sigma_2$.
 \item[(ii)]
$D(s)$ can be meromorphically continued to the half plane $\Re(s)>\sigma_1$ with only finitely many poles 
$\rho$ of $D(s)$  satisfying 
 \[ \sigma_1<\Re(\rho)\leq\sigma_2. \] 
 We shall denote this set of poles by $\mathcal{P}$.
\item[(iii)]
The main term $\M(x)$ is sum of residues of $\frac{D(s)x^s}{s}$ at poles in $\mathcal P$:
 \[\M(x)=\sum_{\rho\in \mathcal{P}}Res_{s=\rho}\left( \frac{D(s)x^s}{s}\right).\] 
\end{enumerate}
\end{asump}
The above assumptions also imply:
\begin{note}\label{note:initial_assumption_consequences}
We may also observe:
\begin{enumerate}
\item[(i)]  For any $\epsilon>0$, we have
  \[|a_n|, |\M(x)|, |\Delta(x)|, \left|\sum_{n\leq x}a_n\right| \ll x^{\sigma_2+\epsilon}. \]
\item[(ii)] The main term $\M(x)$ is a polynomial in $x$, and $\log x$:
  \[\M(x)=\sum_{j\in\mathscr{J}}\nu_{1, j}x^{\nu_{2, j}}(\log x)^{\nu_{3, j}},\]
  where $\nu_{1, j}$ are complex numbers, $\nu_{2, j}$ are real numbers with $\sigma_1<\nu_{2, j}\leq \sigma_2$, $\nu_{3, j}$ are positive integers, and $\mathscr J$ is a finite index set.
\end{enumerate}
\end{note}
\noindent
To express $\Delta(x)$ in terms of a contour integration, we define the following contour.
\begin{defi}\label{def:contour}
Let $\sigma_1, \sigma_2$ be as defined in Assumptions~\ref{as:for_continuation_mellintran}.
Choose a positive real number  $\sigma_3$ such that
$\sigma_3>\sigma_2.$
We define the contour $\mathscr{C}$, as in Figure~\ref{fg:contour_c0}, as the union of the following five
line segments: 
\[\mathscr{C}=L_1\cup L_2\cup L_3 \cup L_4 \cup L_5 ,\]
 where
\begin{align*}
L_1=&\{\sigma_3+iv: T_0 \leq v < \infty\}, 
&L_2=\{u+iT_0: \sigma_1\leq u\leq \sigma_3 \}, \\
L_3=&\{\sigma_1+iv: -T_0 \leq v \leq T_0\}, 
&L_4=\{u-iT_0: \sigma_1\leq u\leq \sigma_3 \}, \\
L_5=&\{\sigma_3+iv: -\infty< v \leq -T_0\}. \\
\end{align*}
\end{defi}

 \begin{figure}
 \begin{tikzpicture}[yscale=0.8]
\draw [<->][dotted] (0, -4.4)--(0, 4.4);
\node at (-0.5, 2) {$T_0$}; 
\draw [thick] (-0.1, 2 )--(0.1,2);
\node at (-0.3, 0.3) {$0$};
\node at (-0.5, -2) {$-T_0$};
\draw [thick] (-0.1,-2 )--(0.1, -2);
\draw [<->][dotted] (5, 0)--(-2, 0);

\draw [dotted] (2, -4.4)--(2, 4.4);
\node at (1.7, 0.3) {$\sigma_1$};
\draw [dotted] (3, -4.4)--(3, 4.4);
\node at (2.7, 0.3) {$\sigma_2$};
\draw [dotted] (4, -4.4)--(4, 4.4);
\node at (3.7, 0.3) {$\sigma_3$};

\draw [thick] [postaction={decorate, decoration={ markings,
mark= between positions 0.1 and 0.99 step 0.2 with {\arrow[line width=1.2pt]{>},}}}]
(4, -4.4)--(4, -2)--(2, -2)--(2, 2)--(4, 2)--(4, 4.4);
\node [below left] at (1.8, 1.6) {$\mathscr{C}$};
\end{tikzpicture}
\caption{Contour $\mathscr{C}$}\label{fg:contour_c0}
\end{figure}
\noindent
Now, we write $\Delta(x)$ as an integration over $\Co$ in the following lemma.
\begin{lem}\label{lem:integral_exp_delta}
Under Assumptions~\ref{as:for_continuation_mellintran}, the error term $\Delta(x)$ can be expressed as:
\begin{equation*}
\Delta(x) = \int_{\mathscr{C}}\frac{D(\eta)x^{\eta}}{\eta}{\d \eta}. 
\end{equation*}
\end{lem}
\begin{proof}
Follows from Theorem~\ref{thm:perron_formula}.
\end{proof}

\section{Analytic continuation of $A(s)$}
Now, we shall discuss a method to 
obtain a meromorphic continuation of $A(s)$, which will serve as an important tool to obtain $\Omega_\pm$ results for $\Delta(x)$ in the 
following chapter. 

\noindent
Below we present the main theorem of this chapter.
\begin{thm}\label{thm:analytic_continuation_mellin_transform}
Under Assumptions-\ref{as:for_continuation_mellintran}, we have
\[A(s)=\int_{\Co}\frac{D(\eta)}{\eta(s-\eta)}\d\eta,\]
when $s$ lies right to $\Co$.
\end{thm}
\subsection{Preparatory Lemmas}
We shall need the following preparatory lemmas to prove the above theorem. 

From Lemma~\ref{lem:integral_exp_delta}, we have:
\begin{equation}\label{eq:A_in_eta_x}
A(s)= \int_1^{\infty}\int_{\mathscr{C}}\frac{D(\eta)x^{\eta}}{\eta}{\d\eta}\frac{\d x}{x^{s+1}}.
\end{equation}
To justify Theorem~\ref{thm:analytic_continuation_mellin_transform}, we need to justify the interchange of 
the integrals of $\eta$ and $x$ in (\ref{eq:A_in_eta_x}).

\begin{defi}
Define the following complex valued function $B(s)$:
\begin{align*}
B(s)&:=\int_{\mathscr{C}}\frac{D(\eta)}{\eta}\int_1^{\infty}\frac{\d x}{x^{s-\eta+1}}{\d\eta}\\ 
&=\int_{\mathscr{C}}\frac{D(\eta){\d\eta}}{(s-\eta)\eta}\quad \quad \text{ for } \Re(s)>\Re(\eta).
\end{align*}
\end{defi}
The integral defining $B(s)$ being absolutely convergent, we have $B(s)$ is 
well defined and analytic.

\begin{defi}
 For a positive integer $N$, define the contour $\mathscr{C}(N)$ as:
 \[\mathscr{C}(N)=\{\eta \in \mathscr{C}: |\Im(\eta)|\leq N\}.\]
\end{defi}
\begin{defi}
Integrating the integrals of $\eta$ and $x$, define $B_N(s)$ as:
\begin{align*}
 B_N(s)&=\int_{\mathscr{C}(N)}\frac{D(\eta){\d \eta}}{\eta}\int_1^{\infty}\frac{\d x}{x^{s-\eta+1}}\\ 
&=\int_{\mathscr{C}(N)}\frac{D(\eta){\d\eta}}{(s-\eta)\eta}\quad \quad \text{ for } \Re(s)>\Re(\eta).
\end{align*}
\end{defi}
With above definitions we prove:
\begin{lem}\label{lem:FubiniForExchange}
The functions $B$ and $B_N$ satisfy the following identities:
\begin{align}
\label{eq:limitAn_prime}
B(s)&= \lim_{N\rightarrow \infty}B_N(s)\\
\label{eq:fubini_onAn_prime}
&=\lim_{N\rightarrow \infty}\int_1^{\infty}\int_{\mathscr{C}(N)}\frac{D(\eta)x^{\eta}}{\eta}{\d\eta}\frac{\d x}{x^{s+1}}.
\end{align}
\end{lem}
\begin{proof}
Assume $N>T_0$. To show (\ref{eq:limitAn_prime}), note:
 \begin{align*}
 |B(s)-B_N(s)|&\leq\left|\int_{\mathscr{C}-\mathscr{C}(n)} \frac{D(\eta){\d\eta}}{(s-\eta)\eta}\right| \\
 &\ll \left| \int_{\sigma_3+iN}^{\sigma_3+i\infty}\frac{D(\eta){\d\eta}}{(s-\eta)\eta} + \int_{\sigma_3-i\infty}^{\sigma_3-iN}\frac{D(\eta){\d\eta}}{(s-\eta)\eta}\right|\\
 & \ll \int_N^{\infty}\frac{\d v}{v^2} \ll \frac{1}{N}.\quad (\text{ substituting }\eta=\sigma_3+iv)
 \end{align*}
This completes proof of (\ref{eq:limitAn_prime}). 

We shall prove (\ref{eq:fubini_onAn_prime}) using a theorem of Fubini and Tonelli \cite[Theorem~B.3.1,~(b)]{DeitmarHarmonic}. To show that the integrals commute, we need to show that
one of the iterated integrals in (\ref{eq:fubini_onAn_prime}) converges absolutely. We note:
\begin{align*}
&\int_{\mathscr{C}(N)}\int_1^{\infty}\left|\frac{D(\eta)}{\eta x^{s-\eta+1}}\right|  \d x|\d\eta|\\
&\ll \int_{\mathscr{C}(N)}\left|\frac{D(\eta)}{\eta (\Re(s)-\Re(\eta))}\right| |\d\eta|< \infty.
\end{align*}
This implies (\ref{eq:fubini_onAn_prime}).
\end{proof}
Let
\begin{equation}\label{eq:B_N_s_defi}
 B'_N(s):= \int_1^{\infty}\int_{\mathscr{C}(N)}\frac{D(\eta)x^{\eta}}{\eta}{\d\eta}\frac{\d x}{x^{s+1}}.
\end{equation}
We re-write (\ref{eq:fubini_onAn_prime}) of Lemma~\ref{lem:FubiniForExchange} as:
\[\lim_{N\rightarrow\infty}B_N'(s)=B(s).\]
Observe that $A(s)=B(s)$, if
\[\lim_{N\rightarrow \infty} \int_1^{\infty}\int_{\mathscr{C}-\mathscr{C}(N)}
\frac{D(\eta)x^{\eta}}{\eta}{\d\eta}\frac{\d x}{x^{s+1}}=0;\]
can be shown by interchanging the integral of $x$ with the limit. For this, 
we need the uniform convergence of the integrand, which we do not have. It is easy to see from Theorem~\ref{thm:effective_perron_formula}
that the problem arises when $x$ is an integer. To handle this problem, we shall divide the integral in two parts,
with one part having neighborhoods of integers.
\begin{defi}
 For $\delta=\frac{1}{\sqrt N}$ ( where $N\geq 2$ ),  we construct the following set as a neighborhood of integers:
\[\mathcal{S(\delta)}:=[1, 1+\delta]\cup(\cup_{m\geq 2}[m-\delta, m+\delta]).\]
\end{defi}
\noindent
Write
\begin{equation}\label{eq:A_min_BN}
A(s)-B'_N(s)=J_{1, N}(s) + J_{2, N}(s) - J_{3, N}(s),\\
\end{equation}
where
\begin{align*}
&J_{1, N}(s)= \int_{\mathcal{S}(\delta)^c}\int_{\Co-\Con}\frac{D(\eta)x^\eta}{\eta}\d\eta \frac{\d x}{x^{s+1}},&&\\
&J_{2, N}(s)= \int_{\mathcal{S}(\delta)}\int_{\sigma_3-i\infty}^{\sigma_3+i\infty}\frac{D(\eta)x^\eta}{\eta}\d\eta \frac{\d x}{x^{s+1}},&&\\
&J_{3, N}(s)= \int_{\mathcal{S}(\delta)}\int_{\sigma_3-iN}^{\sigma_3+iN}\frac{D(\eta)x^\eta}{\eta}\d\eta \frac{\d x}{x^{s+1}}.&&\\
\end{align*}
In the next three lemmas, we shall show that each of $J_{i, N}(s) \rightarrow 0$ as 
$N \rightarrow \infty$. 
\begin{lem}\label{lem:J1lim}
For $\Re(s)=\sigma>\sigma_3+1$, we have the limit
\begin{equation*}
\lim_{N\rightarrow \infty} J_{1, N}(s)= 0. 
\end{equation*}
\end{lem}
\begin{proof}
 Using Theorem~\ref{thm:effective_perron_formula} with $x\in \mathcal{S}(\delta)^c$, we have
\begin{align*}
 \left| \int_{\Co-\Con}\frac{D(\eta)x^\eta}{\eta}\d\eta \right|& \ll x^{\sigma_3}\sum_{n=1}^{\infty}
\frac{|a_n|}{n^{\sigma_3}(1+N|\log(x/n)|)} \\
&\ll \frac{x^{\sigma_3}}{N}\sum_{n=1}^{\infty}\frac{|a_n|}{n^{\sigma_3}}
+ \frac{1}{N}\sum_{x/2\leq n \leq 2x}\frac{x|a_n|}{|x-n|}\left(\frac{x}{n}\right)^{\sigma_3}\\
&\ll \frac{x^{\sigma_3}}{N} + \frac{x^{\sigma_3+1+\epsilon}}{\delta N} 
\ll \frac{x^{\sigma_3+1+\epsilon}}{\sqrt N} \quad (\text{ as } \delta= N^{-\half}).
\end{align*}
From the above calculation, we see that
\[|J_{1, N}|\ll \frac{1}{\sqrt N}\int_{1}^{\infty}x^{\sigma_3-\sigma+\epsilon}dx\ll \frac{1}{\sqrt N}\]
for $\sigma=\Re(s)>\sigma_3+1+\epsilon$. This proves our required result.
\end{proof}
\begin{lem}\label{lem:J2lim}
For $\Re(s)=\sigma>\sigma_3$,
\[\lim_{N\rightarrow \infty} J_{2, N}(s) = 0. \]
\end{lem}
\begin{proof}
Recall that
\begin{equation*}
\sum_{n\leq x}^*a_n 
= \begin{cases}
 \sum_{n<x} a_n + a_x/2 &\mbox{ if } x \in \mathbb{N}, \\
 \sum_{n\leq x}a_n &\mbox{ if } x \notin \mathbb{N}.
\end{cases}
\end{equation*}
By Note~\ref{note:initial_assumption_consequences}, 
$$\sum_{n\leq x}^*a_n  \ll x^{\sigma_3}.$$
Using this bound, we calculate an upper bound for $J_{2, N}$ as follows:
\begin{align*}
&\left|\int_{\mathcal{S}(\delta)}\int_{\sigma_3-i\infty}^{\sigma_3+i\infty}\frac{D(\eta)x^\eta}{\eta}\d\eta\frac{\d x}{x^{s+1}}\right| 
\leq\int_{\mathcal{S}(\delta)}\frac{\left|\sum_{n\leq x}^* a_n\right|}{x^{\sigma+1}}\d x\\
&\ll \int_{\mathcal{S}(\delta)} x^{\sigma_3-\sigma-1}\d x
\ll \int_1^{1+\delta}x^{\sigma_3-\sigma-1} + \sum_{m=2}^{\infty} \int_{m-\delta}^{m+\delta}x^{\sigma_3-\sigma-1}\d x .
\end{align*}
This gives
\[|J_{2, N}(s)|\ll \delta +  \sum_{m\geq 2}\left(\frac{1}{(m-\delta)^{\sigma-\sigma_3}}-\frac{1}{(m+\delta)^{\sigma-\sigma_3}}\right).\]
Using the mean value theorem, for all $m\geq 2$ there exists a real number $\overline{m}\in[m-\delta, m+\delta]$ such that
\begin{align*}
|J_{2, N}(s)| \ll \delta + \sum_{m\geq 2}\frac{\delta}{\overline{m}^{\sigma-\sigma_3+1}}
\ll \delta=\frac{1}{\sqrt N} \quad \text{ by choosing } \sigma> \sigma_3.
\end{align*}
This implies that $J_{2, N}\rightarrow 0$ as $N\rightarrow\infty$. 
\end{proof}

\begin{lem}\label{lem:J3lim}
For $\sigma>\sigma_3$, we have
\[\lim_{N\rightarrow \infty} J_{3, N}(s) = 0.\]
\end{lem}
\begin{proof}
Consider
\begin{align*}
J_{3, N}(s) &= \int_{\mathcal{S}(\delta)} \int_{\sigma_3-iN}^{\sigma_3+iN}\frac{D(\eta)x^\eta}{\eta}\d\eta \frac{\d x}{x^{s+1}}.
\end{align*}
This double integral is absolutely convergent for $\Re(s)>\sigma_3$. Using the Theorem of Fubini and Tonelli \cite[Theorem~B.3.1,~(b)]{DeitmarHarmonic}, 
we can interchange the integrals: 
\begin{align*}
J_{3, N}(s)&=\int_{\sigma_3-iN}^{\sigma_3+iN}\frac{D(\eta)}{\eta}\int_{\mathcal{S}(\delta)}x^{\eta-s-1}{\d x}~{\d\eta}\\
&=\int_{\sigma_3-iN}^{\sigma_3+iN}
\frac{D(\eta)}{\eta}\left\{\int_1^{1+\delta}\frac{x^{\eta}}{x^{s+1}}\d x
+\sum_{m\geq 2}\int_{m-\delta}^{m+\delta}\frac{x^{\eta}}{x^{s+1}}\d x\right\}\d\eta.
\end{align*}
For any $\theta_1, \theta_2 $ such that $0<\theta_1<\theta_2<\infty$, we have
\begin{align*}
\int_{\theta_1}^{\theta_2}x^{\eta-s-1}\d x = \frac{1}{s-\eta}\left\{ \frac{1}{\theta_1^{s-\eta}}-\frac{1}{\theta_2^{s-\eta}} \right\}
= \frac{\theta_2-\theta_1}{\overline{\theta}^{s-\eta+1}},
\end{align*}
for some $\overline{\theta}\in [\theta_1,\theta_2]$.
Applying the above formula to $J_{3, N}(s)$, we get
\begin{align*}
J_{3, N}(s)&=\int_{\sigma_3-iN}^{\sigma_3+iN}
\frac{D(\eta)}{\eta}\sum_{m\geq 1} 
\frac{2\delta}{\overline{m}^{s-\eta+1}}\d\eta
=2\delta\sum_{m\geq 1} \int_{\sigma_3-iN}^{\sigma_3+iN}\frac{D(\eta)}{\overline{m}^{s-\eta+1}\eta}\d \eta,
\end{align*}
 where $\overline{1/2}\in [1, 1+\delta]$ and $\overline{m}\in [m-\delta, m+\delta]$ for all integers $m\geq2$.
In the above calculation, we can interchange the series and the integral as the series is absolutely convergent. 
So we have 
\begin{align*}
J_{3, N}(s)&\ll \delta \sum_{m\geq 1} \int_{-N}^{N}\frac{1}{(1+|v|)\overline{m}^{\sigma-\sigma_3+1}}\d v \quad
\text{( substituting }\eta=\sigma_3+iv\text{ )}\\
&\ll \delta \log N \sum_{m\geq 1} \frac{1}{\overline{m}^{\sigma-\sigma_3+1}} \ll \frac{\log N}{\sqrt N}.
\end{align*}
Here we used the fact that for $\sigma>\sigma_3$, the series
\[\sum_{m\geq 1}\frac{1}{\overline{m}^{s-\eta+1}}\]
is absolutely convergent.
This proves our required result.
\end{proof}

\subsection{Proof of  Theorem \ref{thm:analytic_continuation_mellin_transform}}
\begin{proof}
 From equation (\ref{eq:A_min_BN}) and Lemma \ref{lem:J1lim}, \ref{lem:J2lim} and \ref{lem:J3lim}, we get 
 \[A(s)=\lim_{N\rightarrow \infty}B'_N(s)\]
 for  $\Re(s)>\sigma_3+1$, and where $B'_N(s)$ is defined by (\ref{eq:B_N_s_defi}).
 From Lemma~\ref{lem:FubiniForExchange}, we have
 \[B(s)=\lim_{N\rightarrow \infty}B'_N(s).\]
 This gives $A(s)$ and $B(s)$ are equal for $\Re(s)>\sigma_3+1$. By analytic continuation,
 $A(s)$ and $B(s)$ are equal for any $s$ that lies right to $\Co$.
\end{proof}

In this chapter, we shall use the meromorphic continuation of $A(s)$ 
derived in Theorem~\ref{thm:analytic_continuation_mellin_transform}
to obtain mesure theoretic $\Omega_\pm$ results for $\Delta(x)$.

\section{Alternative Approches}
Theorem~\ref{thm:analytic_continuation_mellin_transform} gives a way for meromorphic continuation of $A(s)$ by formulating 
it as a contour integral. This theorem has its significance in terms of elegance and generality. 
However, there are alternative and easier ways in many cases. Below we give an example.

Note that
\[\sum_{n=1}^{\infty}\frac{a_n}{n^s}=-\int_{1}^{\infty}\Big(\sum_{n\leq x}a_n\Big){\d x^{-s}}\quad\text{ for }\Re(s)>\sigma_2.\]
This gives
\[\frac{D(s)}{s}=\int_{1}^{\infty}\Big(\sum_{n\leq x}a_n\Big)x^{-s-1}\d x\quad\text{ for }\Re(s)>\sigma_2.\]
So we can express $A(s)$ as
\begin{equation}
A(s)=\frac{D(s)}{s}-\int_1^{\infty}M(x)x^{-s-1}\d x\quad\text{ for }\Re(s)>\sigma_2. 
\end{equation}

The above formula reduces the problem of meromorphically continuing $A(s)$ to that of
\[\int_1^{\infty}M(x)x^{-s-1}\d x.\]

To demonstrate this method, we consider the case when $D(\eta)$ has a pole at $\eta=1$ and 
residue at this pole gives the main term $M(x)$, i.e $\mathcal P=\{1\}$. The following meromorphic functions may serve as 
examples of $D(\eta)$ in this situation:
\[\frac{\zeta(s)}{\zeta(2s)}, \zeta^2(s), \frac{\zeta^2(s)}{\zeta(2s)}, -\frac{\zeta'(s)}{\zeta(s)}, \ldots .\]
For a small positive real number $r$, we can 
write $M(x)$ as
\[M(x)=\frac{1}{2\pi i}\int_{|\eta-1|=r}\frac{D(\eta)x^\eta}{\eta}\d\eta.\]
Thus
\begin{align}
\notag
 \int_{1}^{\infty}\frac{M(x)}{x^{s+1}}\d x
 &=\int_1^{\infty}\frac{1}{2\pi i}\int_{|\eta-1|=r}\frac{D(\eta)x^\eta}{\eta}d\eta\frac{\d x}{x^{s+1}}\\
 \notag
 &=\frac{1}{2\pi i}\int_{|\eta-1|=r}\frac{D(\eta)}{\eta}\left(\int_1^{\infty}\frac{\d x}{x^{s-\eta+1}}\right) \d\eta\\
 \notag
 & \text{( using \cite[Theorem~B.3.1,~(b)]{DeitmarHarmonic} )}\\ \label{eq:analytic_continuation_main_term}
 &=\frac{1}{2\pi i}\int_{|\eta-1|=r}\frac{D(\eta)}{\eta(s-\eta)}\d\eta.
\end{align}
Let the Laurent series expansion of $D(\eta)$ at $\eta=1$ be 
\[\frac{D(\eta)}{\eta}=\sum_{n\leq N} \frac{b_n}{(\eta-1)^n}+H(\eta),\]
where $H(\eta)$ is holomorphic for $\Re(\eta)>\sigma_1$. Plugging in this expression for $D(\eta)$ in (\ref{eq:analytic_continuation_main_term}),
we get
\begin{equation}\label{eq:analytic_continuation_main_term2}
\int_{1}^{\infty}\frac{M(x)}{x^{s+1}}\d x = \sum_{n\leq N} b_n \frac{1}{2\pi i}\int_{|\eta-1|=r}\frac{\d\eta}{(\eta-1)^n(s-\eta)}.
\end{equation}
Let $\Re(s)\geq1+2r$, then
\[\frac{|\eta-1|}{|s-1|}\leq \frac{1}{2}\quad\text{ for }|\eta-1|=r.\]
This gives
\[\frac{1}{s-\eta}=\sum_{n=0}^{\infty}\frac{(\eta-1)^n}{(s-1)^{n+1}}\]
is an absolutely convergent series. Using the above expansion of $(s-\eta)^{-1}$ 
in (\ref{eq:analytic_continuation_main_term2}), we have
\begin{align*}
\int_{1}^{\infty}\frac{M(x)}{x^{s+1}}\d x &=  
\sum_{n\leq N} b_n \frac{1}{2\pi i}\int_{|\eta-1|=r}\left\{\sum_{m=0}^{\infty}\frac{(\eta-1)^m}{(s-1)^{m+1}}\right\}
\frac{\d\eta}{(\eta-1)^n}\\
&=\sum_{n\leq N}\frac{b_n}{(s-1)^n} \quad (\text{ by \cite[Theorem~6.1]{ComplexAnKra}})\\
&=\frac{D(s)}{s}-H(s).
\end{align*}
Thus we got
\[A(s)=H(s)\ \text{ for } \ \Re(s)\ge 1+2r.\]
But the right hand side is holomorphic for $\Re(s)>\sigma_1$ hence the formula
gives analytic continuation of $A(s)$ in the half plane $Re(s)>\sigma_1$.

Similar calculations can be done when the main term $M(x)$ is more complicated. 

%% file: landau_theorem.tex
\chapter{Landau's Oscillation Theorem}\label{chap:landau_theorem}
In this chapter, we revisit a result due to Landau and obtain $\Omega_\pm$ results for $\Delta(x)$ using certain singularities of $D(s)$.
Also we shall measure the fluctuations of $\Delta(x)$ in terms of $\Omega$ bounds, which generalizes a result of Kaczorowski and Szyd{\l}o \cite{KaczMeasure},
and a result of Bhowmik, Schlage-Puchta and Ramar{\'e} \cite{gautami}. 

\section{Landau's Criterion for Sign Change}
We begin with a result on real valued functions that do not change sign. This appears in a paper of Landau \cite{Landau}, attributed 
to Phragm{\'e}n and stated without a proof. Here we present a proof of this result following \cite[II.1.3, Theorem~6]{TenenAnPr}. 
\begin{thm}[Phragm{\'e}n-Landau]\label{thm:phragmen_landau}
 Let $f(x)$ be a real valued piecewise continuous function defined for $x\geq 1$, and bounded on every compact intervals. Let $F(s)$ be 
 its Mellin transform:
 \[F(s)=\int_1^\infty \frac{f(x)}{x^{s+1}}\d x,\]
 converges absolutely in some complex right half plane. Also assume that $f(x)$ does not change sign for $x\geq x_0$, for some $x_0\geq 1$. 
 If $F(s)$ diverges for some real $s$, then there exist a real number $\sigma_0$ satisfying the following properties: 
 \begin{enumerate}
\item[(1)] the integral defining $F(s)$ is divergent for $s<\sigma_0$ and convergent for $s>\sigma_0$,
\item[(2)] $s=\sigma_0$ is a singularity of $F(s)$,
\item[(3)] and $F(s)$ is analytic for $\Re(s)>\sigma_0$.
\end{enumerate}
\end{thm}
\begin{proof}
Let $\sigma_0$ be:
 \[\sigma_0=\inf\{\sigma\in \mathbb R:F(\sigma)\ \text{ converges}\}.\]
We shall show that $\sigma_0$ satisfies the properties given in the theorem. 

As $f(x)$ does not change 
sign for $x\geq x_0$, convergence of $F(\sigma)$ implies the absolute convergence of $F(s)$ for $\Re(s)\geq\sigma$. This proves (1) and (3).
To prove (2), we proceed by method of contradiction. Assume that $s=\sigma_0$ is not a singularity of $F(s)$. Then there exist 
$\sigma_0'>\sigma_0$ and $r>\sigma_0'-\sigma_0$ such that $F(s)$ has the following Taylor series expansion: 
\[F(s)=\sum_{k=0}^\infty\frac{1}{k!}F^{(k)}(\sigma_0')(s-\sigma_0')^k,\]
for all $s$ satisfying $|s-\sigma_0'|<r$.
\begin{claim}[1] For $\sigma_0'$ as above, we have
\[F(s)=\sum_{k=0}^{\infty}\frac{1}{k!}(s-\sigma_0')^k\int_1^{\infty}(-\log x)^k\frac{f(x)}{x^{\sigma_0'+1}}\d x.\] 
\end{claim}
\textit{Proof of Claim (1).}
By Cauchy's integral formula, we can write 
\[F^{(k)}(\sigma_0')=  \frac{k!}{2\pi i}\int_{\mathcal C}\frac{F(z)}{(z-\sigma_0')^{k+1}}\d z,\]
where $\mathcal C$ is a circle with a small enough radius having its center at $\sigma_0'$.
So we have
\[F(s)=\sum_{k=0}^{\infty}\frac{(s-\sigma_0')^k}{2\pi i }\int_{\mathcal C}\frac{1}{(z-\sigma_0')^{k+1}}\int_1^{\infty}\frac{f(x)}{x^{z+1}}\d x~\d z.\]
Suppose we can exchange the integrals of $x$ and $z$, then
\begin{align*}
F(s)&=\sum_{k=0}^{\infty}\frac{(s-\sigma_0')^k}{k!}\int_1^{\infty}\frac{f(x)}{x}
\frac{k!}{2\pi i}\int_{\mathcal C}\frac{x^{-z}\d z}{(z-\sigma_0')^{k+1}}\d x\\
&=\sum_{k=0}^{\infty}\frac{1}{k!}(s-\sigma_0')^k\int_1^{\infty}(-\log x)^k\frac{f(x)}{x^{\sigma_0'+1}}\d x,
\end{align*}
which proves Claim~1 conditionally. The only thing remains is to show that we can exchange integrals of $x$ and $z$.
If we choose $\mathcal C$ with a small enough radius, then 
\[\int_1^\infty\frac{f(x)}{x^{\Re(z) + 1}}\d x\]
is absolutely convergent and so is the 
double integral 
\[\int_{\mathcal C}\frac{1}{(z-\sigma_0')^{k+1}}\int_1^{\infty}\frac{f(x)}{x^{z+1}}\d x~\d z.\]
By the theorem of 
Fubinni and Tonelli \cite[Theorem~B.3.1,~(b)]{DeitmarHarmonic},
we can exchange these two iterated integrals. This completes the proof of Claim~1.
\begin{claim}[2]
For $|s-\sigma_0'|<r$, the integral
\[F(s)=\int_{1}^{\infty}\frac{f(x)}{x^{s+1}}\d x\]
converges.
\end{claim}
\textit{Proof of Claim (2).}
We shall simplify $F(s)$ using Claim~1. We write
\[ F(s)=\sum_{k=0}^{\infty}\frac{(\sigma_0'-s)^k}{k!}\int_{1}^{\infty}\frac{(\log x)^kf(x)}{x^{\sigma_0' +1}}\d x. \]
 In the above identity, we can exchange the series and the integral as the series is absolutely convergent. So we have
\begin{align*} 
 &\int_{1}^{\infty}\frac{f(x)}{x^{\sigma_0'+1}}\left(\sum_{k=0}^{\infty}\frac{(\sigma_0'-s)^k}{k!}(\log x)^k\right)\d x\\
  &=\int_{1}^{\infty}\frac{f(x)}{x^{\sigma_0'+1}}\exp((\sigma_0'-s)\log x)\d x
 =\int_{1}^{\infty}\frac{f(x)}{x^{s+1}}\d x.
\end{align*}
This completes the proof of Claim~2.

But Claim~2 implies that we have a real number smaller than $\sigma_0$, say $\sigma_0''$, such that the integral of $F(\sigma_0'')$ converges. This
is a contradiction to the definition of $\sigma_0$. So $\sigma_0$ is a singularity of $F(s)$, which proves (2).
\end{proof}

The following theorem appears in \cite[Section~2]{AnderOsci} without a proof and is attributed to Landau. We shall 
prove this theorem using Theorem~\ref{thm:phragmen_landau}.

\begin{thm}[Phragm{\'e}n-Landau-Anderson-Stark \label{thm:landau_representation_integral}]
 Let $f(x)$ be a real valued piecewise continuous function defined on $[1, \infty)$,
 bounded on every compact intervals, and does not change sign when $x>x_0$ for some $1<x_0<\infty$. Define 
 \[F(s):=\int_{1}^{\infty}\frac{f(x)}{x^{s+1}}\d x,\]
and assume that the above integral is absolutely convergent in some half plane.
Further, assume that we have an analytic continuation of $F(s)$ in a region containing a part of the real line
\[l(\sigma_0, \infty):=\{\sigma + i0: \sigma > \sigma_0\}.\]
Then the integral representing $F(s)$ is absolutely convergent for $\Re(s)>\sigma_0$, and hence $F(s)$ is an analytic function in this region. 
\end{thm}
\begin{proof}
By Theorem~\ref{thm:phragmen_landau}, if 
\[\int_1^\infty \frac{f(x)}{x^{\sigma'+1}}\d x\]
diverges for some $\sigma'>\sigma_0,$ 
then there exist a real number $\sigma_0'\geq\sigma'>\sigma_0$ such that $F$ is not analytic at $\sigma_0'$. But this contradicts our assumption 
that $F$ is analytic on $l(\sigma_0, \infty)$. So the integral
\[\int_1^\infty \frac{f(x)}{x^{\sigma'+1}}\d x \ \text{ converges } \  \forall \ \sigma'>\sigma_0,\]
and since $f$ does not change
sign for $x\geq x_0$, $F(s)$ converges absolutely for $\Re(s)>\sigma_0$. This also gives that $F(s)$ is analytic for $\Re(s)>\sigma_0$.
\end{proof}

The above two theorems give some criteria when a function does not change sign. In the next section we will use these results to
show the sign changes of $\Delta(x)$.

\section{$\Omega_\pm$ Results}
Consider the Mellin transform $A(s)$ of $\Delta(x)$. We need the following assumptions to apply Theorem~\ref{thm:landau_representation_integral}.

\begin{asump}\label{as:for_landau}
 Suppose there exists a real number $\sigma_0$, $0<\sigma_0<\sigma_1$,
 such that $A(s)$ has the following properties.
\begin{itemize}
\item[(i)] There exists $t_0\neq 0$ such that 
 \begin{equation*}
 \lambda:=\limsup\limits_{\sigma\searrow\sigma_0}(\sigma-\sigma_0)|A(\sigma+it_0)|>0 .
 \end{equation*}
\item[(ii)] At $\sigma_0$ we have 
\begin{align*}
l_s&:=\limsup\limits_{\sigma\searrow\sigma_0}(\sigma-\sigma_0)A(\sigma) < \infty,\\
l_i&:= \liminf\limits_{\sigma\searrow\sigma_0}(\sigma-\sigma_0)A(\sigma) > -\infty.
\end{align*}
\item[(iii)]
The limits $l_i, l_s$ and $\lambda$ satisfy
\[l_i+\lambda>0\quad \text{ and } \quad l_s-\lambda<0.\]
\item[(iv)] We can analytically continue $A(s)$ in a region containing the real line
$l(\sigma_0, \infty)$.
\end{itemize} 
\end{asump}

\begin{rmk}
From Assumptions~\ref{as:for_landau},(i), we see that $\sigma_0+it_0$ is a singularity of $A(s)$.
\end{rmk}
We construct the following sets for further use.
\begin{defi}\label{def:A1A2}
With $l_s, l_i$ and $\lambda$ as in Assumptions~\ref{as:for_landau} and for an $\epsilon$ such that $0<\epsilon<\min(l_i+\lambda, \lambda-l_s)$, define
 \begin{align*}
 &\mathcal{A}_1:=\{x:x \in [1, \infty), \Delta(x)>(l_i+\lambda-\epsilon)x^{\sigma_0}\},\\
\text{ and }\quad&\mathcal{A}_2:=\{x:x \in [1, \infty), \Delta(x)<(l_s-\lambda+\epsilon)x^{\sigma_0}\}.
\end{align*}
\end{defi}

Under Assumptions~\ref{as:for_landau} and using methods from \cite{KaczMeasure}, we can derive the following measure theoretic theorem.

\begin{thm}\label{thm:landu_omegapm}
Let the conditions in Assumptions~\ref{as:for_landau} hold. Then for any 
 real number $M>1$, we have
 \begin{align*}
 \mu(\mathcal{A}_1\cap[M, \infty])>0,\\
 \text{ and }\quad \mu(\mathcal{A}_2\cap[M, \infty])>0.
 \end{align*}
This implies
 \begin{align*}
\Delta(x)=\Omega_\pm(x^{\sigma_0}). 
\end{align*}
\end{thm}
 
 \begin{proof}
We prove the Theorem only for $\mathcal{A}_1$ as the other part is similar.
 
 Define 
 \begin{align*}
  &g(x):=\Delta(x)-(l_i+\lambda-\epsilon)x^{\sigma_0}, \quad &&G(s):=\int_{1}^{\infty}\frac{g(x)}{x^{s+1}}{\d x};\\
  &g^+(x):=\max(g(x), 0), \quad &&G^+(s):= \int_{1}^{\infty}\frac{g^+(x)}{x^{s+1}}{\d x};\\
  &g^-(x):=\max(-g(x), 0), \quad &&G^-(s):= \int_{1}^{\infty}\frac{g^-(x)}{x^{s+1}}{\d x}.
 \end{align*}
With the above notations, we have
\begin{align*}
 &g(x)=g^+(x)-g^-(x), \\
 \text{ and }\quad &G(s)=G^+(s)-G^-(s).
\end{align*}
Note that
\begin{align*}
G(s)&=A(s) - \int_1^{\infty}(l_i+\lambda-\epsilon)x^{\sigma_0-s-1}{\d x} \\
&=A(s) + \frac{l_i+\lambda-\epsilon}{\sigma_0-s},\quad \text{for }  \Re(s)> \sigma_0.
\end{align*}
So $G(s)$ is analytic wherever $A(s)$ is, except possibly for a pole at $\sigma_0$.
This gives
\begin{align}\label{eq:G_lim_sigma0t0}
 \limsup\limits_{\sigma\searrow\sigma_0}(\sigma-\sigma_0)|G(\sigma+it_0)|
 =\limsup\limits_{\sigma\searrow\sigma_0}(\sigma-\sigma_0)|A(\sigma+it_0)|=\lambda.
\end{align}
We shall use the above limit to prove our theorem. 
We proceed by method of contradiction. 
Assume that there exists an $M>1$ such that 
\[\mu(\mathcal{A}_1\cap[M, \infty))=0.\] 
This implies
\[G^+(\sigma)=\int_{1}^{\infty}\frac{g^+(x)}{x^{s+1}}{\d x} = \int_{1}^{M}\frac{g^+(x)}{x^{s+1}}{\d x} \] 
is bounded for any $s$, and so is an entire function.
By Assumptions~\ref{as:for_landau}, 
$A(s)$ and $G(s)$ can be analytically continued on the line $l(\sigma_0, \infty)$.
As $G(s)$ and $G^+(s)$ are analytic on $l(\sigma_0, \infty)$, $G^-(s)$ is also analytic on $l(\sigma_0, \infty)$.
The integral for $G^-(s)$ is absolutely convergent for $\Re(s)>\sigma_3+1$, and $g^-(x)$ is a piecewise continuous function bounded 
on every compact sets. This suggests that we can apply Theorem~\ref{thm:landau_representation_integral} to 
$G^-(s)$, and conclude that 
\[G^-(s)=\int_{1}^{\infty}\frac{g^-(x)}{x^{s+1}}{\d x} \]
is absolutely convergent for $\Re(s)>\sigma_0$.

From the above discussion, we summarize that the Mellin transforms of $g, g^+$ and $g^-$ converge absolutely 
for $\Re(s)>\sigma_0$. As a consequence, we see that $G(\sigma), G^+(\sigma)$ and $G^-(\sigma)$ 
are finite real numbers for $\sigma>\sigma_0$. For $\sigma>\sigma_0$, we compare $G^+(\sigma)$ and $G^-(\sigma)$ in the following two cases.

\begin{itemize}
 \item [Case 1: ] $G^+(\sigma)<G^-(\sigma)$.
 
 In this case,
 \begin{align*}
 (\sigma-\sigma_0)|G(\sigma+ it_0)|&\leq (\sigma-\sigma_0)|G(\sigma)|\\
 &=-(\sigma-\sigma_0)G(\sigma)\\
 &=-(\sigma-\sigma_0)A(\sigma) + l_i+\lambda-\epsilon.
\end{align*}
So we have
\begin{equation*}
 \limsup\limits_{\sigma\searrow\sigma_0}(\sigma-\sigma_0)|G(\sigma+it_0)|
 \leq  l_i+\lambda-\epsilon - \liminf\limits_{\sigma\searrow\sigma_0}(\sigma-\sigma_0)A(\sigma)
 \leq \lambda-\epsilon.
\end{equation*}
This contradicts (\ref{eq:G_lim_sigma0t0}). 

\item[Case 2: ] $G^+(\sigma)\geq G^-(\sigma)$.

We have,
\begin{align*}
 (\sigma-\sigma_0)|G(\sigma+it_0)|&\leq (\sigma-\sigma_0)G^+(\sigma)\\
 &=O(\sigma-\sigma_0) \quad (G^+(\sigma)\text{ being a bounded integral}).
\end{align*}
Thus
\[ \limsup\limits_{\sigma\searrow\sigma_0}(\sigma-\sigma_0)|G(\sigma+it_0)|=0.\]
This contradicts (\ref{eq:G_lim_sigma0t0}) again.
\end{itemize}

Thus $\mu(\mathcal{A}_1\cap[M, \infty))>0$ for any $M>1$, which completes the proof.
\end{proof}

\section{Measure Theoretic $\Omega_\pm$ Results}
Now we know that $\mathcal{A}_1$ and $\mathcal{A}_2$ 
are unbounded. But we do not know how the size of these sets grow. 
An answer to this question was given by Kaczorowski and Szyd{\l}o in \cite[Theorem~4]{KaczMeasure}.
\begin{thm}[Kaczorowski and Szyd\l o \cite{KaczMeasure}]\label{thm:kaczorowski_correct}
Let the conditions in Assumptions~\ref{as:for_landau} hold.
Also assume that for a non-decreasing positive continuous function $h$ satisfying
\[h(x)\ll x^{\epsilon},\]
we have
\begin{equation}\label{eq:normDelta}
\int_{T}^{2T}\Delta^2(x){\d x}\ll T^{2\sigma_0 + 1}h(T).
\end{equation}
Then as $T\rightarrow \infty$,
\[\mu\left( \mathcal{A}_j\cap[1, T]\right)=\Omega\Big(\frac{T}{h(T)}\Big)\quad \text{ for } j=1, 2.\]
\end{thm}
In \cite{KaczMeasure}, Kaczorowski and Szyd{\l}o applied this theorem to the error term appearing in the asymptotic formula for the fourth power moment
of Riemann zeta function. We write this error term as $E_2(x)$: 
\[ \int_0^x \left|\zeta\left(\half+it\right)\right|^4 \d t = x P(\log x) + E_2(x),\]
where $P$ is a polynomial of degree $4$. Motohashi \cite{Motohashi1} proved that 
\[ E_2(x) \ll x^{2/3+\epsilon}, \]
and further in \cite{Motohashi2} he showed that 
\[ E_2(x)=\Omega_{\pm}(\sqrt x). \]
Theorem of  Kaczorowski and Szyd{\l}o ( Theorem~\ref{thm:kaczorowski} ) gives that there exist $\lambda_0, \nu>0$ such that 
\[\mu\{1\le x\le T: E_2(x)>\lambda_0\sqrt x \} = \Omega(T/{(\log T)^{\nu}}) \]
and 
\[\mu\{1\le x\le T: E_2(x)<-\lambda_0\sqrt x \} = \Omega(T/{(\log T)^{\nu}})\]
as $T\rightarrow \infty$.
These results not only prove $\Omega_{\pm}$-results, but also give quantitative estimates for the occurrences 
of such fluctuations.
The above theorem of Kaczorowski and Szyd\l o has been generalized by Bhowmik, Ramar\'e and Schlage-Puchta by localizing the fluctuations
of $\Delta(x)$ to $[T, 2T]$.
Proof of this theorem follows from \cite[Theorem~2]{gautami} (also see Theorem~\ref{thm:gautami} below).
\begin{thm}[Bhowmik, Ramar\'e and Schlage-Puchta \cite{gautami}]\label{thm:kaczorowski}
Let the assumptions in Theorem~\ref{thm:kaczorowski_correct} hold.
Then as $T\rightarrow \infty$,
\[\mu\left( \mathcal{A}_j\cap[T, 2T]\right)=\Omega\Big(\frac{T}{h(T)}\Big)\quad \text{ for } j=1, 2.\]
\end{thm}
An application of the above theorem to Goldbach's problem is given in \cite{gautami}.
Let
\begin{align*}
&\sum_{n\leq x} G_k(n) = \frac{x^k}{k!} -k \sum_\rho \frac{x^{k-1+\rho}}{\rho(1+\rho)\cdots(k-1+\rho)} + \Delta_k(x), 
\end{align*}
where the Goldbach numbers $G_k(n)$ are defined as
\[\quad G_k(n)=\sum_{\substack{n_1,\ldots n_k \\ n_1+\cdots+n_k=n}}\Lambda(n_1)\cdots\Lambda(n_k),\]
and $\rho$ runs over nontrivial zeros of the Riemann zeta function $\zeta(s)$.
Bhowmik, Ramar\'e and Schlage-Puchta proved that under Riemann Hypothesis
\begin{align*}
&&\mu\{T\le x\le 2T: \Delta_k(x)>(\mathfrak c_k + \mathfrak c_k')x^{k-1} \} = \Omega(T/{(\log T)^{6}})& \\
&\text{ and }&\mu\{T\le x\le 2T: \Delta_k(x)<(\mathfrak c_k-\mathfrak c_k') x^{k-1} \} = \Omega(T/{(\log T)^{6}})& \ \text{ as } T\rightarrow \infty,
\end{align*}
where $k\geq 2$ and $\mathfrak{c}_k, \mathfrak{c}_k'$ are well defined real number depending on $k$ with $\mathfrak{c}_k'>0$.

Note that Theorem~\ref{thm:kaczorowski_correct} implies Theorem~\ref{thm:kaczorowski}, but both the theorems are applicable to the same set of examples.
The main obstacle in applicability of these theorems is the condition (\ref{eq:normDelta}).
For example, if $\Delta(x)$ is the error term in approximating
$\sum_{n\leq x}|\tau(n, \theta)|^2$, we can not apply Theorem~\ref{thm:kaczorowski_correct} and Theorem~\ref{thm:kaczorowski}. However, the following theorem due to
the author and A. Mukhopadhyay \cite[Theorem~3]{MeasureOmega} overcomes this obstacle by replacing the condition (\ref{eq:normDelta}). 

\begin{thm}\label{thm:omega_pm_main}
Let the conditions in Assumptions~\ref{as:for_landau} hold.
 Assume that there is an 
analytic continuation of $A(s)$ in a region containing the real line $l(\sigma_0, \infty)$.
Let $h_1$ and $h_2$ be two positive functions
such that
\begin{equation}\label{eq:second_moment_error}
\int_{[T, 2T]\cap \mathcal{A}_j}
\frac{\Delta^2(x)}{x^{2\sigma_0+1}}\d x\ll h_j(T)\quad\text{ for } j=1, 2.
\end{equation}
Then as $T \longrightarrow \infty$, 
\begin{equation}\label{eq:omega_pm_measure}
 \mu(\mathcal{A}_j\cap[T, 2T])=\Omega\Big(\frac{T}{h_j(T)}\Big)\quad\text{ for } j=1, 2.
\end{equation}
\end{thm}
Below we state an integral version of Theorem~\ref{thm:kaczorowski} as in \cite{gautami}. 
\begin{thm}[Bhowmik, Ramar\'e and Schlage-Puchta \cite{gautami}]\label{thm:gautami}
Suppose the conditions in Assumptions~\ref{as:for_landau} hold, and let $h(x)$ be as in Theorem~\ref{thm:kaczorowski}. Then as $\delta\rightarrow 0^+$,
\[\int_1^\infty\frac{\mu(\A_j\cap[x, 2x])h(4x)}{x^{2+\delta}}\d x=\Omega\Big(\frac{1}{\delta}\Big),\ \text{ for } j=1, 2.\] 
\end{thm}
In our next theorem, we generalize Theorem~\ref{thm:kaczorowski_correct}, \ref{thm:kaczorowski},  \ref{thm:omega_pm_main} and \ref{thm:gautami}.

\begin{thm}\label{thm:omega_pm_main_new}
Let the conditions in Theorem~\ref{thm:omega_pm_main} hold. Then as $\delta\rightarrow 0^+$,
\begin{equation}
 \int_1^\infty\frac{\mu(\mathcal{A}_j\cap[x, 2x])h_j(x)}{x^{2+\delta}}\d x=\Omega\Big(\frac{1}{\delta}\Big)\quad\text{ for }j=1, 2.
\end{equation}
\end{thm}
\begin{proof}
 We shall prove the theorem for $j=1$; the proof is similar for $j=2$.
We define $g, g^+, g^-, G, G^+$ and $G^-$, as in Theorem~\ref{thm:landu_omegapm}.
Let 
\[m^\#(x):=h_1(x)\mu(\A_1\cap[x, 2x])x^{-1}.\]
First, we shall show:
\begin{claim}[1] As $\delta\rightarrow 0$,
\[ \sum_{k\geq 0}\frac{m^\#(2^k)}{2^{k\delta}}=\Omega\left(\frac{1}{\delta}\right).\]
\end{claim}
Assume that
\begin{equation}\label{eq:measure_contradic_A1}
\sum_{k\geq 0}\frac{m^\#(2^k)}{2^{k\delta}}=o\left(\frac{1}{\delta}\right).
\end{equation}
From the above assumption, we may obtain an upper bound for $G^+(\sigma)$ as follows:
\begin{align*}
& \int_{\A_1}\frac{g^+(x)\d x}{x^{\sigma+1}}
\leq \sum_{k\geq 0}\int_{\A_1\cap[2^k, 2^{k+1}]}\frac{\Delta(x)\d x}{x^{\sigma+1}}
\quad ( \text{as } \Delta(x)> g(x) \text{ on } \A_1 )\\
&\leq \sum_{k\geq 0}\left(\int_{\A_1\cap[2^k, 2^{k+1}]}\frac{\Delta^2(x)\d x}{x^{2\sigma_0+1}}\right)^{\half}
\left(\frac{\mu(\A_1\cap[2^k, 2^{k+1}])}{2^{k(2\delta+1)}}\right)^{\half}
\ (\text{where } \sigma-\sigma_0=\delta>0)\\
&\leq c_3 \sum_{k\geq 0}\left(\frac{h_1(2^k)\mu(\A_1\cap[2^k, 2^{k+1}])}{2^{k(2\delta+1)}}\right)^{\half}
\leq c_3\sum_{k\geq 0}\left(\frac{m^\#(2^k)}{2^{2k\delta}}\right)^{\half}.
\end{align*}
From the above inequality, we get 
\begin{align}\label{eq:Gplus_upper_bound}
 \delta G^+(\sigma)\ll \delta\left(\sum_{k\geq 0}\frac{1}{2^{k\delta}}\right)^\half 
 \left(\sum_{k\geq 0}\frac{m^\#(2^k)}{2^{k(\sigma-\sigma_0)}}\right)^\half 
 =o(1)
\end{align}
as $\delta\rightarrow 0^+$.
Therefore
\[G^+(s)=\int_1^{\infty}\frac{g^+(x)\d x}{x^{s+1}}\]
is absolutely convergent for $\Re(s)>\sigma_0$, and so it is analytic in this region.
But
\[G^-(s)=G(s)-G^+(s),\]
and $G$ is analytic on $l(\sigma_0, \infty)$. So $G^-$ is also analytic on $l(\sigma_0, \infty)$. 
Using Theorem~\ref{thm:landau_representation_integral}, we get
\[G^+(s)=\int_1^{\infty}\frac{g^+(x)\d x}{x^{s+1}}\]
is absolutely convergent for $\Re(s)>\sigma_0$. Absolute convergence of the integrals of $G$ and $G^+$
implies that the Mellin transformation of $g^-(x)$, given by
\[ G^-(s)=\int_1^{\infty}\frac{g^-(x)\d x}{x^{s+1}},\]
is also absolutely convergent for $\Re(s)>\sigma_0$. As a consequence, we get $G(\sigma), G^+(\sigma)$, and $G^-(\sigma)$
are finite non-negative real numbers for $\sigma>\sigma_0$.
As indicated in Case-1 of Theorem~\ref{thm:landu_omegapm}, we can not have 
\[G^+(\sigma)<G^-(\sigma)\ \text{ when }\sigma > \sigma_0.\] 
So we always have
\[G^+(\sigma)\geq G^-(\sigma).\]
Using (\ref{eq:Gplus_upper_bound}), 
\[\limsup\limits_{\sigma\searrow\sigma_0}(\sigma-\sigma_0)|G(\sigma+it_0)|\leq 
\limsup\limits_{\sigma\searrow\sigma_0}(\sigma-\sigma_0)G(\sigma)=0.\]
This is a contradiction to (\ref{eq:G_lim_sigma0t0}), and so (\ref{eq:measure_contradic_A1}) is wrong. This proves our Claim.

Now we are ready to prove the theorem. 
For $k\geq1$, observe that
\begin{align*}
\int_{k-1}^{k}\frac{m^\#(2^x)}{2^{\delta x}} \d x
&=\int_{k-1}^{k} \frac{h_1(2^x)\mu(\A_1\cap[2^x, 2^{x+1}])}{2^{x(\delta +1)}} \d x
=\int_{k-1}^{k}\int_{2^x}^{2^{x+1}}\frac{h_1(2^x)}{2^{\delta x +x}}\d\A_1(t)\d x\\
&(\text{where } \A_1(t) \ \text{ is the indicator function of }\ \A_1) \\
= &\int_{2^{k-1}}^{2^k}\int_{k-1}^{\frac{\log t}{\log2}}\frac{h_1(2^x)}{2^{x(1+\delta)}}\d x\d\A_1(t) 
  +\int_{2^k}^{2^{k+1}}\int_{\frac{\log t}{\log2}-1}^{k}\frac{h_1(2^x)}{2^{x(1+\delta)}}\d x\d\A_1(t)
\end{align*}
From the above identity, we have
\[\int_{k-1}^{k}\frac{m^\#(2^x)}{2^{\delta x}}\d x\geq \int_{2^k}^{2^{k+1}}\int_{\frac{\log t}{\log2}-1}^{k}
\frac{h_1(2^x)}{2^{x(1+\delta)}}\d x\d\A_1(t)\]
and  
\[\int_{k}^{k+1}\frac{m^\#(2^x)}{2^{\delta x}}\d x\geq 
\int_{2^{k}}^{2^{k+1}}\int_{k}^{\frac{\log t}{\log2}} 
\frac{h_1(2^x)}{2^{x(1+\delta)}}\d x\d\A_1(t). 
\]
So we get
\begin{align}
\notag
\int_{k-1}^{k+1}\frac{m^\#(2^x)}{2^{\delta x}}\d x
&\geq \int_{2^k}^{2^{k+1}}\int_{\frac{\log t}{\log 2}-1}^{k}\frac{h_1(2^x)}{2^{x(1+\delta)}}\d x\d\A_1(t)
+\int_{2^k}^{2^{k+1}}\int_{k}^{\frac{\log t}{\log2}}\frac{h_1(2^x)}{2^{x(1+\delta)}}\d x\d\A_1(t)\\
\notag
&=\int_{2^k}^{2^{k+1}}\int_{\frac{\log t}{\log 2}-1}^{\frac{\log t}{\log 2}}\frac{h_1(2^x)}{2^{x(1+\delta)}}\d x \d\A_1(t).
\end{align}
Now, we may use the fact that $h_1$ is a monotonically increasing function having polynomial growth, and simplify the above calculation as follows:
\begin{align}
\notag
&\int_{k-1}^{k+1}\frac{m^\#(2^x)}{2^{\delta x}} \d x
\geq h_1(2^{k})\int_{2^k}^{2^{k+1}}\int_{\frac{\log t}{\log2}-1}^{\frac{\log t}{\log2}}
\frac{\d x}{2^{x(1+\delta)}}\d\A_1(t)\\
\notag
&= \frac{h_1(2^{k})}{\log 2}\int_{2^k}^{2^{k+1}}
\Big(2^{-\left(\frac{\log t}{\log2}-1\right)(1+\delta)}-2^{-\frac{\log t}{\log2}(1+\delta)}\Big)\d\A_1(t)\\
\label{simplification_m3}
&= \frac{h_1(2^{k})}{\log 2}\int_{2^k}^{2^{k+1}}\frac{2^{1+\delta}-1}{t^{1+\delta}}\d\A_1(t)
\geq \frac{h_1(2^{k})}{2^{(k+1)(\delta +1)}}\mu(\A_1\cap[2^k, 2^{k+1}])
\geq \quater\frac{m^\#(2^k)}{2^{k\delta}}. 
\end{align}
Now using Claim~(1), we get
\[\int_0^\infty\frac{m^\#(2^x)}{2^{\delta x}} \d x\gg \sum_{k=1}^\infty\frac{m^\#(2^k)}{2^{k\delta}}=\Omega\left(\frac{1}{\delta}\right).\]
Changing the variable $x$ to $u=2^x$ in the above inequality gives
\begin{align*}
 &&\frac{1}{\log2} \int_1^\infty\frac{m^\#(u)}{u^{1+\delta}}\d u =\Omega \left(\frac{1}{\delta}\right),\\
 &\text{or }& \int_1^\infty\frac{\mu(\A_j\cap[u, 2u])h_j(u)}{u^{2+\delta}}\d u  =\Omega \left(\frac{1}{\delta}\right).
\end{align*}
This proves the theorem.
\end{proof}

\begin{coro}\label{cor:measure_omega_pm_from_upper_bound}
 Let the conditions given in Theorem~\ref{thm:omega_pm_main} hold. Suppose we have a monotonically increasing positive function $h$ such that
\begin{equation}
 \Delta(x)=O(h(x)),
\end{equation}
then
 \begin{equation}
\mu(\mathcal{A}_j\cap[T, 2T])=\Omega\left(\frac{T^{1+2\sigma_0}}{h^2(T)}\right)\quad \text{ for } j=1, 2. 
 \end{equation}
\end{coro}

\begin{coro}\label{coro:omega_pm_secondmoment}
 Similar to Corollary~\ref{cor:measure_omega_pm_from_upper_bound}, we assume that the conditions in 
 Theorem~\ref{thm:omega_pm_main} hold. Then we have
 \begin{equation}\label{eq:omega_pm_secondmoment}
 \int_{[T, 2T]\cap \A_j}\Delta^2(x)\d x = \Omega(T^{2\sigma_0 + 1}) \quad \text{ for } j=1, 2. 
 \end{equation}
\end{coro}
\begin{proof}
 This Corollary follows from the proof of Theorem \ref{thm:omega_pm_main_new}. 
 We shall prove this Corollary for $\A_1$, and the proof for $\A_2$ is similar.  
Note that in the proof of Theorem~\ref{thm:omega_pm_main_new}, we showed that the integral for $G^+(s)$ 
is absolutely convergent for $\Re(s)>\sigma_0$ by assuming (\ref{eq:measure_contradic_A1}). Then we got a contradiction which 
proves Claim~(1) of Theorem~\ref{thm:omega_pm_main_new}. Now we proceed in a similar manner 
by assuming (\ref{eq:omega_pm_secondmoment}) is false. So we have
\[\int_{[T, 2T]\cap \A_1}\Delta^2(x)\d x = o(T^{2\sigma_0 + 1}). \]
So for an arbitrarily small constant $\varepsilon$, we have 
\begin{align*}
& |G^+(s)|\leq \int_{\A_1}\frac{g^+(x)\d x}{x^{\sigma+1}}
\leq \sum_{k\geq 0}\int_{\A_1\cap[2^k, 2^{k+1}]}\frac{\Delta(x)\d x}{x^{\sigma+1}}\\
&\leq \sum_{k\geq 0}\frac{1}{2^{k(\sigma-\sigma_0)}}\left(\int_{\A_1\cap[2^k, 2^{k+1}]}\frac{\Delta^2(x)\d x}{x^{2\sigma_0+1}}\right)^{1/2}\\
&\leq c_4(\varepsilon) + \varepsilon\sum_{k\geq k(\varepsilon)}\frac{1}{2^{k(\sigma-\sigma_0)}}, \\
\end{align*}
where $c_4(\varepsilon)$ is a positive constant depending on $\varepsilon$. From this we obtain that $G^+(s)$ is absolutely convergent
for $\Re(s)>\sigma_0$. Now onwards the proof is same as that of Theorem~\ref{thm:omega_pm_main_new}.
\end{proof}

\section{Applications}
Now we demonstrate applications of our theorems in the previous section to error terms appearing in two well known asymptotic 
formulas.
\subsection{Square Free Divisors}\label{subsec:sqfree_divisors}

Let $a_n=2^{\omega(n)}$, where $\omega(n)$ denotes the number of distinct prime factors of $n$; equivalently, $a_n$ denotes the number of square free
divisors of $n$. We write
\begin{align*}
\sum_{n\leq x}^* 2^{\omega(n)}=\M(x) + \Delta(x), 
\end{align*}
where 
\[\M(x)= \frac{x\log x}{\zeta(2)}+\left(-\frac{2\zeta'(2)}{\zeta^2(2)} + \frac{2\gamma - 1}{\zeta(2)}\right)x,\]
and by a theorem of H{\"o}lder \cite{holder}
\begin{equation}\label{eq:sqfree_divisors_ub_unconditional}
 \Delta(x)\ll x^{1/2}.
\end{equation}
Under Riemann Hypothesis, Baker \cite{baker} has improved the above upper bound to 
\[\Delta(x)\ll x^{4/11}.\]
It is easy to see that the Dirichlet series $D(s)$ has the following form:
\[D(s)=\sum_{n=1}^{\infty}\frac{2^{\omega(n)}}{n^s}=\frac{\zeta^2(s)}{\zeta(2s)}.\]
Let $A(s)$ be the Mellin transform of $\Delta(x)$ at $s$, and let $s_0$ be 
the zero of $\zeta(2s)$ with least positive imaginary part:
\begin{equation}\label{first_zeta_zero}
 2s_0=\half + i 14.134\ldots.
\end{equation}
 \begin{figure}
 \begin{tikzpicture}[yscale=0.9]
\draw [<->][thin, gray] (-1, -4.4)--(-1, 4.4);
\node at (-1.3, 0.22) {$0$};
\draw [<->][thin, gray] (4, 0)--(-3, 0);
\fill (3, 0) circle[radius=2pt];
\node at (2.8, 0.3) {$1$};
\fill (0.5, 3) circle[radius=1.5pt];
\node at (0.3, 3.2) {\small{$s_0$}};

\draw [thick] (3.6,-0.1)--(3.6, 0.1);
\node at (3.35, 0.35) {$\frac{5}{4}$};
\draw [thick] (0.5,-0.1)--(0.5, 0.1);
\node at (0.3, 0.35) {$\quater$};
\draw [thick] (2,-0.1)--(2, 0.1);
\node at (1.7, 0.35) {$\frac{3}{4}$};
\draw [thick] (0,-0.1)--(0, 0.1);
\node at (-0.3, 0.35) {$\frac{1}{5}$};

\draw [thick] (-1.1,-0.7)--(-0.9, -0.7);
\node at (-1.5, -0.7) {\small{$-2$}};
\draw [thick] (-1.1,0.7)--(-0.9, 0.7);
\node at (-1.33, 0.7) {\small{$2$}};
\draw [thick] (-1.1,4)--(-0.9, 4);
\node at (-1.5, 4) {\small{$14$}};

\draw [dashed] [postaction={decorate, decoration={ markings,
mark= between positions 0.09 and 0.98 step 0.28 with {\arrow[line width=1pt]{>},}}}]
(3.6, -4.4)--(3.6, -0.7)--(2, -0.7)--(2, 0.7)--(3.6, 0.7)--(3.6, 4.4);

\draw [dotted, thick] [postaction={decorate, decoration={ markings,
mark= between positions 0.3 and 0.7 step 0.2 with {\arrow[line width=1pt]{>},}}}]
(3.6, -4.4)--(3.6, -0.7)--(0, -0.7)--(0, 4)--(1, 4)--(1, 0.7)--(3.6, 0.7)--(3.6, 4.4);

\end{tikzpicture}
\caption{Contours for square-free divisors.}\label{fg:sq_free_divisors}
\end{figure}
We define a contour $\Co^{(1)}$ as union of the following five lines:
\begin{align*}
 \Co^{(1)}:=&\left(\frac{5}{4}-i\infty,\ \frac{5}{4}-i2\right]\cup\left[\frac{5}{4}-i2,\ \frac{3}{4}-i2\right]
 \cup\left[\frac{3}{4}-i2,\ \frac{3}{4}+i2\right]\\
 &\cup\left[\frac{3}{4}+i2,\ \frac{5}{4}+i2\right]\cup\left[\frac{5}{4}+i2,\ \frac{5}{4}+i\infty\right)
\end{align*}
The contour $\Co^{(1)}$ is represented by \lq dashed\rq \ lines in Figure~\ref{fg:sq_free_divisors}.
By Theorem~\ref{thm:analytic_continuation_mellin_transform}, we have
\[A(s)=\int_1^\infty\frac{\Delta(x)}{x^{s+1}}\d x=\int_{\Co^{(1)}}\frac{D(\eta)}{\eta(s-\eta)}\d \eta.\]
Now, we shift the contour $\Co^{(1)}$ to form a new contour $\Co^{(2)}$, so that 
\[1, \ s_0,\  l\left(\quater, \infty\right)\]
lie to the right of $\Co^{(2)}$. We have represented the contour $\Co^{(2)}$ by dotted lines in Figure~\ref{fg:sq_free_divisors}.

Since $s_0$ is a pole of $D(s)$ and is on the right side of $\Co^{(1)}$, we have
\begin{align*}
A(s)=\int_{\Co^{(2)}}\frac{D(\eta)}{\eta(s-\eta)}\d \eta + \underset{\eta=s_0}{\res}\left(\frac{D(\eta)}{\eta(s-\eta)}\right).
\end{align*}
From the above formula, we may compute the following limits:
\begin{equation*}
\lambda_1:= \lim_{\sigma\searrow0}\sigma|A(\sigma+s_0)| = |s_0|^{-1} \left|\underset{\eta=s_0}{\res}D(\eta)\right|>0
\end{equation*}
and 
\begin{equation*}
\lim_{\sigma\searrow0}\sigma A(\sigma+ 1/4)=0.
\end{equation*}
For a fixed $\epsilon_0>0$, 
\begin{align*}
 \A_1&=\left\{x: \Delta(x)>(\lambda_1-\epsilon_0)x^{1/4}\right\}\\
\text{and} \quad \A_2&=\left\{ x : \Delta(x)<(-\lambda_1+\epsilon_0)x^{1/4}\right\}.
\end{align*}
Using Corollary~\ref{cor:measure_omega_pm_from_upper_bound} and (\ref{eq:sqfree_divisors_ub_unconditional}), we get
\begin{equation}\label{eq:sqfree_divisors_omegaset_uc}
 \mu\left(\A_j\cap[T, 2T]\right)=\Omega\left(T^{1/2}\right) \text{ for } j=1, 2.
\end{equation}
Under Riemann Hypothesis 
we may show (also see Proposition~\ref{prop:upper_bound_second_moment_twisted_divisor}),
\[\int_{T}^{2T}\Delta^2(x)\ll T^{3/2+\epsilon}\ \text{ for any } \epsilon>0.\]
The above upper bound and Theorem~\ref{thm:omega_pm_main} give
\begin{equation}\label{eq:sqfree_divisors_omegaset}
 \mu\left(\A_j\cap[T, 2T]\right)=\Omega\left(T^{1-\epsilon}\right), \text{ for } j=1, 2 \ \text{ and for any } \ \epsilon>0.
\end{equation}

\subsection{The Prime Number Theorem Error}\label{subsec:pnt_error}
Consider the error term in the Prime Number Theorem:
\[\Delta(x)=\sum_{n\leq x}^*\Lambda(n)-x.\]
Let
\[\lambda_2=|2s_0|^{-1}, \]
where $2s_0$ is the first nontrivial zero of $\zeta(s)$ as in (\ref{first_zeta_zero}).
We shall apply Corollary~\ref{cor:measure_omega_pm_from_upper_bound} to prove the following proposition.
\begin{prop}\label{prop:pnt}
  We write
\begin{align*}
 \A_1&=\left\{x: \Delta(x)>(\lambda_2-\epsilon_0)x^{1/2}\right\}\\
 \text{and} \quad \A_2&=\left\{ x : \Delta(x)<(-\lambda_2+\epsilon_0)x^{1/2}\right\},
\end{align*}
for a fixed $\epsilon_0$ such that $0<\epsilon_0<\lambda_2$. Then
\[\mu\left(\A_j\cap[T, 2T]\right)=\Omega\left(T^{1-\epsilon}\right), \text{ for } j=1, 2 \ \text{ and for any } \epsilon>0.\]
\end{prop}
\begin{proof}
 Here we apply Corollary~\ref{cor:measure_omega_pm_from_upper_bound} in a similar way as in the previous application,
so we shall skip the details. 

The Riemann Hypothesis, Theorem~\ref{thm:kaczorowski} and Theorem~PNT** give
\[\mu\left(\A_j\cap[T, 2T]\right)=\Omega\left(\frac{T}{\log^4 T}\right) \text{ for } j=1, 2; \]
this implies the proposition. But if the Riemann Hypothesis is false, then there exists  a constant $\mathfrak{a}$, with $1/2<\mathfrak{a}\leq1$, such that
\[\mathfrak{a}=\sup\{\sigma:\zeta(\sigma+it)=0\}.\]
Using Perron summation formula, we may show that
\[\Delta(x)\ll x^{\mathfrak{a}+\epsilon},\]
for any $\epsilon>0$. Also for any arbitrarily small $\delta$, we have $\mathfrak{a}-\delta<\sigma'<\mathfrak{a}$ such that
$\zeta(\sigma'+it')=0$ for some real number $t'$. If $\lambda'':=|\sigma'+it'|^{-1}$, then by Corollary~\ref{cor:measure_omega_pm_from_upper_bound}
we get 
\begin{align*}
 \mu\left(\left\{x\in[T, 2T]:\Delta(x)>(\lambda''/2)x^{\sigma'}\right\}\right)&=\Omega\left(T^{1-2\delta-2\epsilon}\right)\\
 \text{ and } \quad \mu\left(\left\{x\in[T, 2T]:\Delta(x)<-(\lambda''/2)x^{\sigma'}\right\}\right)&=\Omega\left(T^{1-2\delta-2\epsilon}\right).
\end{align*}
As $\delta$ and $\epsilon$ are arbitrarily small and $\sigma'>1/2$, the above $\Omega$ bounds imply the proposition.
\end{proof}

\begin{rmk}
Results similar to Proposition~\ref{prop:pnt} can be obtained for error terms in asymptotic formulas for partial sums of Mobius function and for 
partial sums of the indicator 
function of square-free numbers.  
\end{rmk}
\begin{rmk}
In Section~\ref{subsec:sqfree_divisors}~and~\ref{subsec:pnt_error}, we saw that $\mu(\A_j)$ are large. Now suppose that 
$\mu(\A_1\cup\A_2)$ is large, then what can we say about the individual sizes of $\A_j$? We may guess that $\mu(\A_1)$ and $\mu(\A_2)$
are both large and almost equal. But this may be very difficult to prove. In the next chapter, we shall show that if $\mu(\A_1\cup\A_2)$ is large,
then both $\A_1$ and $\A_2$ are nonempty. 
\end{rmk}

%% file: influence_measure.tex
\chapter{Influence Of Measure}\label{chap:measure_analysis}
In this chapter, we study the influence of measure of the set where $\Omega$-result holds, on its possible improvements. The 
following proposition is an interesting application of the main theorem (Theorem~\ref{thm:omega_pm_measure}) of this chapter.

Let $\Delta(x)$ denotes the error term appearing in the assymptotic formula for average order of non-isomorphic abelian groups:
\begin{equation}\label{eq:non_isomorphic_ab_gr_delta}
\Delta(x)=\sum_{n\leq x}^*a_n - \sum_{k=1}^{6}\Big(\prod_{j \neq k}\zeta(j/k)\Big) x^{1/k},
\end{equation}
where $a_n$ denotes the number of non-isomophic abelian groups of order $n$.
One would expect that 
\[\Delta(x)=O\left(x^{1/6+\epsilon}\right)\ \text{ for any } \epsilon>0\]
(see Section~\ref{sec:sub_abelian_group} for more details), so an analogus $\Omega_\pm$ bound for $\Delta(x)$ is also 
expected. The proposition below gives a sufficient condition to obtain such an $\Omega_\pm$ bound.
\begin{prop}\label{prop:abelian_group}
 Let $\delta$ be such that $0<\delta<1/42$, and $\Delta(x)$ be as in (\ref{eq:non_isomorphic_ab_gr_delta}). Then 
 either
  \[\int_T^{2T}\Delta^4(x)\d x=\Omega( T^{5/3+\delta} ),\]
or
  \[\Delta(x)=\Omega_\pm(x^{1/6-\delta}).\]
\end{prop}
It may be conjectured that
\[\int_T^{2T}\Delta^4(x)\d x=O( T^{5/3+\epsilon} )\]
for any $\epsilon>0$. By the above proposition, this conjecture implies 
\[\Delta(x)=\Omega_\pm(x^{1/6-\epsilon}) \ \text{ for any } \ \epsilon>0.\]

We begin by assuming the conditions and notations given in Assumptions~\ref{as:for_continuation_mellintran}. Further we have the 
following notations for this chapter. 
\begin{notations}
For $i=0, 1, 2$, let $\alpha_i(T)$ denote positive monotonic functions such that 
$\alpha_i(T)$ converges to a constant as $T\rightarrow \infty$. For example, in some cases 
$\alpha_i(T)$ could be $1-1/\log(T)$, which tend to $1$ as $T\rightarrow \infty$. 

For $i=0, 1$, let $h_i(T)$ be positive monotonically increasing functions such that $h_i(T)\rightarrow\infty$
as $T\rightarrow\infty$.

For a real valued and non-negative function $f$, we denote 
\[\mathcal{A}(f(x)):=\{x\geq 1: |\Delta(x)|>f(x)\}.\]
\end{notations}

\section{Refining Omega Result from Measure}
Define an $\xset$ as follows.
\begin{defi}
An infinite unbounded subset $\set$ of non-negative real numbers is called an $\xset$ . 
\end{defi}
Now we hypothesize a situation when there is a lower bound estimate for the second moment of the error term. 
\begin{asump}\label{as:measure_to_omega}
 Let $\mathcal{S}$ be an $\xset$. Define a real valued positive bounded function $\alpha(T)$ on $\set$, such that
 \[0\leq \alpha(T)<M<\infty\] 
 for some constant $M$. For a fixed 
 $T$, we write
 \[\mathcal{A}_T:=[T/2, T]\cap\mathcal{A}(c_5 x^{\alpha(x)}) \ \text{ for } c_5>0.\] 
 For all $T\in \set$ and for constants $c_6, \ c_7 > 0$, assume the following bounds hold:
 \begin{enumerate}
  \item[(i)] \[\int_{\mathcal{A}_T}\frac{\Delta^2(x)}{x^{2\alpha+1}}\mathrm{d}x>c_6,\]
  \item[(ii)]
  \[\mu(\mathcal{A}_T)<c_7h_0(T),\quad \text{ and }\]
  \item[(iii)] the function 
  \[x^{\alpha+1/2}h_0^{-1/2}(x)\]
  is monotonically increasing for $x\in [T/2, T]$.
  
 \end{enumerate}
\end{asump}
We note that the first assumption indicates an $\Omega$-estimate. The next two assumptions indicate that the measure of 
the set on which the $\Omega$ estimate holds is not \lq too big\rq.

\begin{prop}\label{prop:refine_omega_from_measure}
 Suppose there exists an $\xset$ $\set$ for $\Delta(x)$, having properties as described in Assumptions~\ref{as:measure_to_omega}.
 Let the constant $c_8$ be given by
 \[c_8:=\sqrt{\frac{c_6}{2^{2M+1}c_7}}.\]
 Then there exists a $T_0$ such that for all $T>T_0$ and $T\in \set$, we have 
 \[|\Delta(x)|>c_8 x^{\alpha+1/2}h_0^{-1/2}(x)\]
for some $x\in [T/2, T]$.

In particular 
\[\Delta(x)=\Omega(x^{\alpha+1/2}h_0^{-1/2}(x)).\]
\end{prop}

\begin{proof}
 If the statement of the above proposition is not true, then for all $x\in [T/2, T]$ we have
 \[\Delta(x)\leq c_8 x^{\alpha + 1/2}h_0^{-1/2}(x).\] 
 From this, we may derive an upper bound for second moment of $\Delta(x)$:
 \begin{align*}
 \int_{\mathcal{A}_T}\frac{\Delta^2(x)}{x^{2\alpha+1}}\mathrm{d}x &\leq 
 \frac{c_8^2 T^{2\alpha+1}\mu(\mathcal{A}_T\cap[T/2, T])}{h_0(T)(T/2)^{2\alpha+1}} \\
 &\leq c_8^2 2^{2M + 1} c_7 \leq c_6.
 \end{align*}
This bound contradicts (i) of Assumptions~\ref{as:measure_to_omega}, which proves the proposition.
\end{proof}
The above proposition will be used in the next chapter to obtain a result on the error term appearing in the asymptotic formula
for $\sum_{n\leq x}^*|\tau(n, \theta)|^2$.
\section{Omega Plus-Minus Result from Measure}
 In this section, we prove an $\Omega_\pm$ result for $\Delta(x)$ when $\mu(A_T)$ is big. 
We formalize the conditions in the following assumptions.
\begin{asump}\label{as:measure_omega_plus_minus}
Suppose the conditions in Assumptions~\ref{as:for_continuation_mellintran} hold.
Let $l$ be an integer such that
\[l>\max(\sigma_2, 1),\]
and let $\alpha_1(T)$ be a monotonic function satisfying the inequality 
\[0<\alpha_1(T)\leq \sigma_1.\]
Furthermore: 
\begin{enumerate}
 \item[(i)] the Dirichlet series $D(\sigma+it)$ has no pole when $\alpha_1(T)\le \sigma\le \sigma_1$;
 \item[(ii)] if $|t|\leq T^{2l}$ and $\alpha_1(T)\le \sigma\le \sigma_1$, we have 
 \[|D(\sigma + it)|\leq c_9 (|t|+1)^{l-1}\]
 for some constant $c_9>0$.
\end{enumerate}
\end{asump}

\begin{asump}\label{as:measure_omega_plus_minus_weak}
Suppose there exists $\epsilon>0$ such that the following holds:
\begin{flushleft}
 if $D(\sigma+it)$ has no pole for $\alpha_1(T)-\epsilon< \sigma \le \sigma_1$ and $|t|\leq 2T^{2l}$, then 
there exists a constant $c_{10}>0$ depending on $\epsilon$ such that
  \[|D(\sigma + it)|\leq c_{10} (|t|+1)^{l-1}, \]
when $\alpha_1(T)\le \sigma \le \sigma_1$ and $|t|\leq T^{2l}$. 
\end{flushleft}
\end{asump}

Assumptions~\ref{as:measure_omega_plus_minus_weak} says that if $D(s)$ does not have pole in 
$\alpha_1(T)-\epsilon< \sigma \le \sigma_1,$ then it has polynomial growth in a certain region.

\begin{lem}\label{lem:perron_for_omegapm_measure}
Under the conditions in Assumptions~\ref{as:measure_omega_plus_minus}, we have
\[\Delta(x)
=\int_{\alpha_1-iT^{2l}}^{\alpha_1+iT^{2l}}\frac{D(\eta)x^\eta}{\eta}\d\eta + O(T^{-1}),\]
for all $x\in [T/2, 5T/2]$. 
\end{lem}
\begin{proof}Follows from Perron summation formula.
\end{proof}

\begin{lem}[Balasubramanian and Ramachandra \cite{BaluRamachandra2}]\label{lem:ramachandra_trick}
Let  $T\ge 1,$ $\delta_0>0$ and $f(x)$ be a real-valued integrable function such that
\[f(x)\geq0 \quad \text{ for } x\in [T-\delta_0T, \ 2T+ \delta_0T].\]
Then for  $\delta>0$ and for a positive integer $l$ satisfying
$\delta l\leq \delta_0,$
we have
\[\int_T^{2T}f(x)\d x \leq \frac{1}{(\delta T)^l}\mulint \int_{T-\sum_{1}^l y_i}^{2T+\sum_{1}^l y_i}f(x)\d{x}~\d y_1 \ldots \d y_l.\]
\end{lem}
\begin{proof}
For $0\le y_i \le \delta T$, $i=1,2,...,l$ 
\begin{align*}
\int_T^{2T}f(x)\d x \leq \int_{T-\sum_{1}^l y_i}^{2T+\sum_{1}^l y_i}f(x)\d{x}, 
\end{align*}
as $f(x)\ge 0$ in 
\[\left[T-\sum_{1}^l y_i, 2T+\sum_{1}^l y_i\right]\subseteq [T-\delta_0T, 2T+ \delta_0T].\]
This gives
\begin{align*}
 &\frac{1}{(\delta T)^l}\mulint \int_{T-\sum_{1}^l y_i}^{2T+\sum_{1}^l y_i}f(x)\d{x}~\d y_1 \ldots \d y_l \\
 & \geq \frac{1}{(\delta T)^l}\mulint \int_{T}^{2T}f(x)\d{x}~\d y_1 \ldots \d y_l
 = \int_{T}^{2T}f(x)\d{x}. 
\end{align*}

\end{proof}
The next theorem shows that if $\Delta(x)$ does not change sign then the set on which $\Omega$-estimate holds can not be \lq too big\rq.

\begin{thm}\label{thm:upper_bound_measure}
 Suppose the conditions in Assumptions~\ref{as:measure_omega_plus_minus} hold. Let $h_1(T)$ be a monotonically 
 increasing function such that $h_1(T)\rightarrow\infty$. Let $\alpha_2(T)$ be a bounded positive monotonic function, such that
 \begin{align*}
& 0<\alpha_1(T)<\alpha_2(T)\leq \sigma_1, \text{ and} \\
& \frac{h_1(T)}{T^{\alpha_1}}\rightarrow \infty \text{ as } T\rightarrow \infty.
\end{align*}
 If there exist a constant $x_0$ such that $\Delta(x)$ does not change sign on $\mathcal{A}(h_1(x))\cap[x_0, \infty)$, then 
 \[\mu(\mathcal{A}(x^{\alpha_2(x)})\cap [T, 2T])\leq 4h_1(5T/2)T^{1-\alpha_2(T)}+ O(1 + T^{1-\alpha_2(T) + \alpha_1(T)})\]
 for $T\geq 2x_0$.
\end{thm}
\begin{proof}
Trivially we have
\[\mu(\mathcal{A}(x^{\alpha_2})\cap [T, 2T])\leq \int_T^{2T}\frac{|\Delta(x)|}{x^{\alpha_2}}\d x .\]
Using Lemma~\ref{lem:ramachandra_trick} on the above inequality, we get
\[\mu(\mathcal{A}(x^{\alpha_2})\cap [T, 2T])\leq 
\frac{1}{(\delta T)^l}\mulint  \int_{T-\sum_{1}^l y_i}^{2T+\sum_{1}^l y_i}\frac{|\Delta(x)|}{x^{\alpha_2}}\d{x}~\d y_1 \ldots \d y_l,\]
where $\delta=\frac{1}{2l}.$

Let $\chi$ denote the characteristic function of the complement of $\mathcal{A}(h_1(x))$:
\[\chi(x)=\begin{cases}
           1 \quad \mbox{ if } x \notin  \mathcal{A}(h_1(x)),\\
           0 \quad \mbox{ if } x \in \mathcal{A}(h_1(x)).
          \end{cases}
\]
For $T\geq 2x_0$, $\Delta(x)$ does not change sign on
$$\left[T-\sum_{1}^l y_i, \ 2T+\sum_{1}^l y_i\right]\cap\A(h_1(x)),$$ 
as $0\le y_i\le \delta T$ for all $i=1,...,l$.
So we can write 
the above inequality as
\begin{align} \label{eq:measure_omega_pm_first_bound}
\notag
\mu(\mathcal{A}(x^{\alpha_2})\cap [T, 2T]) 
&\leq \frac{2}{(\delta T)^l}
\mulint  \int_{T-\sum_{1}^l y_i}^{2T+\sum_{1}^l y_i}\frac{|\Delta(x)|}{x^{\alpha_2}}\chi(x)\d{x}~\d y_1 \ldots \d y_l\\
&+ \frac{1}{(\delta T)^l} 
\left| \mulint  \int_{T-\sum_{1}^l y_i}^{2T+\sum_{1}^l y_i}\frac{\Delta(x)}{x^{\alpha_2}}\d{x}~\d y_1 \ldots \d y_l\right|.
\end{align}
Since $x\notin \A(h_1(x))$ implies $|\Delta(x)|\le h_1(x)$,  we get
\begin{align}\label{eq:measure_omega_pm_trivial_part}
\notag
 \frac{2}{(\delta T)^l}&
\mulint  \int_{T-\sum_{1}^l y_i}^{2T+\sum_{1}^l y_i}\frac{|\Delta(x)|}{x^{\alpha_2}}\chi(x)\d{x}~\d y_1 \ldots \d y_l \\
&\leq 4 h_1(5T/2) T^{1-\alpha_2}.
\end{align}
 We use the integral expression for $\Delta(x)$ as given in Lemma~\ref{lem:perron_for_omegapm_measure}, and get
 \begin{align}\label{eq:measure_omega_pm_perron_part}
 \notag
 & \frac{1}{(\delta T)^l} 
\left| \mulint  \int_{T-\sum_{1}^l y_i}^{2T+\sum_{1}^l y_i}\frac{\Delta(x)}{x^{\alpha_2}}\d{x}~\d y_1 \ldots \d y_l\right|\\
\notag
 &= \frac{1}{(\delta T)^l}\left| \mulint  \int_{T-\sum_{1}^l y_i}^{2T+\sum_{1}^l y_i}
 \int_{\alpha_1-iT^{2l}}^{\alpha_1+iT^{2l}}\frac{D(\eta)x^{\eta-\alpha_2}}{\eta}\d\eta~\d{x}~\d y_1 \ldots \d y_l\right| +O(1) \\
 \notag
 &\ll 1 + \frac{1}{(\delta T)^l}\left| \int_{\alpha_1-iT^{2l}}^{\alpha_1+iT^{2l}}\frac{D(\eta)}{\eta}
  \mulint  \int_{T-\sum_{1}^l y_i}^{2T+\sum_{1}^l y_i} x^{\eta-\alpha_2}\d{x}~\d y_1 \ldots \d y_l~\d\eta \right| \\
 \notag
&\ll 1 + \frac{1}{(\delta T)^l}\left| \int_{\alpha_1-iT^{2l}}^{\alpha_1+iT^{2l}} 
\frac{D(\eta)(2T+l\delta T)^{\eta-\alpha_2 + l + 1}}{\eta\prod_{j=1}^{l+1}(\eta-\alpha_2+j)}\d\eta\right|\\ \notag
&\ll 1 + \frac{T^{\alpha_1-\alpha_2+l + 1}}{(\delta T)^l}
 \int_{-T^{2l}}^{T^{2l}}\frac{(1+|t|)^{l-1}}{(1+ |t|)^{l+2}}\d t 
 \ll 1 + T^{1-\alpha_2+\alpha_1}.\\ 
\end{align}
The theorem follows from (\ref{eq:measure_omega_pm_first_bound}), (\ref{eq:measure_omega_pm_trivial_part}) 
and (\ref{eq:measure_omega_pm_perron_part}).
\end{proof}

\begin{thm}\label{thm:omega_pm_measure}
Consider $\alpha_1(T), \alpha_2(T), \sigma_1, h_1(T)$ as in Theorem~\ref{thm:upper_bound_measure},
and $\mathcal P$ as in Assumptions \ref{as:for_continuation_mellintran}.
Let $D(s)$ does not have a real pole in $[\alpha_1-\epsilon_0, \infty)-\mathcal{P}$, for some $\epsilon_0>0$. 
Suppose there exists an $\xset$ $\set$ such that for all $T\in \set$ 
 \[\mu(\mathcal{A}(x^{\alpha_2})\cap [T, 2T])> 5h_1(5T/2)T^{1-\alpha_2}.\]
 Then: 
  \begin{enumerate}
   \item[(i)]under Assumptions~\ref{as:measure_omega_plus_minus}, we have
  \[\Delta(x)=\Omega_\pm(h_1(x))\]
 ( In this case $\Delta(x)$ changes sign in $[T/2, 5T/2]\cap \mathcal{A}(h_1(x))$ for $T\in S$
  and $T$ is sufficiently large.);
   \item[(ii)]under Assumptions~\ref{as:measure_omega_plus_minus_weak}, we have
   \[\Delta(x)=\Omega_\pm(x^{\alpha_1-\epsilon}), \quad \text{for any } \ \epsilon>0.\]
 \end{enumerate}
\end{thm}
\begin{proof}
If the conditions in Assumptions~\ref{as:measure_omega_plus_minus} hold, then (i) follows from Theorem~\ref{thm:upper_bound_measure}. 
To prove (ii), choose an $\epsilon$ such that $0<\epsilon<\epsilon_0$.
Now suppose 
 $\eta_0$ is a pole of $D$ for $\Re(\eta)\geq\alpha_1(T)-\epsilon$ and $t\leq 2T^{2l}$, then by Theorem~\ref{thm:landu_omegapm}
  \[\Delta(x)=\Omega_\pm(T^{\alpha_1-\epsilon}).\]
If there are no poles in the above described region of $\sigma + it$, then we are in the set-up of 
Assumptions~\ref{as:measure_omega_plus_minus}, and get 
\[\Delta(x)=\Omega_\pm(h_1(x)).\]
We have
\[T^{\alpha_1(T)}=o(h_1(T)),\]
which gives
\[\Delta(x)=\Omega_\pm(x^{\alpha_1-\epsilon}).\] 
This completes the proof of (ii). 
\end{proof}

\section{Applications}\vspace{0.05mm}
Now we shall see two examples demonstrating applications of Theorem~\ref{thm:omega_pm_measure}.

\subsection{Error term of the divisor function} 
Let $d(n)$ denote the number of divisors of $n$:
\[d(n)=\sum_{d|n}1.\]
Dirichlet \cite[Theorem~320]{HardyWright} showed that
\[\sum_{n\leq x}^*\tau(n) = x\log(x) + (2\gamma -1)x + \Delta(x), \]
where $\gamma$ is the Euler constant and 
\[\Delta(x)=O(\sqrt{x}).\]
Latest result on $\Delta(x)$ is due to Huxley \cite{HuxleyDivisorProblem}, which is
\[\Delta(x)=O(x^{131/416}).\]
On the other hand, Hardy \cite{HardyDirichletDivisor} showed that 
\begin{align*}
 \Delta(x)&=\Omega_+((x\log x)^{1/4}\log\log x),\\
 &=\Omega_-(x^{1/4}).
\end{align*}
There are many improvements of Hardy's result. Some notable results are due to K. Corr{\'a}di and I. K{\'a}tai \cite{CorradiKatai},
J. L. Hafner \cite{Hafner}, and K. Sounderarajan \cite{Sound}. Below, we shall show that $\Delta(x)$ is $\Omega_\pm(x^{1/4})$ as a consequence of Theorem~\ref{thm:omega_pm_measure} and 
results of Ivi{\'c} and Tsang ( see below ).
Moreover, we shall how that such fluctuations occur in $[T, 2T]$ for every sufficiently large $T$.

\noindent Ivi{\'c} \cite{Ivic_second_moment_divisor_problem} proved that for a positive constant $c_{11}$,
\[\int_{T}^{2T}\Delta^2(x)\d x \sim c_{11} T^{3/2}.\]
A similar result for fourth moment of $\Delta(x)$ was proved by Tsang \cite{Tsang}:
\[\int_T^{2T}\Delta^4(x)\d x\sim c_{12} T^2,\]
for a positive constant $c_{12}$.
Let $\A$ denote the following set: 
\[\A:=\left\{x:|\Delta(x)|>\frac{c_{11} x^{1/4}}{6}\right\}.\]
For sufficiently large $T$, using the result of Ivi{\'c} \cite{Ivic_second_moment_divisor_problem}, we get
\begin{align*}
 \int_{[T, 2T]\cap \A} \frac{\Delta^2(x)}{x^{3/2}}\d x &=\int_T^{2T}\frac{\Delta(x)^2}{x^{3/2}}\d x 
 -\int_{[T, 2T]\cap\A^c}\frac{\Delta^2(x)}{x^{3/2}}\d x\\
& \geq \frac{1}{4T^{3/2}}\int_T^{2T}\Delta^2(x)\d x -\frac{c_{11}}{6} \\
& \geq \frac{c_{11}}{5} -\frac{c_{11}}{6} \geq \frac{c_{11}}{30}.
\end{align*}
 Using Cauchy-Schwarz inequality and the result due to Tsang \cite{Tsang} we get
\begin{align*}
\int_{[T, 2T]\cap \mathcal{A}} \frac{\Delta^2(x)}{x^{3/2}}\d x &\leq \left(\int_{[T, 2T]\cap\A}\frac{\Delta^4(x)}{x^2}\d x\right)^{1/2}
\left(\int_{[T, 2T]\cap\A}\frac{1}{x}\d x\right)^{1/2}\\
& \leq \left(\frac{c_{12}\mu([T, 2T]\cap \A)}{T}\right)^{1/2}. 
\end{align*}
The above lower and upper bounds on second moment of $\Delta$ gives the following lower bound for measure of $\A$:
\begin{equation*}
\mu([T, 2T]\cap \A)> \frac{c_{11}^2}{901 c_{12}}T,
\end{equation*}
for some $T\geq T_0$.
Now, Theorem ~\ref{thm:omega_pm_measure} applies with the following choices:
\[\alpha_1(T)=1/5, \quad \alpha_2(T)=1/4, \quad h_1(T)=\frac{c_{11}^2}{9000c_{12}}T^{1/4}.\]  
Finally using Theorem~\ref{thm:omega_pm_measure}, we get that for all $T\geq T_0$ there exists $x_1, x_2 \in [T, 2T]$ such that
\begin{align*}
 \Delta(x_1)> h_1(x_1) \ \text{ and } \ \Delta(x_2)< - h_1(x_2).
\end{align*}
In particular we get 
\begin{align*}
\Delta(x)=\Omega_\pm(x^{1/4}).
\end{align*}

\subsection{Average order of Non-Isomorphic abelian Groups}\label{sec:sub_abelian_group}

Let $a_n$ denote the number of non-isomorphic abelian groups of order $n$. 
The Dirichlet series $D(s)$ is given by 
\[D(s)=\sum_{n=1}^{\infty}\frac{a_n}{n^s} =\prod_{k=1}^{\infty}\zeta(ks),\quad \Re(s)>1.\]
The meromorphic continuation of $D(s)$ has poles at $1/k$, for all positive integer $k\geq 1$. Let the 
main term $\M(x)$ be
\[\M(x)=\sum_{k=1}^{6}\Big( \prod_{j\neq k} \zeta(j/k) \Big)x^{1/k},\]
and the error term $\Delta(x)$ be
\[\sum_{n\leq x}^* a_n - \M(x).\]
Balasubramanian and Ramachandra \cite{BaluRamachandra2} proved that
\begin{equation*}
 \int_T^{2T}\Delta^2(x)\d x=\Omega(T^{4/3}\log T),
\text{ and }
\Delta(x)=\Omega_{\pm}(x^{92/1221}).
\end{equation*}
Sankaranarayanan and Srinivas \cite{srini} improved the $\Omega_\pm$ bound to
\[ \Delta(x)=\Omega_{\pm}\left(x^{1/10}\exp\left(c\sqrt{\log x}\right)\right)\]
for some constant $c>0$.
An upper bound for the second moment of $\Delta(x)$ was first given by Ivi{\'c} \cite{Ivic_abelian_group}, 
and then improved 
by Heath-Brown \cite{Brown} to
\[\int_T^{2T}\Delta^2(x)\d x\ll T^{4/3}(\log T)^{89}.\]
This bound of Heath-Brown is best possible in terms of power of $T$. 
But for the fourth moment, the similar statement
 \[\int_T^{2T}\Delta^4(x)\d x\ll T^{5/3}(\log T)^C,\]
which is best possible in terms of power of $T$, is an open problem. 
Another open problem is to show that 
\[\Delta(x)=\Omega_\pm(x^{1/6-\delta})\ \text{ for any }\ \delta>0.\]
For $0<\delta<1/42$, we have stated in Proposition~\ref{prop:abelian_group} that either
  \[\int_T^{2T}\Delta^4(x)\d x=\Omega( T^{5/3+\delta} )\ \text{ or }\ \Delta(x)=\Omega_\pm(x^{1/6-\delta}).\]
Below, we present a proof of this proposition.
\begin{proof}[Proof of Proposition~\ref{prop:abelian_group}]
 If the first statement is false, then we have 
\[\int_T^{2T}\Delta^4(x)\d x\leq c_{13} T^{5/3+\delta}, \]
 for some constant $c_{13}$ depending on $\delta$ and for all $T\geq T_0$.
 Let $\A$ be defined by:
 \[\A=\{x: |\Delta(x)|>c_{14}x^{1/6}\}, \quad c_{14}>0.\]
  By the result of Balasubramanian and Ramachandra \cite{BaluRamachandra2}, we have an $\xset$ $\set$, such that
 \[\int_{[T, 2T]\cap\A}\Delta^2(x)\d x \geq c_{15}T^{4/3}(\log T)\]
 for $T\in S$. Using Cauchy-Schwartz inequality, we get
 \begin{align*}
  c_{15}T^{4/3}(\log T)&\leq \int_{[T, 2T]\cap\A}\Delta^2(x)\d x 
  \leq \left(\int_T^{2T}\Delta^4(x)\d x\right)^{1/2}(\mu(\A\cap[T, 2T]))^{1/2}\\
  &\leq c_{13}^{1/2}T^{5/6+\delta/2}(\mu(\A\cap[T, 2T]))^{1/2}.
 \end{align*}
This gives, for a suitable positive constant $c_{16}$,
\[\mu(\A\cap[T, 2T])\geq c_{16}T^{1-\delta}(\log T)^2.\]
Now we use Theorem~\ref{thm:omega_pm_measure}, (i), with
\[\alpha_2=\frac{1}{6}, \quad \alpha_1=\frac{13}{84}-\frac{\delta}{2}, \quad \mbox{ and } \quad h_1(T)=T^{1/6-\delta}.\]
So we get 
\[\Delta(x)=\Omega_\pm(x^{1/6-\delta}).\]
This completes the proof.
\end{proof}

%% file: twisted_divisor_new.tex
\chapter{The Twisted Divisor Function}\label{chap:twisted_divisor}

Recall that in Chapter~\ref{chap:intro}, we have defined the twisted divisor function $\tau(n, \theta)$
as follows: 
\[\tau(n, \theta)=\sum_{d | n} d^{i\theta}, \quad \text{for } \ \theta\in \mathbb R -\{0\},\  n\in \mathbb N.\]
We also have stated the following asymptotic formula:
\begin{equation*}
 \sum_{n\leq x}^*|\tau(n, \theta)|^2=\omega_1(\theta)x\log x + \omega_2(\theta)x\cos(\theta\log x)
+\omega_3(\theta)x + \Delta(x),
\end{equation*}
where $\omega_i(\theta)$s are explicit constants depending only on $\theta$ and
\[\Delta(x)=O_\theta(x^{1/2}\log^6x).\]
In this chapter, we give a proof of this formula (see Section~\ref{sec:asymptotic_formula}, Theorem~\ref{thm:asymp_formula_tau_n_theta}).
In Section~\ref{sec:twisted_divisor_application_landau}, we use Theorem~\ref{thm:omega_pm_main} to obtain some measure
theoretic $\Omega_\pm$ results. Further, we obtain an $\Omega$ bound for the second moment of $\Delta(x)$ in Section~\ref{sec:stronger_omega}
by adopting a technique due to Balasubramanian, Ramachandra and Subbarao \cite{BaluRamachandraSubbarao}. In the final section, we prove 
that if the $\Omega$ bound obtained in the previous section can not be improved, then
\[\Delta(x)=\Omega(x^{3/8-\epsilon}) \ \text{ for  any }\ \epsilon>0.\]

Now we motivate with a brief note on few applications of $\tau(n, \theta)$.

\section{Applications of $\tau(n, \theta)$}\label{sec:applications_tau_n_theta}
The function $\tau(n, \theta)$ can be used to study various properties related to the distribution of divisors of an integer: 
\[\sum_{\substack{d|n \\ a\leq \log d \leq b}}^*1=\frac{1}{2\pi}\int_{-\infty}^{\infty}\tau(n, \theta)
\frac{e^{-ib\theta}-e^{-ia\theta}}{-i\theta};\]
here $\sum^*$ means that the corresponding contribution to the sum is $\half$ if $e^a|n$ or $e^b|n$.
Below we present two applications.

\subsection{Clustering of Divisors}
The following function measures the clustering of divisors of an integer:
\[W(n, f):=\sum_{d, d'|n} f(\log(d/d')),\]
for some constant $c>0$ and for a function $f\in L^1(\mathbb R)$. We assume that $f$ has a Fourier transformation, say $\hat f$,
and $\hat f \in L^1(\mathbb{R})$.
\begin{prop}\label{prop:clustering}
With the above notations:
\[\sum_{n\le x} W(n, f) = \frac{1}{2\pi} \int_{-\infty}^{\infty} \hat{f}( \theta)\sum_{n\le x} |\tau(n,\theta)|^2 \d\theta.\]
\end{prop}
\begin{proof}
Note that by the Fourier inversion formula, we get
\begin{align*}
W(n, f)&= \sum_{d, d'|n} f(\log(d/d'))= \frac{1}{2\pi}\sum_{d, d'|n} \int_{-\infty}^{\infty} \hat{f}(\theta)\left(\frac{d}{d'}\right)^{i\theta}\d \theta\\
&= \frac{1}{2\pi}\int_{-\infty}^{\infty} \hat{f}(\theta) \left( \sum_{d, d'|n} \left(\frac{d}{d'}\right)^{i\theta} \right)\d \theta
= \frac{1}{2\pi}\int_{-\infty}^{\infty} \hat{f}(\theta)|\tau(n, \theta)|^2\d \theta.
\end{align*}
This implies the proposition.
\end{proof}
Using Proposition~\ref{prop:clustering} and the formula in (\ref{eq:formmula_tau_ntheta}), we may write
\begin{align*}
\sum_{n\le x} W(n, f) &= \frac{x\log x}{2\pi}  \int_{-\infty}^{\infty} \hat{f}( \theta) \omega_1(\theta)\d\theta
+ \frac{x}{2\pi}  \int_{-\infty}^{\infty} \hat{f}( \theta)\big(\omega_2(\theta)\cos(\theta\log x)+\omega_3(\theta)\big)\d\theta \\
&+ \frac{x}{2\pi}  \int_{-\infty}^{\infty} \hat{f}( \theta)\Delta(x, \theta)\d\theta. \\
&\text{(In the above identity, we denoted $\Delta(x)$ by $\Delta(x, \theta)$.)}
\end{align*}
This gives that the function $\sum_{n\leq x}W(n, f)$ behaves like $x\log x$. Further, if we want to obtain more information on $\sum_{n\leq x}W(n, f)$,
we may analyzing other terms in the above formula. But now, we skip the details and refer to \cite[Chapter~4]{DivisorsHallTenen}.

\subsection{The Multiplication Table Problem}
The multiplication table problem asks for an estimate on the order of the growth of $|\text{Mul}(N)|$ as $N\rightarrow\infty$, where
\[\text{Mul}(N):=\{1\leq m \leq N^2:  m= ab, \ a, b \in \mathbb{Z}\text{ and }1\leq a, b \leq N\}.\]
The initial attempts in this direction
are due to Erd{\H{o}}s \cite{ErdosMultTab}. He used a result of Hardy and Ramanujan \cite{HardyRamanujan} (also see \cite{MyExpo}) to show 
\[|\text{Mul}(N)|\ll \frac{N^2}{(\log N)^{\nu_0}\sqrt{\log\log N}}\ \text{as} \ N\rightarrow\infty,\]
and here
\[\nu_0=1-\frac{1+\log\log 2}{\log2}.\]
Intuitively, the theorem of Hardy and Ramanujan says that most of the positive integers less than $x$ have around $\log\log x$ prime factors;
more precisely
\[\#\{n\leq x: |\omega(n)-\log\log n|<\sqrt{\log\log n}\}\sim x \ \text{as} \ x\rightarrow\infty.\]
This gives that most of the positive integers less than $N^2$ have around $\log\log N$ prime factors, whereas most of
the integers in the multiplication table 
have around $2\log\log N$ prime factors, which implies $|\text{Mul}(N)|$ is $o(N^2)$. Erd{\H{o}}s used a refined version of this argument 
to obtain the given upper bound for $|\text{Mul}(N)|$. 

The best known bound on the asymptotic growth of $|\text{Mul}(n)|$ is due to Ford \cite{FordAnn}:
\[|\text{Mul}(N)|\asymp \frac{N^2}{(\log N)^{\nu_0}(\log\log N)^{3/2}},\ \text{ as } N\rightarrow\infty.\]
To obtain the expected lower bound for $|\text{Mul}(N)|$, Ford first proved that
\[|\text{Mul}(N)|\gg \frac{N^2}{(\log N)^2}\sum_{n\leq N^{1/8}}\frac{L(n)}{n},\
\text{ where } \ L(n):=\mu\left(\cup_{d|n}[\log(d/2), \log d]\right).\]
We may also observe that
\[\sum_{n\leq N^{1/8}}\frac{L(n)}{n}\geq\frac{\left(\sum_{n\leq N^{1/8}}\frac{d(n)}{n}\right)^2}{6\sum_{n\leq N^{1/8}}\frac{W(n)}{n}}.\]
Rest of the part in Ford's argument deals with the above sums involving the divisor function $d(n)$ and 
$W(n):=W\left(n, 1_{[\half, 2]}\right)$, where $1_{[\half, 2]}$ is the indicator function of the interval $[\half, 2]$. 
We skip the details and refer to \cite{Fordy2y}.

\section{Asymptotic Formula for $\sum_{n\leq x}^*|\tau(n, \theta)|^2$}\label{sec:asymptotic_formula}
In this section, we shall prove the following asymptotic formula for $\sum_{n\leq x}^*|\tau(n, \theta)|^2$.
\begin{thm}[Theorem~33, \cite{DivisorsHallTenen}]\label{thm:asymp_formula_tau_n_theta}
 Let $\theta\neq 0$ be a fixed real number. Then for $x\geq 1$, we have
 \[\sum_{n\leq x}^*|\tau(n, \theta)|^2=\omega_1(\theta)x\log x + \omega_2(\theta)x\cos(\theta\log x)
+\omega_3(\theta)x + O_\theta(x^{1/2}\log^6x)\,\]
where $\omega_i(\theta)$s are explicit constants depending only on $\theta$. 
\end{thm}
\begin{proof}
Recall that the corresponding Dirichlet series $D(s)$ has the following meromorphic continuation:
\[D(s)=\sum_{1}^\infty\frac{|\tau(n, \theta)|^2}{n^s}=\frac{\zeta^2(s)\zeta(s+i\theta)\zeta(s-i\theta)}{\zeta(2s)},\ \text{ for } s>1.\]
For $x\geq 2$, we denote $\kappa=1+\frac{1}{\log x}$ and $T=x+|\theta|+1$. By Perron's formula
\[\sum_{n\leq x}^*|\tau(n, \theta)|^2=\frac{1}{2\pi i}\int_{\kappa-iT}^{\kappa+iT}D(s)x^s\frac{\d s}{s} + O(x^{\epsilon}).\]
After shifting the line of integration to $\Re(s)=\half$, we may estimate the contributions from horizontal lines as follows:
\begin{align*}
 T^{-1}\int_{\half}^{1}|D(\sigma \pm iT)|x^{\sigma}\d \sigma \ll T^{-1}\int_{\half}^{1}T^{1-\sigma+\epsilon}x^{\sigma}\d \sigma \ll x^{\epsilon}.
\end{align*}
To obtain an asymptotic formula for $\sum_{n\leq x}^*|\tau(n, \theta)|^2$, we add up the residues from the poles $1, 1\pm i\theta$
after shifting the line of integration to $\Re(s)=\half$:
\begin{align*}
\sum_{n\leq x}^*|\tau(n, \theta)|^2 = \M(x)  + 
O\left(x^\epsilon + x^\half\int_{-T}^{T}\left|\frac{\zeta^2(\half+it)\zeta(\half+i(t+\theta))\zeta(\half+i(t-\theta))}{\zeta(1+2it)(\half + it)}\right|\d t\right),
\end{align*}
where 
\[\M(x)=\omega_1(\theta)x\log x + \omega_2(\theta)x\cos(\theta\log x)+\omega_3(\theta)x.\]
If we write
\[\mathcal J(\mathfrak{a}, T):=\int_{-T}^{T}\frac{\zeta^4(\half+i(\mathfrak{a}+t))}{\sqrt{t^2+\quater}}\d t \quad \text{ for } \ \mathfrak{a}, T
\in \mathbb R \ \text{ and } \ T\geq 1\ ,\]
then we have \cite[Theorem~5.1]{IvicBook}
\begin{equation}\label{eq:upper_bound_4thmoment_zeta}
 \mathcal J(\mathfrak a, T)\ll_{\mathfrak a} \log^5 T.\
\end{equation}
To express $\Delta(x)$ in terms of $\mathcal J(\mathfrak a, T)$, observe that 
\begin{align*}
\Delta(x)&=\sum_{n\leq x}^* |\tau(n, \theta)|^2 -\M(x)\\
&\ll x^\epsilon + x^\half\int_{-T}^{T}\left|\frac{\zeta^2(\half+it)\zeta(\half+i(t+\theta))\zeta(\half+i(t-\theta))}{\zeta(1+i2t)(\half + it)}\right|\d t \\
&\ll x^\epsilon + x^\half\log x \int_{-T}^{T}|\zeta^2(\half+it)\zeta(\half+i(t+\theta))\zeta(\half+i(t-\theta))|\frac{\d t}{|\half + it|} .\\
\end{align*}
From (\ref{eq:upper_bound_4thmoment_zeta}) and using the Cauchy-Schwartz inequality twice, we get
\begin{align*}
\Delta(x)\ll x^\epsilon + x^\half\log x \mathcal J^\half(0, x) \mathcal J^\quater(\theta, x) \mathcal J^\quater(-\theta, x) 
\ll_{\theta} x^\half \log ^6 x,
\end{align*}
which gives the required result.
\end{proof}
In the following sections, we shall obtain various $\Omega$ and $\Omega_\pm$ bounds for $\Delta(x)$.

\section{Oscillations of the Error Term}\label{sec:twisted_divisor_application_landau}
Here we shall apply results in Chapter~\ref{chap:landau_theorem} to $\Delta(x)$ and obtain some measure theoretic $\Omega_\pm$ results.
We begin by defining a contour $\Co$ as given in Figure~\ref{fg:contour_c0_taun}:
\begin{align*}
 \Co=&\left(\frac{5}{4}-i\infty, \frac{5}{4}-i(\theta+1)\right]\cup \left[\frac{5}{4}-i(\theta+1), \frac{3}{4}-i(\theta+1)\right]\\
 &\cup \left[\frac{3}{4}-i(\theta+1), \frac{3}{4}+i(\theta+1)\right]\cup \left[\frac{3}{4}+i(\theta+1), \frac{5}{4}+i(\theta+1)\right]\\
 &\cup \left[\frac{5}{4}+i(\theta+1), \frac{5}{4}+i\infty \right).
\end{align*}
From Theorem~\ref{thm:perron_formula}, we have
\[\Delta(x)=\int_{\Co}\frac{D(\eta)x^\eta}{\eta}\d \eta.\]
The above identity expresses the Mellin transform $A(s)$ of $\Delta(x)$ as a contour integral involving $D(s)$.
Using Theorem~\ref{thm:analytic_continuation_mellin_transform}, we write
\[A(s)=\int_{1}^{\infty}\frac{\Delta(x)}{x^{s+1}}\d x=\int_{\Co}\frac{D(\eta)}{\eta(s-\eta)}\d \eta,\]
when $s$ lies right to the contour $\Co$. Denote the first nontrivial zero of $\zeta(s)$ with least positive imaginary part by $2s_0$. 
An approximate value of this point is
\[2s_0=\frac{1}{2}+i14.134\ldots .\]
Define the contour $\Co(s_0)$, as in Figure~\ref{fg:contour_c0s0}, such that $s_0$ and any real number $s\geq1/4$ lie in the 
right side of this contour. A meromorphic continuation of $A(s)$ to all $s$ that lies right side of $\Co(s_0)$ is given by
\begin{equation}\label{eq:analytic_cont1/4}
 A(s)=\int_{\Co(s_0)}\frac{D(\eta)x^\eta}{\eta}{\d \eta} + \frac{\underset{\eta=s_0}{\res}D(\eta)}{s_0(s-s_0)}.
 \end{equation}
\begin{center}
 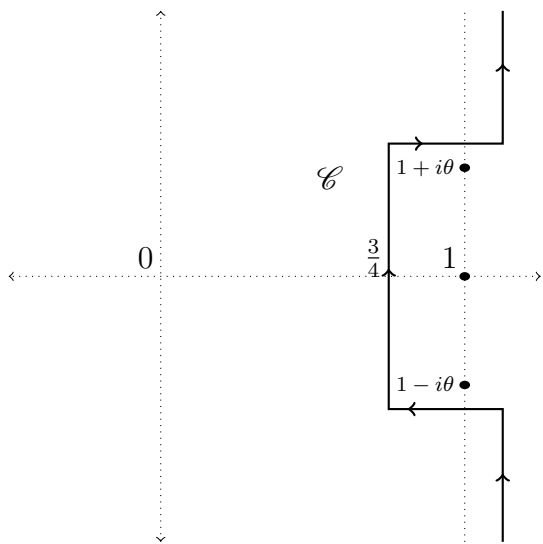
\begin{figure}
 \begin{tikzpicture}[yscale=0.8]
\draw [<->][dotted] (-1, -4.4)--(-1, 4.4);
\node at (-1.2, 0.3) {$0$};
\draw [<->][dotted] (4, 0)--(-3, 0);

\draw [dotted] (3, -4.4)--(3, 4.4);
\fill (3, 0) circle[radius=2pt];
\node at (2.8, 0.3) {$1$};

\fill (3, 1.8) circle[radius=2pt];
\node at (2.48, 1.8) {\scriptsize{$1+i\theta$}};
\fill (3, -1.8) circle[radius=2pt];
\node at (2.48, -1.8) {\scriptsize{$1-i\theta$}};
\node at (1.8, 0.3) {$\frac{3}{4}$};

\draw [thick] [postaction={decorate, decoration={ markings,
mark= between positions 0.09 and 0.98 step 0.21 with {\arrow[line width=1pt]{>},}}}]
(3.5, -4.4)--(3.5, -2.2)--(2, -2.2)--(2, 2.2)--(3.5, 2.2)--(3.5, 4.4);
\node [below left] at (1.55, 2) {$\mathscr{C}$};
\end{tikzpicture}
\caption{Contour $\mathscr{C}$, for $D(s)=\sum_{n=1}^{\infty}\frac{|\tau(n, \theta)|^2}{n^s}$.}\label{fg:contour_c0_taun}
\end{figure}
\end{center}

\begin{center}
 \begin{figure}
 \begin{tikzpicture}[yscale=0.8]
\draw [<->][dotted] (-1, -4.4)--(-1, 4.4);
\node at (-1.2, 0.3) {$0$};
\draw [<->][dotted] (4, 0)--(-3, 0);

\draw [dotted] (3, -4.4)--(3, 4.4);
\fill (3, 0) circle[radius=2pt];
\node at (2.8, 0.3) {$1$};

\fill (3, 1.8) circle[radius=2pt];
\node at (2.48, 1.8) {\scriptsize{$1+i\theta$}};
\fill (3, -1.8) circle[radius=2pt];
\node at (2.48, -1.8) {\scriptsize{$1-i\theta$}};
\draw [dotted] (0, -4.4)--(0, 4.4);
\node at (-0.28, 0.3) {$\frac{1}{4}$};
\fill (0, 3) circle[radius=2pt];
\node at (-0.28, 3) {$s_0$};

\draw [thick] [postaction={decorate, decoration={ markings,
mark= between positions 0.03 and 0.99 step 0.147 with {\arrow[line width=1pt]{>},}}}]
(3.5, -4.4)--(3.5, -2.5)--(2, -2.5)--(2, -0.5)--(-0.6, -0.5)--(-0.6, 3.8)--(0.4, 3.8)--(0.4, 0.65)
--(2, 0.65)--(2, 2.5)--(3.5, 2.5)--(3.5, 4.4);
\node [below left] at (-0.55, 2) {$\mathscr{C}(s_0)$};
\end{tikzpicture}
\caption{Contour $\mathscr{C}(s_0)$}\label{fg:contour_c0s0}
\end{figure}
\end{center}

From (\ref{eq:analytic_cont1/4}), we calculate the following two limits:
\begin{equation}\label{eqn:lambda_theta}
\lambda(\theta):= \lim_{\sigma\searrow0}\sigma|A(\sigma+s_0)| = |s_0|^{-1} \left|\underset{\eta=s_0}{\res}D(\eta)\right|>0
\end{equation}
and 
\begin{equation*}
\lim_{\sigma\searrow0}\sigma A(\sigma+ 1/4)=0.
\end{equation*}
For a fixed small enough $\epsilon>0$, define
\begin{align*}
 \A_1&=\left\{x: \Delta(x)>(\lambda(\theta)-\epsilon)x^{1/4}\right\},\\
 \A_2&=\left\{ x : \Delta(x)<(-\lambda(\theta)+\epsilon)x^{1/4}\right\}.
\end{align*}
Theorem~\ref{thm:asymp_formula_tau_n_theta} and Corollary~\ref{cor:measure_omega_pm_from_upper_bound} give
\begin{equation}\label{result:tau_n_theta_omegapm}
 \mu\left(\A_j\cap[T, 2T]\right)=\Omega\left(T^{1/2}(\log T)^{-12}\right) \text{ for } j=1, 2.
\end{equation}
Under Riemann Hypothesis, Theorem~\ref{thm:omega_pm_main} and Proposition~\ref{prop:upper_bound_second_moment_twisted_divisor} give 
\begin{equation}\label{result:tau_n_theta_omegapm_underRH}
 \mu\left(\A_j\cap[T, 2T]\right)=\Omega\left(T^{3/4-\epsilon}\right) \text{ for } j=1, 2.
\end{equation}
From Corollary~\ref{coro:omega_pm_secondmoment}, we get
\begin{equation}\label{result:tau_n_theta_secondmoment}
 \int_{\A_j\cap[T, 2T]}\Delta^2(x)\d x = \Omega\left(T^{3/2}\right) \text{ for } j=1, 2.
\end{equation}

\section{An Omega Theorem }\label{sec:stronger_omega}
%
%
%
%
%
Recall that (see Theorem~\ref{thm:asymp_formula_tau_n_theta})
$$\sum_{n\le x} |\tau(n,\theta)|^2 =\M(x) + \Delta(x),$$
where the main term
\[\M(x)=\omega_1(\theta)x\log x + \omega_2(\theta)x\cos(\theta\log x)+\omega_3(\theta)x\] 
comes from the poles of $D(s)$ at $s=1, 1+i\theta$ and 
$s=1-i\theta$. 
We may observe from Corollary~\ref{coro:omega_pm_secondmoment} that if $D(s)$ has a complex pole at $s_0=\sigma_0+it_0$, other than $1+i\theta$
and $1-i\theta$, then
\[ \int_T^{2T}\Delta(x)\d x=\Omega(x^{2\sigma_0+1}).\]
By Riemann Hypothesis, the only positive value for $\sigma_0$ is $\quater$, which is same as (\ref{result:tau_n_theta_secondmoment}). 
In this section, 
we shall use a technique due to  Balasubramanian, Ramachandra and Subbarao \cite{BaluRamachandraSubbarao} to improve this omega bound further.
Now we state the main theorem of this section.
\begin{thm}\label{omega_integral}
For any $c>0$ and for a sufficiently large $T$ depending on $c$, we get 
\begin{equation}\label{lb-increasing}
\int_T^{\infty} \frac{|\Delta(x)|^2}{x^{2\alpha+1}}e^{-2x/y} \d u 
\gg_c\exp\left( c(\log T)^{7/8} \right),
\end{equation}
where
\[ \alpha=\alpha(T)=\frac{3}{8} -\frac{c}{(\log T)^{1/8} }.\]
\end{thm}
\noindent 
In particular, this implies
\[ \Delta(x)=\Omega (x^{3/8}\exp(-c(\log x)^{7/8}), \]
for some suitable $c>0$.

In order to prove the theorem, we need several lemmas, which form the content of this section. 
We begin with a fixed $\delta_0 \in (0,1/16]$ for which we would choose a numerical value at the end of this section.

\begin{defi}
 For $T>1$, let $Z(T)$ be the set of all $\gamma$ such that 
\begin{enumerate}
\item $T\le \gamma \le 2T$,
\item either $\zeta(\beta_1+i\gamma)=0$ for some $\beta_1\ge \half+\delta_0$ \\
or $\zeta(\beta_2+i 2\gamma)=0$ for some $\beta_2\ge \half +\delta_0$.
\end{enumerate}
Let
\[ I_{\gamma,k} = \{ T\le t \le 2T: |t-\gamma| \le k\log^2 T \}  \text{ for }  k=1, 2.\]
We finally define
\[J_k(T)=[T,2T] \setminus \cup_{\gamma\in Z(T)} I_{\gamma,k}. \]
\end{defi}

\begin{lem}\label{size-J(T)}
With the above definition, we have for $k=1,2$
\[ \mu(J_k(T)) = T +O\left( T^{1-\delta_0/4} \log^3 T \right). \]
\end{lem}

\begin{proof}
We shall use an estimate on the function $N(\sigma, T)$, which is defined as 
\[N(\sigma, T):=\left|\{\sigma'+it:\sigma'\ge\sigma,\ 0<t\leq T,\ \zeta(\sigma'+it)=0\}\right|.\]
Selberg \cite[Page~237]{Titchmarsh} proved that
$$N(\sigma, T) \ll T^{1-\frac{1}{4}(\sigma -\half)} \log T, \ \text{ for } \ \sigma>1/2.$$
Now the lemma follows from the above upper bound on $N(\sigma, t)$, and the observation that
$$|\cup_{\gamma\in Z(T)} I_{\gamma,k}| \ll N\left(\half+ \delta_0, T\right)\log^2 T.$$
\end{proof}

The next lemma closely follows Theorem 14.2 of  \cite{Titchmarsh}, but does not depend on Riemann Hypothesis.
\begin{lem}\label{estimate-on-J(T)}
For $t\in J_1(T)$ and $\sigma= 1/2+\delta$ with $\delta_0 < \delta < 1/4-{\delta_0}/2$,  we have
$$|\zeta(\sigma+it)|^{\pm 1} \ll 
\exp\left(\log\log t \left(\frac{\log t}{\delta_0}\right)^{\frac{1-2\delta}{1-2\delta_0}}\right)$$

and
\[ |\zeta(\sigma+2it)|^{\pm 1} \ll
\exp\left(\log\log t \left(\frac{\log t}{\delta_0}\right)^{\frac{1-2\delta}{1-2\delta_0}}\right).\]
\end{lem}

\begin{proof}
We provide a proof of the first statement, and the second statement can be similarly proved.

Let $1 <\sigma' \le \log t$. We consider two concentric circles centered at $\sigma'+it$,
with radius $\sigma'-1/2-\delta_0/2$ and $\sigma' -1/2- \delta_0$. 
Since $t\in J_1(T)$ and the radius of the circle is $\ll \log t$, we conclude that
\[ \zeta(z)\neq 0 \ \text{ for } \  |z-\sigma'-it | \le \sigma' - \half -  \frac{\delta_0}{2} \]
and also $\zeta(z)$ has polynomial growth in this region. 
Thus on the larger circle, $\log |\zeta(z)| \le c_{17}\log t$, for some constant $c_{17}>0$.
By Borel-Carath\'eodory theorem, 
\[\  |z-\sigma'-it | \le \sigma' - \half - \delta_0 \  \text{ implies } \
|\log \zeta(z)| \le \frac{c_{18}\sigma'}{\delta_0} \log t, \]
for some $c_{18}>0$.
Let $1/2+\delta_0< \sigma < 1$, and $\xi>0$ be such that $1+\xi< \sigma'$. 
We consider three concentric circles centered at 
$\sigma'+it$ with radius $r_1=\sigma'-1-\xi$, $r_2=\sigma'-\sigma$ and 
$r_3=\sigma'-1/2-\delta_0$,
and call them $\mathcal  C_1, \mathcal C_2$ and $\mathcal C_3$ respectively. 
Let
$$M_i = \sup_{z\in \mathcal C_i} |\log \zeta(z)|.$$
From the above bound on $|\log\zeta(z)|$, we get 
$$M_3 \le  \frac{c_{18}\sigma'}{\delta_0} \log t.$$
Suitably enlarging $c_{18}$, we see that 
\[ M_1 \le \frac{c_{18}}{\xi}.\]
Hence we can apply the Hadamard's three circle theorem to conclude that
\[ M_2 \le M_1^{1-\nu} M_3^\nu,  \ \text{ for } \ \nu=\frac{\log(r_2/r_1)}{\log(r_3/r_1)}. \]
Thus
\[ M_2 \le \left( \frac{c_{18}}{\xi} \right)^{1-\nu}\left(\frac{c_{18}\sigma' \log t}{\delta_0}\right)^\nu. \]
It is easy to see that
\[ \nu=2-2\sigma + \frac{4\delta_0(1-\sigma)}{1+2\xi-2\delta_0} 
+O(\xi) + O\left( \frac{1}{\sigma'}\right). \]
Now we put 
\[ \xi=\frac{1}{\sigma'}=\frac{1}{\log\log t} .\]
Hence
\[ M_2 \le \frac{c_{18}  \log^\nu t \log\log t}{\delta_0^\nu}
=\frac{c_{19} \log\log t}{\delta_0^\nu}(\log t)^{2-2\sigma+\frac{4\delta_0(1-\sigma)}{1+2\xi-2\delta_0} }, \]
for some $c_{19}>0$.
We observe that
\[  2-2\sigma+\frac{4\delta_0(1-\sigma)}{1+2\xi-2\delta_0} < 
2-2\sigma+\frac{4\delta_0(1-\sigma)}{1-2\delta_0} =\frac{1-2\delta}{1-2\delta_0}. \]
So we get 
\[ |\log \zeta(\sigma +it) | 
\le c_{19} \log\log t \left(\frac{\log t}{\delta_0}\right)^{\frac{1-2\delta}{1-2\delta_0}},\]
and hence the lemma.
\end{proof}

We put $y=T^{\mathfrak b}$, for a constant $\mathfrak b \ge 8$. Now suppose that
$$\int_T^{\infty} \frac{|\Delta(u)|^2}{u^{2\alpha+1}}e^{-u/y}du \ge \log^2 T,$$
for sufficiently large $T$. Then clearly 
$$\Delta(u) =\Omega( u^{\alpha}) .$$ 
Our next result explores the situation when such an inequality does not hold.

\begin{prop}\label{main-prop}
Let $\delta_0<\delta<\quater-\frac{\delta_0}{2}$.
For $1/4+\delta/2 < \alpha <1/2$, suppose that
 \begin{equation}\label{assumption}
\int_T^{\infty} \frac{|\Delta(u)|^2}{u^{2\alpha+1}}e^{-u/y}du \le \log^2 T
\end{equation}
for any sufficiently large $T$.
Then we have
$$\int_{\substack{Re(s)=\alpha \\ t\in J_2(T)}} \frac{|D(s)|^2}{|s|^2} 
\ll 
1 + \int_T^{\infty} \frac{|\Delta(u)|^2}{u^{2\alpha+1}}e^{-2u/y} du.$$
\end{prop}

Before embarking on a proof, we need the following technical lemmas.

\begin{lem}\label{gamma}
For $0\le \Re(z) \le 1$ and $|Im(z)|\ge \log^2T$, we have
\begin{equation}\label{gamma1}
\int_T^{\infty} e^{-u/y}u^{-z} du =\frac{T^{1-z}}{1-z} + O(T^{-\mathfrak b'})
\end{equation}
and
\begin{equation}\label{gamma2}
\int_T^{\infty} e^{-u/y} u^{-z}\log u\  du =\frac{T^{1-z}}{1-z}\log T + O(T^{-\mathfrak b'}),
\end{equation}
where $\mathfrak b'>0$ depends only on $\mathfrak b$.
\end{lem}

\begin{proof}

By changing variable by $v= u/y$, we get
\[ \int_T^{\infty} \frac{e^{-u/y}}{u^z} du = y^{1-z} \int_{T/y}^{\infty} e^{-v} v^{-z} dv. \]
Integrating the right hand side by parts
\[ \int_{T/y}^{\infty} e^{-v} v^{-z} dv = \frac{e^{-T/y}}{1-z}\left( \frac{T}{y} \right)^{1-z}
 + \frac{1}{1-z} \int_{T/y}^{\infty} e^{-v} v^{1-z} dv \]
It is easy to see that
 \[ \int_{T/y}^{\infty} e^{-v} v^{1-z} dv = \Gamma(2-z) + O\left( \left(\frac{T}{y} \right)^{2-Re(z)}\right). \]
Hence (\ref{gamma1})  follows using $e^{-T/y} =1+O(T/y)$ and Stirling's formula along with the assumption
that $ |Im(z)|\ge \log^2T$. 

Proof of (\ref{gamma2}) proceeds in the same line and uses the fact that
 \[ \int_{T/y}^{\infty} e^{-v} v^{1-z}\log v\ dv = \Gamma'(2-z) 
+ O\left( \left(\frac{T}{y} \right)^{2-Re(z)}\log T\right). \]
Then  we apply Stirling's formula for $\Gamma'(s)$ instead of $\Gamma(s)$.
\end{proof}

\begin{lem}\label{initial-estimates}
Under the assumption (\ref{assumption}), there exists $T_0$ with $T\le T_0 \le 2T$ such that
\begin{equation*}
\frac{\Delta(T_0)e^{-T_0/y}}{T_0^{\alpha}} \ll \log^2 T,
\end{equation*}
\begin{equation*}
\text{and}\quad\frac{1}{y}\int_{T_0}^{\infty}\frac{\Delta(u)e^{-u/y}}{u^{\alpha}} du \ll \log T.
\end{equation*}

\end{lem}

\begin{proof}

The assumption (\ref{assumption}) implies that
\begin{eqnarray*}
\log^2T &\ge & \int_T^{2T} \frac{|\Delta(u)|^2}{u^{2\alpha+1}}e^{-u/y}du \\
&=& \int_T^{2T} \frac{|\Delta(u)|^2}{u^{2\alpha}}e^{-2u/y}\frac{e^{u/y}}{u} du \\
&\ge & \min_{T\le u\le 2T}\left(\frac{|\Delta(u)|}{u^{\alpha}}e^{-u/y}\right)^2, 
\end{eqnarray*}
which proves the first assertion.
To prove the second assertion, we use the previous assertion and Cauchy- Schwartz inequality along with assumption (\ref{assumption}) to get
\begin{eqnarray*}
\left( \int_{T_0}^{\infty}\frac{|\Delta(u)|^2}{u^{\alpha}}e^{-u/y}du \right)^2
&\le & \left( \int_{T_0}^{\infty}\frac{|\Delta(u)|^2}{u^{2\alpha+1}}e^{-u/y}du \right)
\left(  \int_{T_0}^{\infty} u e^{-u/y}du \right) \\
&\ll & y^2 \log^2 T.
\end{eqnarray*}
This completes the proof of this lemma.
\end{proof}

We now recall a mean value theorem due to Montgomery and Vaughan \cite{MontgomeryVaughan}.
\begin{notation}
 For a real number $\theta$, let $\|\theta\|:=\min_{n\in \mathbb Z}|\theta -n|.$
\end{notation}

\begin{thm}[Montgomery and Vaughan \cite{MontgomeryVaughan}]\label{mean-value}
Let $a_1,\cdots, a_N$ be arbitrary complex numbers, and let $\lambda_1,\cdots,\lambda_N$ be distinct real numbers such that 
\[\delta = \min_{m,n}\| \lambda_m-\lambda_n\|>0.\]
Then
\[ \int_0^T \left| \sum_{n\le N} a_n \exp(i\lambda_n t) \right|^2 dt 
=\left(T +O\left(\frac{1}{\delta}\right)\right)\sum_{n\le N} |a_n|^2.\]
\end{thm}

\begin{lem}\label{mean-value-estimate}
For $T\le T_0\le 2T$ and $\Re(s)=\alpha$, we have
$$\int_T^{2T} \left| \sum_{n\le T_0}\frac{|\tau(n,\theta)|^2}{n^s}e^{-n/y}\right|^2 t^{-2} dt \ll 1.$$
\end{lem}

\begin{proof}
Using  theorem \ref{mean-value}, we get
\begin{eqnarray*}
\int_T^{2T} \left| \sum_{n\le T_0}\frac{|\tau(n,\theta)|^2}{n^s}e^{-n/y}\right|^2 t^{-2} \d t 
&\le&  \frac{1}{T^2} \left( T \sum_{n\le T_0}  |b(n)|^2 
+ O\left( \sum_{n\le T_0} n|b(n)|^2\right)\right),\\
&&\text{where } \quad b(n)=\frac{|\tau(n,\theta)|^2}{n^{\alpha}}e^{-n/y}.
\end{eqnarray*}
Thus
\[ \sum_{n\le T_0}  |b(n)|^2 \le \sum_{n\le T_0}\frac{d(n)^4}{n^{2\alpha}}
\ll T_0^{1-2\alpha+\epsilon}\
\text{ and } 
\ \sum_{n\le T_0}  n|b(n)|^2 \le \sum_{n\le T_0}\frac{d(n)^4}{n^{2\alpha-1}}
\ll T_0^{2-2\alpha + \epsilon}\]
for any $\epsilon>0$, since the divisor function $d(n)\ll n^\epsilon$.  As we have $\alpha>0$, this completes the proof.
\end{proof}

\begin{lem}\label{mean-value-error}
For $\Re(s)=\alpha$ and $T\le T_0 \le 2T$, we have
$$\int_T^{2T} 
\left| \sum_{n\ge 0}\int_0^1 \frac{\Delta(n+x+T_0) 
e^{-(n+x+T_0)/y}}{(n+x+T_0)^{s+1}} \d x \right|^2 \d t
\ll \int_T^{\infty} \frac{|\Delta(x)|^2}{x^{2\alpha+1}}e^{-2x/y} \d x.$$
\end{lem}

\begin{proof}  
Using Cauchy- Schwarz inequality, we get
\begin{align*}
 & \left| \sum_{n\ge 0}\int_0^1  \frac{\Delta(n+x+T_0)}{(n+x+T_0)^{s+1}}
e^{-(n+x+T_0)/y} \d x \right|^2 \\
\le& \int_0^1 \left| \sum_{n\ge 0} \frac{\Delta(n+x+T_0)}{(n+x+T_0)^{s+1}}e^{-(n+x+T_0)/y} \right|^2 \d x.
\end{align*}
Hence 
\begin{align*}
&\int_T^{2T} 
\left| \int_0^1 \sum_{n\ge 0}\frac{\Delta(n+x+T_0) e^{-(n+x+T_0)/y}}{(n+x+T_0)^{s+1}} \d x \right|^2 \d t \\
\le&  \int_T^{2T}\int_0^1 \left| \sum_{n\ge 0} \frac{\Delta(n+x+T_0)}{(n+x+T_0)^{s+1}}
e^{-(n+x+T_0)/y} \right|^2 \d x \d t \\
=& \int_0^1 \int_T^{2T}\left| \sum_{n\ge 0} \frac{\Delta(n+x+T_0)}{(n+x+T_0)^{s+1}}
e^{-(n+x+T_0)/y} \right|^2 \d t \d x.
\end{align*}
From Theorem \ref{mean-value}, we can get
\begin{eqnarray*}
&&\int_T^{2T}\left| \sum_{n\ge 0} \frac{\Delta(n+x+T_0)}{(n+x+T_0)^{s+1}}e^{-(n+x+T_0)/y} \right|^2 \d t\\
&=& T\sum_{n\ge 0}\frac{ |\Delta(n+x+T_0)|^2}{(n+x+T_0)^{2\alpha+2}}
e^{-2(n+x+T_0)/y} 
+ O\left( \sum_{n\ge 0} \frac{ |\Delta(n+x+T_0)|^2}{(n+x+T_0)^{2\alpha+1}}\right)\\
&\ll & \sum_{n\ge 0} \frac{ |\Delta(n+x+T_0)|^2}{(n+x+T_0)^{2\alpha+1}}e^{-2(n+x+T_0)/y}.
\end{eqnarray*}
Hence
\begin{eqnarray*}
&&\int_T^{2T} 
\left| \sum_{n\ge 0}\int_0^1 \frac{\Delta(n+x+T_0) e^{-(n+x+T_0)/T}}{(n+x+T_0)^{s+1}} \d x \right|^2 \d t \\
&\ll & \int_0^1 \sum_{n\ge 0} \frac{ |\Delta(n+x+T_0)|^2}{(n+x+T_0)^{2\alpha+1}}e^{-2(n+x+T_0)/y}\d x
\ll  \int_T^{\infty} \frac{|\Delta(x)|^2}{x^{2\alpha+1}}e^{-2x/y} \d x,
\end{eqnarray*}
completing the proof.
\end{proof}

\begin{proof}[\textbf{Proof of Proposition \ref{main-prop}.} ]
For $s=\alpha+it$ with $1/4 +\delta /2 < \alpha < 1/2$ and $t\in J_2(T)$, we have
\begin{eqnarray*}
\sum_{n=1}^\infty \frac{|\tau(n,\theta)|^2}{n^s} e^{-n/y}
&=& \frac{1}{2\pi i} \int_{2-i\infty}^{2+i\infty}
D(s+w) \Gamma(w) y^w \d w \\
&=& \frac{1}{2\pi i} \int_{2-i\log^2T}^{2+i\log^2T}
+O\left( y^2\int_{\log^2T}^{\infty} |D(s+w)||\Gamma(w)|\d w \right).
\end{eqnarray*}
The above error term is estimated to be $o(1)$. We move the integral to 
$$\left[\frac{1}{4}+\frac{\delta}{2} -\alpha-i\log^2T, 
\ \frac{1}{4}+\frac{\delta}{2} -\alpha+i\log^2 T\right].$$
Let $\delta'=1/4+\delta/2 -\alpha$.
In this region $\Re(2s+2w)=1/2+\delta$ . So we can apply 
Lemma~\ref{estimate-on-J(T)} to 
conclude that
$D(s+w) =O(T^\kappa)$, for some constant $\kappa>0$. Thus the integrals along horizontal lines are 
$o(1)$.
Since the only pole inside this contour is at $w=0$, we get
\begin{eqnarray*}
\sum_{n=1}^{\infty} \frac{|\tau(n,\theta)|^2}{n^s} e^{-n/y} 
= D(s) + \frac{1}{2\pi i}
\int_{\delta'-i\log^2 T}^{\delta'+i\log^2 T}
D(s+w)\Gamma(w)y^w \d w +o(1).
\end{eqnarray*}
Since $\delta' <0$, the remaining integral can be shown to be $o(1)$ for $\mathfrak b\ge 8$.
Using $T_0$ as in Lemma~\ref{initial-estimates}, we now divide the sum into two parts:
$$D(s)= \sum_{n\le T_0} \frac{|\tau(n, \theta)|^2}{n^s} e^{-n/y} 
+ \sum_{n > T_0} \frac{|\tau(n, \theta)|^2 }{n^s} e^{-n/y}+o(1).$$
To estimate the second sum, we write
\begin{eqnarray*}
\sum_{n > T_0} \frac{|\tau(n, \theta)|^2}{n^s} e^{-n/y} 
&=& \int_{T_0}^{\infty} \frac{ e^{-x/y}}{x^s} \d \left(\sum_{n\le x}|\tau(n, \theta)|^2 \right)\\
 &=& \int_{T_0}^{\infty} \frac{ e^{-x/y}}{x^s} \d (\M(x)+\Delta(x))\\
&=&  \int_{T_0}^{\infty} \frac{ e^{-x/y}}{x^s} \M'(x)\d x 
   + \int_{T_0}^{\infty} \frac{ e^{-x/y}}{x^s} \d (\Delta(x)).
\end{eqnarray*}
Recall that
\[ \M(x)=\omega_1(\theta)x\log x + \omega_2(\theta)x\cos(\theta\log x)
+\omega_3(\theta)x,\]
thus 
\[ \M'(x)=\omega_1(\theta)\log x + \omega_2(\theta)\cos(\theta\log x)
-\theta\omega_2(\theta)\sin(\theta\log x)+\omega_1(\theta)+\omega_3(\theta).\]
 Observe that
\[ \int_{T_0}^{\infty}\frac{ e^{-x/y}}{x^s}\cos(\theta\log x) \d x
=\half \int_{T_0}^{\infty} \frac{ e^{-x/y}}{x^{s+i\theta}} \d x
+\half \int_{T_0}^{\infty} \frac{ e^{-x/y}}{x^{s-i\theta}} \d x.\]
Applying Lemma~\ref{gamma}, we conclude that 
\[ \int_{T_0}^{\infty} \frac{ e^{-x/y}}{x^s} \M'(x)\d x = o(1).\]
Integrating the second integral by parts:
\begin{eqnarray*}
\int_{T_0}^{\infty} \frac{ e^{-x/y}}{x^s} \d (\Delta(x)) 
&=& \frac{e^{-T_0/y} \Delta(T_0)}{T_0^s} \\
&+& \frac{1}{y}\int_{T_0}^\infty \frac{ e^{-x/y}}{x^s}\Delta(x)  \d x
-s\int_{T_0}^{\infty}\frac{ e^{-x/y}}{x^{s+1}}\Delta(x) \d x.
\end{eqnarray*}
Applying Lemma~\ref{initial-estimates}, we get
\begin{eqnarray*}
\sum_{n > T_0} \frac{|\tau(n, \theta)|^2}{n^s} e^{-n/y} &=&
s\int_{T_0}^{\infty}\frac{\Delta(x) e^{-x/y}}{x^{s+1}} \d x  + O(\log T) \\
&=& s\sum_{n\ge 0} \int_{0}^{1}\frac{\Delta(n+x+T_0) e^{-(n+x+T_0)/y}}{(n+x+T_0)^{s+1}} \d x +O(\log T).
\end{eqnarray*}
Hence we have
$$D(s)= \sum_{n\le T_0} \frac{|\tau(n, \theta)|^2}{n^s} e^{-n/y} 
+s\sum_{n\ge 0} \int_{0}^{1}\frac{\Delta(n+x+T_0) e^{-(n+x+T_0)/y}}{(n+x+T_0)^{s+1}} \d x +O(\log T) .$$
Squaring both sides, and then integrating on $J_2(T)$, we get
\begin{align*}
\int_{J_2(T)} \frac{|D(\alpha+it)|^2}{|\alpha+it|^2} \d t 
& \ll \int_T^{2T} \left|  \sum_{n\le T_0} \frac{|\tau(n, \theta)|^2}{n^s} 
  e^{-n/y} \right|^2 \frac{ \d t}{t^2} \\
& + \int_T^{2T} 
\left| \sum_{n\ge 0} \int_{0}^{1}\frac{\Delta(n+x+T_0) e^{-(n+x+T_0)/y}}{(n+x+T_0)^{s+1}} \d x \right|^2 \d t.
\end{align*}
The proposition now follows using Lemma~\ref{mean-value-estimate} and Lemma~\ref{mean-value-error}.
\end{proof}
We are now ready to prove our main theorem of this section.
\begin{proof}[\textbf{Proof of Theorem \ref{omega_integral}}]
We prove by contradiction. Suppose that (\ref{lb-increasing}) does not hold.
Then, given any $N_0>1$, there exists $T>N_0$ such that
\[\int_T^{\infty} \frac{|\Delta(x)|^2}{x^{2\alpha+1}}e^{-2x/y} \d x 
\ll \exp\left( c(\log T)^{7/8} \right),\]
for all $c>0$.
This gives
\[ \int_T^{\infty} \frac{|\Delta(x)|^2}{x^{2\beta+1}}e^{-2x/y} \d x \ll 1, \]
where
\[
\beta=\frac{3}{8}-\frac{c}{2(\log T)^{1/8}} .
\]
We apply Proposition~\ref{main-prop} to get 
\begin{equation}\label{contra}
\int_{J_2(T)} \frac{|D(\beta+it)|^2}{|\beta+it|^2} \d t 
\ll 1.
\end{equation}
Now we compute a lower bound for the last integral over $J_2(T)$.
Write the functional equation for $\zeta(s)$ as
$$\zeta(s) =\pi^{1/2-s}\frac{\Gamma((1-s)/2)}{\Gamma(s/2)}\zeta(1-s).$$
Using the Stirling's formula for $\Gamma$ function, we get
$$|\zeta(s)|=\pi^{1/2-\beta}t^{1/2-\beta}|\zeta(1-s)|\left(1+O\left(\frac{1}{T}\right)\right),$$
for $s=\beta+it$. This implies
$$|D(\beta+it)|=t^{2-4\beta}\frac{|\zeta(1-\beta+it)^2\zeta(1-\beta-it-i\theta)
\zeta(1-\beta-it+i\theta)|}{|\zeta(2\beta+i2t)|}.$$
Let $\delta_0=1/16$, and 
\[\beta=\frac{3}{8} -\frac{c}{2(\log T)^{1/8} }=\half-\delta \]
with 
\[\delta=\frac{1}{8}+\frac{c}{2(\log T)^{1/8} }.\]
Then using Lemma~\ref{estimate-on-J(T)}, we get
\[ |\zeta(1-\beta+it)| = \left|\zeta\left(\half+\delta+it\right)\right|
\gg \exp\left(\log\log t \left(\frac{\log t}{\delta_0}\right)^{\frac{1-2\delta}{1-2\delta_0}}\right).\]
For $t\in J_2(T)$ we observe that $t\pm\theta \in J_1(T)$, and so the same
bounds hold for $\zeta(1-\beta+it+i\theta)$ and $\zeta(1-\beta+it -i\theta)$. 
Further
\[ |\zeta(2\beta+i2t)| = \left|\zeta\left(\half+\left(\half-2\delta\right)+i2t\right)\right|
\ll \exp\left(\log\log t \left(\frac{\log t}{\delta_0}\right)^{\frac{4\delta}{1-2\delta_0}}\right).\]
Combining these bounds, we get
\[ |D(\beta+it)| \gg t^{2-4\beta}
\exp\left(-5\log\log t \left(\frac{\log t}{\delta_0}\right)^{\frac{1-2\delta}{1-2\delta_0}}\right).\]
Therefore
\begin{eqnarray*}
\int_{J_2(T)} |D(\beta+it)|^2 \d t 
&\gg & \int_{J_2(T)} |D(\beta+it)|^2 \d t 
\gg  \int_{J_2(T)} |D(\beta+it)|^2 \d t\\
&\gg & T^{4-8\beta}
 \exp\left(-10\log\log T\left(\frac{\log T}{\delta_0}\right)^{\frac{1-2\delta}{1-2\delta_0}}\right)
\mu(J_2(T)) \\
&\gg & T^{5-8\beta}\exp\left(-10\log\log T \left(\frac{\log T}{\delta_0}\right)^{\frac{1-2\delta}{1-2\delta_0}}\right),
\end{eqnarray*}
where we use Lemma  \ref{size-J(T)} to show that $\mu(J_2(T))\gg T$.
Now putting the values of  $\delta$ and $\delta_0$ as chosen above, we get
$$\int_{J(T)} \frac{|D(\beta+it)|^2}{|\beta+it|^2} dt 
\gg \exp\left(3c(\log T)^{7/8}\right),$$
since $\frac{1-2\delta}{1-2\delta_0}< 7/8$. This contradicts (\ref{contra}), and hence the theorem follows.
\end{proof}

The following two corollaries are immediate.
\begin{coro}\label{coro:balu_ramachandra1}
For any $c>0$ there exists an $\xset$ $\set$, such that for sufficiently large $T$ depending on $c$
there exists an 
\[ X \in \left[ T, \frac{T^{\mathfrak b}}{2}\log^2 T\right]\cap \set, \] for which we have
\[ \int_X^{2X} \frac{|\Delta(x)|^2}{x^{2\alpha+1}} dx \ge \exp\left( (c-\epsilon)(\log X)^{7/8}\right)\]
with $\alpha$ as in Theorem \ref{omega_integral} and for any $\epsilon>0$.
\end{coro}

\begin{coro}\label{coro:balu_ramachandra2}
For any $c>0$ there exists an $\xset$ $\set$, such that for sufficiently large $T$ depending on $c$
there exists an 
\[ x \in \left[ T, \frac{T^{\mathfrak b}}{2}\log^2 T\right]\cap \set, \] for which we have
\[ \Delta(x) \ge  x^{3/8} \exp\left( - c(\log x)^{7/8}\right). \]
\end{coro}
\noindent We can now prove a "measure version" of the above result.
\begin{prop}\label{Balu-Ramachandra-measure}
For any $c>0$, let
\[\alpha(x)=\frac{3}{8}- \frac{c}{(\log x)^{1/8} }\] and  
$\A=\{x: |\Delta(x)|\gg x^{\alpha(x)} \}$. 
Then for every sufficiently large $X$ depending on $c$, we have
\[ \mu(\A\cap [X,2X])=\Omega(X^{2\alpha(X)}).\]
\end{prop}

\begin{proof} 
Suppose that the conclusion does not hold, hence
\[  \mu(\A\cap [X,2X]) \ll X^{2\alpha(X)}.\] 
Thus for every sufficiently large $X$, we get
\[ \int_{\A\cap [X,2X] }\frac{|\Delta(x)|^2}{x^{2\alpha+1}}dx 
\ll X^{2\alpha}\frac{ M(X)}{X^{2\alpha+1}}=\frac{ M(X)}{X},\]
where $\alpha=\alpha(X)$ and $M(X)=\sup_{X\le x \le 2X} |\Delta(x)|^2$.
Using dyadic partition, we can prove
\[ \int_{\A\cap [T,y] }\frac{|\Delta(x)|^2}{x^{2\alpha+1}}dx \ll \frac{M_0(T)}{T}\log T, \
\text{ where }
\ M_0(T) =\sup_{T\le x\le y} |\Delta(x)|^2 \]
and $y=T^{\mathfrak b}$ for some $\mathfrak b>0$ and $T$ sufficiently large. 
This gives
\[ \int_T^{\infty} \frac{|\Delta(x)|^2}{x^{2\alpha+1}}e^{-2x/y} dx
\ll \frac{M_0(T)}{T}\log T. \]
Along with (\ref{lb-increasing}), this implies
\[ M_0(T)\gg T\exp\left( \frac{c}{2}(\log T)^{7/8} \right).\]
Thus
\[|\Delta(x)| \gg x^{\half} \exp\left(\frac{c}{4}(\log x)^{7/8}\right),\]
for some $x\in [T,y].$
This contradicts the fact that $|\Delta(x)| \ll x^{\half} (\log x)^6.$
\end{proof}
\subsection{Optimality of the Omega Bound }
The following proposition shows the optimality of the omega bound in Corollary~\ref{coro:balu_ramachandra1}.
\begin{prop}\label{prop:upper_bound_second_moment_twisted_divisor}
 Under Riemann Hypothesis (RH), we have
 \[\int_{X}^{2X}\Delta^2(x)\d x\ll X^{7/4+\epsilon},\]
for any $\epsilon>0$. 
\end{prop}
\begin{proof}
 Theorem~\ref{thm:effective_perron_formula} (Perron's formula) gives
\begin{equation*}
\Delta(x)=\frac{1}{2\pi}\int_{-T}^{T}\frac{D(3/8+it)x^{3/8+it}}{3/8+it}\d t + O(x^\epsilon),
\end{equation*}
for any $\epsilon>0$ and for $T=X^2$ with $x\in[X, 2X]$. Using this expression for $\Delta(x)$, we write its second moment as
\begin{align*}
&\int_{X}^{2X}\Delta^2(x)\d x 
= \int_{X}^{2X}\int_{-T}^{T}\int_{-T}^{T}\frac{D(3/8+ it_1)D(3/8+ it_2)}{(3/8+it_1)(3/8+ it_2)}x^{3/4+ i(t_1+t_2)}\d x \ \d t_1 \d t_2 \\
&\hspace{2.5 cm}  
+ O\left(X^{1+\epsilon}(1+|\Delta(x)|)\right)\\
&\ll X^{7/4}\int_{-T}^{T}\int_{-T}^{T}\left|\frac{D(3/8+ it_1)D(3/8+ it_2)}{(3/8+it_1)(3/8+it_2)(7/4+ i(t_1+t_2))}\right|\d t_1 \d t_2 
+ O(X^{3/2+\epsilon}).\\
\end{align*}
In the above calculation, we have used the fact that $\Delta(x)\ll x^{\half+\epsilon}$ as in (\ref{eq:upper_bound_delta}). Also note that for 
complex numbers $a, b$, we have $|ab|\leq \half(|a|^2 + |b|^2)$. We use this inequality with
\[a=\frac{|D(3/8+it_1)|}{|3/8+it_1|\sqrt{|7/4+i(t_1+t_2)|}} \ \text{ and } 
\ b=\frac{|D(3/8+it_2)|}{|3/8+it_2|\sqrt{|7/4+i(t_1+t_2)|}},\]
to get
\begin{align*}
\int_{X}^{2X}\Delta^2(x)\d x
&\ll X^{7/4}\int_{-T}^{T}\int_{-T}^{T}\left|\frac{D(3/8+ it_2)}{(3/8+it_2)}\right|^2\frac{\d t_1}{|7/4+ i(t_1+t_2)|} \d t_2 + O(X^{3/2+\epsilon})\\
&\ll X^{7/4}\log X\int_{-T}^{T}\left|\frac{D(3/8+ it_2)}{(3/8+it_2)}\right|^2 \d t_2 + O(X^{3/2+\epsilon}).
\end{align*}
Under RH, $|D(3/8+it_2)|\ll |t_2|^{\half+\epsilon}$. So we have 
\begin{align*}
 \int_{X}^{2X}\Delta^2(x)\d x 
 \ll X^{7/4+\epsilon} \ \text{ for any } \ \epsilon>0.
\end{align*}
\end{proof}
\noindent 
\textbf{Note.}
The method we have used in Theorem~\ref{omega_integral} has its origin from the  Plancherel's formula in Fourier analysis.
For instance, we may observe from Theorem~\ref{thm:perron_formula} that under Riemann Hypothesis and other suitable conditions
\[\frac{\Delta(e^u)}{e^{u\sigma}}=\frac{1}{2\pi}\int_{-\infty}^{\infty}\frac{D(\sigma+it)e^{iut}}{\sigma + it}\d t \ 
\text{ for } \quater<\sigma\le\half.\]
So $\frac{\Delta(e^u)}{e^{u\sigma}}$ is the Fourier transform of $\frac{D(\sigma+it)}{\sigma+it}$. By Plancherel's formula
\[\int_{-\infty}^{\infty}\frac{|\Delta(e^u)|}{e^{2u\sigma}}\d u
=\frac{1}{4\pi^2}\int_{-\infty}^{\infty}\left|\frac{D(\sigma+it)}{\sigma+it}\right|^2\d t.\]
Now we change the variable $u$ to $\log x$ and use the functional equation for $\zeta(s)$ to get
\begin{align*}
\int_1^\infty\frac{\Delta^2(x)}{x^{2\sigma+1}}\d x \asymp \int_1^{\infty}\left|\frac{D(\sigma+it)}{\sigma+it}\right|^2\d t
\gg \int_1^{\infty} t^{2-8\sigma-16\epsilon}
\end{align*}
for any $\epsilon>0$. We may choose $\sigma=\frac{3}{8}-\epsilon$, then the above integral on the left side is convergent. But if 
$\Delta(x)\ll x^{\frac{3}{8}-\epsilon}$, then the integral in the right diverges. This gives
\[\Delta(x)=\Omega(x^{\frac{3}{8}-\epsilon}).\]
In \cite{BaluRamachandra1} and \cite{BaluRamachandra2}, Balasubramanian and Ramachandra used this insight to obtain 
$\Omega$ bounds for
the error terms in asymptotic formulas for partial sums of 
square-free divisors and counting function for non-isomorphic abelian groups. 
This method requires the Riemann Hypothesis to be assumed in certain cases. Later Balasubramanian, Ramachandra and Subbarao \cite{BaluRamachandraSubbarao} 
modified this technique to apply on error term in the asymptotic formula for the counting function of $k$-full numbers 
without assuming Riemann Hypothesis. This method has been used by several authors including \cite{Nowak} and \cite{srini}.

\section{Influence of Measure on $\Omega_\pm$ Results}
In this section, we shall show that for any $\epsilon>0$,
$$\text{ if } \ \Delta(x)\ll x^{3/8+ \epsilon},
  \ \text{ then } \
\Delta(x)=\Omega_\pm\left(x^{3/8-\epsilon}\right).$$
This improves our earlier result, which says that $\Delta(x)$ is $\Omega_\pm\left(x^{1/4}\right)$.
Now, we state the main theorem of this section.
\begin{thm}\label{thm:tau_theta_omega_pm}
Let $\Delta(x)$ be the error term of the summatory function of the twisted divisor function as in Theorem~\ref{thm:asymp_formula_tau_n_theta}.
For $c>0$, let 
\[\alpha(x)=\frac{3}{8}-\frac{c}{(\log x)^{1/8}} .\]
 Let $\delta$ and $\delta'$ be such that
 \[0<\delta<\delta'<\frac{1}{8}.\]
 Then either 
  \[\Delta(x)=\Omega\left(x^{ \alpha(x)+ \frac{\delta}{2}}\right)  \ \text{ or } \ 
  \Delta(x)=\Omega_{\pm}\left(x^{\frac{3}{8}-\delta'}\right).\]
\end{thm}
\noindent
To prove the above theorem, we estimate the growth of the Dirichlet series $D(\sigma+it)$ by assuming that it does not have poles 
in a certain region.
\begin{lem}\label{lem:polynomial_growth_critical_strip_2}
Let $\delta$ and $\sigma$ be such that
\begin{align*}
 0<\delta < \frac{1}{8},
 \mbox{ and } \quad \frac{3}{8}-\delta \leq \sigma < \frac{1}{2}.
\end{align*}
If $D(\sigma+it)$ does not have a pole in the above mentioned range of $\sigma$, then for 
\[\frac{3}{8} -\delta +\frac{\delta}{2(1 + \log\log (3 + |t|))}<\sigma < \frac{1}{2},\]
we have 
\[D(\sigma + it)\ll_{\delta, \theta} |t|^{2-2\sigma+\epsilon}\]
for any positive constant $\epsilon$.
\end{lem}
\begin{proof}
Let $s=\sigma+ it$ with $3/8-\delta\leq\sigma<1/2$.
Recall that 
\[D(s)=\frac{\zeta^2(s)\zeta(s+i\theta)\zeta(s-i\theta)}{\zeta(2s)}.\]
Using functional equation, we write
\begin{equation}\label{eq:functional_eq_tau_n_theta}
D(s)=\mathcal{X}(s)\frac{\zeta^2(1-s)\zeta(1-s-i\theta)\zeta(1-s+i\theta)}{\zeta(2s)},
\end{equation}
where $\mathcal{X}(s)$ is of order (can be obtained from Stirling's formula for $\Gamma$)
\begin{equation}\label{eq:upperbound_chi}
\mathcal{X}(\sigma+it)\asymp t^{2-4\sigma}. 
\end{equation}
Using Stirling's formula and Phragm\'en-Lindel\"of principle, we get
\begin{equation*}
 \zeta(1-s)|\ll |t|^{\sigma/2}\log t.
\end{equation*}
So we get
\begin{equation}\label{eq:upperbound_numerator_tauntheta}
|\zeta^2(1-s)\zeta(1-s-i\theta)\zeta(1-s+i\theta)|\ll t^{2\sigma}(\log t)^4.
\end{equation}


Now we shall compute an upper bound for $|\zeta(2s)|^{-1}$. This
can be obtained in a similar way as in Lemma~\ref{estimate-on-J(T)}.
We choose $t\geq 100$. Similar computation can be done when $t$ is negative.

Consider two concentric circles $\mathcal C_{1,1}$ and $\mathcal C_{1, 2}$, centered at $2+ it$ with radii 
\[\frac{5}{4} + 2\delta \quad \text{ and }\quad\frac{5}{4}+2\delta - \frac{\delta}{1+ \log\log(|t| + 3)}.\]
The circle $\mathcal C_{1, 1}$ passes through $3/4-2\delta +i2t$ and $\mathcal C_{1, 2}$ passes through 
$3/4-2\delta + \delta(1 + \log\log(|t| + 3))^{-1} + i2t$. By our assumption, $\zeta(z)$ does not have any zero for
$|z-2-it|\leq 5/4 + 2\delta$. This implies $\log\zeta(z)$ is a holomorphic function in this region. It is easy to see that
on the larger circle $\mathcal C_{1, 1}$, we have $\log|\zeta(z)|<\sigma'\log t$ for some positive constant $\sigma'$. 
We apply Borel-Carath\'eodory theorem to get an upper bound for $\log\zeta(z)$ on $\mathcal C_{1, 2}$ : 
\begin{align*}
|\log\zeta(z)|&\leq 3\delta^{-1}(1 + \log\log(t + 3))\left(\sigma'\log t  + |\log\zeta(2+it)|\right)\\
& \leq 10\delta^{-1} \sigma' (\log\log t) \log t \quad \text{ for } z\in \mathcal C_{1, 2}.
\end{align*}
We may also note that if $\Re(z-3/4-2\delta)>\delta(\log\log t )^{-1}$ 
and $\Im(z)\leq t/2 $, then similar arguments give
\[|\log\zeta(z)|<\delta^{-1} \sigma'(\log\log t) \log t,\]
for some positive constant $\sigma'$ that has changed appropriately.

Now we consider three concentric circles $\mathcal C_{2, 1}, \mathcal C_{2, 2}, \mathcal C_{2, 3}$, centered at $\sigma'' + i2t$ and with radii 
$r_1=\sigma''-1-\eta, r_2=\sigma''-2\sigma$ and $r_3=\sigma''-\delta_0$ respectively. Here 
\[\delta_0=\frac{3}{4}-2\delta + \frac{\delta}{1+\log\log(t+3)}.\]
We shall choose $\sigma''=\eta^{-1}=\log\log t$.
Let $M_1, M_2, M_3$ denote the supremums of $|\log\zeta(z)|$ on $\mathcal C_{2, 1}, \mathcal C_{2, 2}, \mathcal C_{2, 3}$ respectively. 
We have already calculated that
\[M_3\leq \delta^{-1}\sigma'(\log\log t)\log t.\]
It is easy to show that
\[M_1\leq \sigma'\log\log t,\]
where $\sigma'$ is again appropriately adjusted. Applying the three circle theorem 
we conclude
\[M_2\leq \sigma'(\log\log t)\delta^{-a}\log^{a} t,\]
where
\begin{align*}
 a=\frac{\log(r_2/r_1)}{\log(r_3/r_1)}&=\frac{1-2\sigma + \eta}{1-\delta_0 + \eta} + O\left(\frac{1}{\sigma''}\right)\\
 &=\frac{4(1-2\sigma)}{1+8\delta} + O_\delta\left(\frac{1}{\log\log t}\right).
\end{align*}
This gives
\begin{equation}\label{eq:upperbound_denominator_tauntheta}
 |\zeta(2s)|^{-1}\ll \exp\left(c(\log\log t) (\log t)^{\frac{4(1-2\sigma)}{1+8\delta}}\right),
\end{equation}
for a suitable constant $c>0$ depending on $\delta$.
The bound in the lemma follows from (\ref{eq:functional_eq_tau_n_theta}), (\ref{eq:upperbound_chi}), (\ref{eq:upperbound_numerator_tauntheta}) and 
(\ref{eq:upperbound_denominator_tauntheta}).
\end{proof}

Now we complete the proof of Theorem \ref{thm:tau_theta_omega_pm}.
\begin{proof}[Proof of Theorem \ref{thm:tau_theta_omega_pm}]
Let $M$ be any large positive constant, and define
 \[\A:=\A(Mx^{\alpha(x)}).\]
 Then from Corollary~\ref{coro:balu_ramachandra1}, we have 
 \[\int_{[T, 2T]\cap \A}\frac{\Delta^2(x)}{x^{2\alpha(T)} +1} \d x \gg \exp\left(c(\log T)^{7/8}\right).\]
Assuming 
\begin{equation}\label{eq:upper_bound_measure_asump}
\mu([T, 2T]\cap \A)\leq T^{1-\delta} \quad \text{for}  \ T>T_0,
\end{equation}
Proposition~\ref{prop:refine_omega_from_measure} gives 
\[\Delta(x)=\Omega(x^{\alpha(x) +\delta/2})\]
as $h_0(T)=T^{1-\delta}$,
which is the first part of the theorem.
But if (\ref{eq:upper_bound_measure_asump}) does not hold, 
then we have 
\[\mu([T, 2T]\cap \A)> T^{1-\delta} \]
for $T$ in an $\xset$ .
We choose 
\[h_1(T)=T^{\frac{3}{8}-\frac{2c}{(\log T)^{1/8}}-\delta}, \
\alpha_1(T)=\frac{3}{8}-\frac{3c}{(\log T)^{1/8}}-\delta, \ \alpha_2(T)=\alpha(T).\]
Let $\delta''$ be such that $\delta<\delta''<\delta'$. If $D(\sigma + it)$
does not have pole for $\sigma>3/8-\delta''$ then by Lemma~\ref{lem:polynomial_growth_critical_strip_2},
$D(\alpha_1(T) + it)$ has polynomial growth. So Assumptions~\ref{as:measure_omega_plus_minus_weak} is
valid. 
Since 
\[T^{1-\delta}>5h_1(5T/2)T^{1-\alpha_2(T)},\]
by case (ii) of Theorem~\ref{thm:omega_pm_measure} we have
\[\Delta(T)=\Omega_{\pm}\left(T^{\frac{3}{8}-\frac{3c}{(\log T)^{1/8}}-\delta''}\right).\]
The second part of the above theorem follows from the choice $\delta'>\delta''$.
\end{proof}